\documentclass[twoside,12pt,leqno]{amsart}
\usepackage{amssymb,latexsym,verbatim,enumerate} 
\usepackage{mathtools}
\usepackage{tikz}
\usepackage{hyperref}
\hypersetup{citecolor=red, linkcolor=blue, colorlinks=true}
\usepackage{amsmath,mathrsfs}
\usepackage{stmaryrd}
\usepackage{booktabs}
\usepackage{datetime2}

\DeclareMathAlphabet{\mathpzc}{OT1}{pzc}{m}{it}

\newcommand{\im}{\textup{im}}

\renewcommand{\le}{\leqslant}
\renewcommand{\ge}{\geqslant}
\renewcommand{\leq}{\leqslant}
\renewcommand{\geq}{\geqslant}
\newcommand{\lhdeq}{\trianglelefteqslant}    
\DeclareMathOperator{\lcm}{lcm}

\newcommand{\Binom}[2]{{\genfrac{[}{]}{0pt}{}{#1}{#2}}}

\newcommand\GX{\textup{G{\bf X}}}
\newcommand\OmegaX{\Omega\textup{{\bf X}}}
\newcommand\OmegaY{\Omega\textup{{\bf Y}}}
\newcommand{\coloneq}{\vcentcolon=}      
\newcommand{\Asch}{\mathpzc{A}\mathpzc{C}}
\newcommand{\Aut}{\mathrm{Aut}}

\newcommand{\cC}{\mathscr{C}}
\newcommand{\cD}{{\mathcal D}}
\newcommand{\cM}{{\mathcal M}}

\newcommand{\cU}{{\mathcal U}}

\newcommand{\F}{{\mathbb F}}

\newcommand{\bfa}{{\mathbf a}}
\newcommand{\bfb}{{\mathbf b}}
\newcommand{\bfL}{{\mathbf L}}
\newcommand{\bfO}{{\mathbf O}}
\newcommand{\bfU}{{\mathbf U}}
\newcommand{\bfSp}{{\mathbf{Sp}}}
\newcommand{\bfX}{{\mathbf X}}
\newcommand{\bfY}{{\mathbf Y}}
\newcommand{\SX}{{\mathrm {S{\bf X}}}}

\newcommand{\GY}{{\mathrm {G{\bf Y}}}}
\newcommand{\GL}{{\mathrm {GL}}}
\newcommand{\SL}{{\mathrm {SL}}}
\newcommand{\PSL}{{\mathrm {PSL}}}
\newcommand{\PSU}{{\mathrm {PSU}}}
\newcommand{\GSp}{{\mathrm {GSp}}}
\newcommand{\Sp}{{\mathrm {Sp}}}
\newcommand{\GU}{{\mathrm {GU}}}
\newcommand{\SU}{{\mathrm {SU}}}
\newcommand{\SO}{{\mathrm {SO}}}
\newcommand{\GO}{{\mathrm {GO}}}

\newcommand{\rhogen}{\rho_{\mathrm{gen}}}
\newcommand{\rhonongen}{\rho_{\mathrm{nongen}}}

\newcommand{\eps}{\varepsilon}

\newtheorem{theorem} {Theorem} [section]
\newtheorem{proposition} [theorem] {Proposition}
\newtheorem{lemma} [theorem] {Lemma}
 
\newtheorem{hypothesis}[theorem]{Hypothesis}
\theoremstyle{definition}
\newtheorem{definition}[theorem]{Definition}

\newtheorem{remark}[theorem]{Remark}

\title[Probability that certain elements generate a classical group]{The probability that two elements with large $1$-eigenspaces
  generate a classical group}
\author{S.P.~Glasby}
\address{Center for the Mathematics of Symmetry and Computation, University of
Western Australia, Perth 6009, Australia; 
\qquad\textup{Email: {\tt Stephen.Glasby@uwa.edu.au}}}
\author{Alice C.~Niemeyer}
\address{Chair for Algebra and Representation Theory, RWTH Aachen University, Pontdriesch 10-16, 52062 Aachen, Germany;\qquad\textup{Email: {\tt Alice.Niemeyer@art.rwth-aachen.de}} }
\author{Cheryl E.~Praeger}
\address{Center for the Mathematics of Symmetry and Computation, University of
Western Australia, Perth 6009, Australia;\qquad\textup{Email: {\tt Cheryl.Praeger@uwa.edu.au}}}

\begin{document}
\date{\today}

\begin{abstract}

With high probability, among $O(\log n)$ independent randomly selected elements from  a finite $n$-dimensional classical group, some pair of elements 
power to a $2$-element generating set for  a naturally embedded classical subgroup of dimension $O(\log n)$.
The $2$-element generating set produced consists of certain elements with large $1$-eigenspaces, called stingray elements.
Underpinning this result is a new theorem  on the generation of a finite classical group by a pair of stingray elements. In particular we show that, for classical groups not containing $\SL_n(q)$, the probability of generation is at least $0.975$.
The explicit probability bounds we obtain will be applied to justify complexity analyses
for new constructive recognition algorithms for finite classical groups.


\end{abstract}
\maketitle 

\section{Introduction}

It has been known since 1995 that two random elements of a finite nonabelian simple group $G$ generate the group with probability approaching $1$ as $|G|\to\infty$,  \cite{Dixon, KLub, LieSh95}. That this might be true for $G=A_n$  goes back to a  conjecture in 1882 of Netto \cite[p.76.]{Netto}.  More delicate arguments show that one may control the orders of the generating elements: for example, with high probability a simple group $G$ will be generated by a random involution together with a second randomly selected element (as conjectured in \cite{KLub} and proved in \cite{LieSh96a, LieSh96b}). The main motivation for much of this work was the long-standing problem of determining which finite simple groups $G$ were quotients of the classical modular group $\PSL_2(\mathbb{Z})$, equivalently deciding which simple groups $G$ were $(2, 3)$-generated, see \cite{LieSh96a}. Further work \cite{LieSh02} showed that,  for primes $r, s$ not both $2$, the probability that two randomly chosen elements of a simple classical group $G$, of orders $r$ and $s$, generate $G$ tends to $1$ as $|G|\to\infty$, provided the rank of $G$ is sufficiently large. 

For effective use of these results computationally, explicit probability bounds are required not just asymptotic ones, and these bounds should be available for classical groups of all ranks. Moreover, efficient methods both to construct and to recognise suitable generating elements are needed. An example where two prime order elements were used to recognise classical groups was given in \cite{NieP}; these were so-called `ppd elements' (Definition~\ref{def:stingray}(b)) and they could be efficiently constructed and recognised, and explicit bounds were available for their proportions.

 With a variety of applications in mind,  Burness, Liebeck and Shalev \cite{BLS} showed that all maximal subgroups of a finite nonabelian simple group can also be  generated efficiently. Moreover, several recognition algorithms for (quasi)simple groups require construction of \emph{naturally embedded} smaller (quasi)simple groups; for example a subgroup $A_k$ of a larger alternating group $A_n$ is naturally embedded if it has one orbit of size $k$ and fixes the remaining $n-k$ points (see Subsection~\ref{s:natAn}), and a $k$-dimensional classical group is naturally embedded in a larger $n$-dimensional classical group if it acts naturally on a $k$-dimensional subspace, and fixes pointwise an $(n-k)$-dimensional complementary space. Most constructive recognition algorithms for finite classical groups require the construction of smaller naturally embedded classical subgroups: for example, the 1996 algorithm of Kantor and Seress \cite{KS} to recognise black-box classical groups  constructs a naturally embedded $\SL_3(q)$ (or $\Sp_4(q)$ or $\SU_4(q)$) in an $n$-dimensional classical group, usually via constructing transvections or similar elements from certain $(n-2)$-ppd elements (Definition~\ref{def:stingray}(b)); while a 2013 classical recognition algorithm by Dietrich, Leedham-Green,  L\"ubeck and O'Brien \cite[Section 3]{DLLOB13}  recursively constructs naturally embedded $k$-dimensional classical subgroups in an $n$-dimensional classical group where $k\in [n/3,2n/3]$ using stingray elements (Definition~\ref{def:stingray}(c)).  In an earlier algorithm  for classical groups in odd characteristic, such naturally embedded classical subgroups were obtained by constructing involutions and their centralisers, \cite[Section 4]{LGOB}.

 In all the procedures mentioned above the first objective was to obtain a naturally embedded $k$-dimensional classical subgroup, with $k$ very small, of an $n$-dimensional classical group. To achieve this recursively, by essentially halving the dimension each time, requires $\log n$ steps. It was a request from our colleague \'Akos Seress to the second and third authors that essentially led to the work presented in this paper. \'Akos was convinced by computational evidence that one could achieve, in a single step, a natural embedding of a $k$-dimensional classical subgroup in an $n$-dimensional classical group with $k \approx \log n$. The construction would, as in \cite{DLLOB13}, involve `$e$-stingray elements' (Definition~\ref{def:stingray}(c)), but this time with $e =O(\log n)$ rather than $e=O(n)$. That such elements can be constructed efficiently by examining only $O(\log n)$ random elements of the $n$-dimensional classical group was established in \cite[Theorem 3.3(b)]{NP14}; and in \cite[Section 3]{GNP2} (see also \cite[Theorem 1.1]{GIM}) it was proved that  a random $e_1$-stingray element and a random $e_2$-stingray element (with $e_1, e_2$ roughly $\log n$) should with high probability act on a non-degenerate  $(e_1+e_2)$-subspace and fix a complementary subspace pointwise. What remained was to prove that such a pair of stingray elements generates a  naturally embedded $(e_1+e_2)$-dimensional classical subgroup, with high probability. This is the main result of our paper (Theorem~\ref{thm:main})  and was experimentally predicted by Frank L\"ubeck. Moreover, exploiting this result together with the results from \cite{GNP2, NP14} shows that a naturally embedded $d$-dimensional classical subgroup, for some $d\in [2\log n, 8\log n]$ can be constructed by examining $O(\log n)$ random elements from the $n$-dimensional classical group (Theorem~\ref{thm:appn}; this is  the result mentioned at the beginning of the abstract).

 In Theorems~\ref{thm:appn} and~\ref{thm:main}, by a \emph{classical group of type $\bfX$}, or more precisely, \emph{of type $(\bfX,n,|\F|)$},  we mean a group $G$ satisfying $\OmegaX_n(q)\leq G\leq \GX_n(q)$ in its action on the underlying vector space $V_n=\mathbb{F}^n$, with $\bfX$, $\OmegaX_n(q)$, $\GX_n(q)$, $\F$ as in one of the lines of Table~$\ref{tab:G}$. Note in particular that, for unitary groups, we write $\SU_n(q)$ and $\GU_n(q)$ and that these are subgroups of $\GL_n(q^2)$. Theorem~\ref{thm:appn} is  proved in Section~\ref{ss:thm1.2} as a corollary of a more technical Proposition~\ref{prop:appn}. Note that, in Theorem~\ref{thm:appn} and Proposition~\ref{prop:appn}, the classical group may be an odd dimensional orthogonal group, while in Theorem~\ref{thm:main} orthogonal groups must have even dimension as odd dimensional orthogonal groups contain no `stingray duos'.

  \begin{table}
  \caption{Classical groups of type $(\bfX, n, q)$}
\begin{tabular}{llll|c|l}
  \toprule
  {\bf X} & $\OmegaX_n(q)$ & $\GX_n(q)$ & $|\mathbb{F}|$&$c(\bfX)$& Conditions on {\bf Y} for Theorems~\ref{thm:appn} and~\ref{thm:main}\\
  \midrule
  {\bf L} & $\SL_n(q)$ & $\GL_n(q)$ & $q$&4   & {\bf Y} $=$ {\bf X}   \\
  {\bf U} & $\SU_n(q)$ & $\GU_n(q)$ & $q^2$&4 & {\bf Y} $=$ {\bf X}  \\
  {\bf Sp} & $\Sp_n(q)$ & $\Sp_n(q)$ & $q$ &8 & {\bf Y} $=$ {\bf X}  if  $q$ is odd\\
      &  &  &  & & {\bf Y} $=$ {\bf O}$^\eps$ where $\eps=\pm$ if $q$ is even \\
  {\bf O}$^\eps$  & $\Omega^\eps_n(q)$  &  $\GO^\eps_n(q)$ & $q$&8  & $\eps\in\{\circ, +, -\}$, and {\bf Y} $=$ {\bf O}$^{\eps'}$ where $\eps'=\pm$;  \\
  &&&&& but in Theorem~\ref{thm:main}, $\eps',\eps\in\{+, -\}$.\\
\bottomrule
\end{tabular}
\label{tab:G}
\end{table}

\begin{theorem}\label{thm:appn}
Suppose that $G$ is a classical group of type $(\bfX, n, q)$  with $n > 8$  as in Table~$\ref{tab:G}$. Then for each $\eta > 0$, there is a positive  constant $k(\eta)$ such that, with probability at least $1-\eta$, among $k(\eta) \log n$ independent uniformly distributed random elements from   $G$, some pair  power to a $2$-element generating set for  a naturally embedded $d$-dimensional classical subgroup of type $\bfY$ with $2\log n< d \leq c(\bfX)\log n$ and $c(\bfX), \bfY$ as in Table~$\ref{tab:G}$.
\end{theorem}

 We explain some terminology which we use in Theorem~\ref{thm:main}, and comment on the type $\bfY$ in Table~\ref{tab:G}.
An element $g$ in a classical group $G$ of type $\bfX$, with natural module $V_n$, is called an \emph{$e$-ppd stingray element} if $g$ leaves invariant an $e$-dimensional subspace $U_g$ of $V_n$, fixes pointwise a complementary subspace, and its order $|g|$ is divisible by a primitive prime divisor (ppd) of $|\F|^e-1$ (Definition~\ref{def:stingray}). Moreover, a pair $(g_1,g_2)$ is an \emph{$(e_1,e_2)$-ppd stingray duo} if, for each $i$, $g_i$ is an $e_i$-ppd stingray element relative to an $e_i$-dimensional subspace $U_{g_i}$ of $V_n$, such that $U_{g_1}\cap U_{g_2}=0$ so $V_d:=U_{g_1}\oplus U_{g_2}$ has dimension $d:=e_1+e_2$, and $V_d$ is non-degenerate if $\bfX\ne \bfL$. In proving Theorem~\ref{thm:appn}, the $2$-element generating set for the naturally embedded classical subgroup of type $\bfY$ will form an  $(e_1,e_2)$-ppd stingray duo with each $e_i\in [\log n, 2\log n]$ and with natural module $V_d$, and we will apply Theorem~\ref{thm:main} below to obtain the probability bound for generation, as a function of $d$ and~$q$ (see Section~\ref{s:proofs}). As we observe from Table~\ref{tab:G}, the type $\bfY$ is usually the same as $\bfX$ but may be different, notably  when ${\bf X} = {\bf Sp}$ and $q$ is even. The reason for this difference is explained in  Lemma~\ref{lem:spgen}, namely, whenever the stingray duo generates  a group acting irreducibly on $V_d$, the group generated is  contained in an orthogonal subgroup of $\Sp_d(q)$. 
For the last line of Table~\ref{tab:G} we note that, when $\bfX= \bfO^\eps$, then each $g_i$ in the stingray duo acts irreducibly on the space $U_{g_i}$ and this forces $U_{g_i}$ to have minus type (Lemma~\ref{lem:uperp}), and the non-degenerate direct sum $V_d=U_{g_1}\oplus U_{g_2}$ can have type $\eps'\in\{+, -\}$ (not necessarily equal to  $\eps$). We now introduce the concept of generating duos and their proportion; these are fundamental for Theorem~\ref{thm:main}.

\begin{definition}\label{e:gen}
    {\rm 
Let $G$ be a finite classical group of type $(\bfX, d, |\F|)$ as in one of the lines of Table~\ref{tab:G}, and for $i=1,2$ let  $g_i^G$ be a $G$-conjugacy class of $e_i$-ppd stingray elements, such that $e_1+e_2=d$ (Definition~\ref{def:stingray}).   
 We say that an $(e_1,e_2)$-ppd stingray duo
$(g,g')\in g_1^G\times g_2^G$ is  a \emph{generating stingray duo} if $\langle g, g'\rangle$ contains $\OmegaY_d(q)$ with the type $\bfY$ as in the line for $\bfX$  of Table~\ref{tab:G}. 
The \emph{proportion of generating   stingray duos} in $g_1^G\times g_2^G$ is defined as
\begin{equation}\label{e:rhogen}
   \rhogen(g_1,g_2,G) \coloneq   
   \frac{ \mbox{Number of generating stingray duos in}\ g_1^G\times g_2^G  }{ \mbox{Number of stingray duos in} \ g_1^G\times g_2^G  }.
\end{equation} 
    }
\end{definition}

\begin{theorem}\label{thm:main}
Suppose that $G$ is a classical group of type $(\bfX, d, |\F|)$ on $V$ as in Table~$\ref{tab:G}$, and that $e_1, e_2$ are integers  such that $2\leq e_2\leq e_1$ and $d=e_1+e_2>8$.  If, for $i=1,2$, $g_i$ is an $e_i$-ppd stingray element in $G$, then the proportion $\rhogen(g_1,g_2,G)$ in \eqref{e:rhogen} satisfies
\[
\rhogen(g_1,g_2,G) \geq 1 - \lambda_\bfX\cdot (q^{-1}+q^{-2})  -  \kappa_{\bfX}(q)\cdot q^{-d+3},
\]
where $ \lambda_\bfX=1$ if $\bfX=\bfL$ and otherwise $ \lambda_\bfX=0$ and $\kappa_{\bfX}(q)$ is given in Table~$\ref{t:kappa}$.
Moreover, $\rhogen(g_1,g_2,G)>0.975$ if $\bfX\ne\bfL$, while  if $\bfX=\bfL$ then $\rhogen(g_1,g_2,G)$ is greater than $0.555$ if  $q\geqslant3$, or $0.248$ if $q=2$. 
\end{theorem}

{
\setlength{\tabcolsep}{10pt} 
  \begin{table}
  \caption{Values of $\kappa_{\bfX}(q)$ for Theorem~\ref{thm:main}}
\begin{tabular}{lll|r|l}
  \toprule
  {\bf X} & $\kappa_{\bfX}(2)$ & $\kappa_{\bfX}(q)$  &  \multicolumn{1}{c|}{Exceptional} & \multicolumn{1}{c}{Conditions} \\
         &       &    $q\geq3$       &  \multicolumn{1}{c|}{bound} & \multicolumn{1}{c}{for exception}\\
  \midrule
$\bfL$ & $0.11$&$0.06$&& --\\
$\bfU$ & $5.3\cdot 10^{-6}$&$1.6 \cdot 10^{-8}$&& --\\
$\bfSp$ & $1.15$&$1.52$& $8.42$ & $e_2=2$, $q\geq3$ \\
$\bfO^+$ & $2.08$&$1.54$&$14.61$ & $e_2=2$, $q=2$\\
         &      &           &$3.53$ & $e_2=2$, $q\geq3$\\
         &      &           &$10.43$ & $e_2=2$, $d=32$ and $q\in\{11, 13, 17\}$\\
$\bfO^-$ & $1.85$&$3.02$&  &  --\\
\bottomrule
\end{tabular}
\label{t:kappa}
\end{table}
}

The exceptional term $ \lambda_\bfX\cdot (q^{-1}+q^{-2})$ appearing in the lower bound for type $\bfX=\bfL$ is an upper bound on the probability that the stingray duo generates a reducible subgroup (Proposition~\ref{prop:redL}). For type $\bfL$, the conditional probability that a stingray duo $(g,g')$ is generating, given that $\langle g,g'\rangle$ is irreducible, is at least  $1 -  \kappa_{\bfL}(q)\cdot q^{-d+3}> 0.998$.

In Section~\ref{s:strategy} we discuss our strategy for proving Theorem~\ref{thm:main}, based on the Aschbacher subdivision~\cite{Asch} of maximal subgroups of the finite classical groups into nine families. We then consider these families in separate sections, and finally draw together the results in Section~\ref{s:proofs} to give a proof of Theorem~\ref{thm:main}.

\subsection{A brief commentary on Theorem~\ref{thm:main}}

(a)  We treat the generic case where $|\mathbb{F}|^{e_i}-1$ has a primitive prime divisor for each $i$, so the dimension $e_1+e_2$ is assumed to be not too small (we assume that $e_1+e_2>8$). This assumption facilitates our analysis as it allows us to apply powerful results such as those in \cite{GNP4, GPPS} which depend on the finite simple group classification.

(b) 
Part of the motivation for studying naturally embedded subgroups of finite classical groups comes from the need for explicit probability bounds to justify complexity analyses for new classical recognition algorithms. The bounds we obtain have indeed been adopted in the design of these new algorithms. Moreover the algorithms have been implemented, and early tests indicate that they work exceptionally well for classical groups in all characteristics, even in higher dimensions and over large fields, see Rademacher's thesis   \cite{Rad24}  and \cite{HNPR}.

(c) 
A special case of Theorem~\ref{thm:main} was proved by  Seress, Yal\c{c}inkaya and the third author in \cite[Theorem 2]{PSY}, namely where $e_1=e_2$ and the random stingray elements are conjugate ppd-elements. This special case was fundamental in the analysis of the algorithm in \cite{DLLOB13}. Experiments conducted by  L\"ubeck in connection with that algorithm suggested that the two stingray generators could be chosen independently and still generate a naturally embedded classical subgroup with high probability. This possibility was conjectured formally in \cite[Conjecture 1]{PSY}. Thus Theorem~\ref{thm:main} confirms \cite[Conjecture 1]{PSY}.

 Finally we note two additional points  of difference between~\cite[Theorem 2]{PSY} and  Theorem~\ref{thm:main}. Firstly, ~\cite[Theorem 2]{PSY} estimates the fraction
 \[  
   \frac{ \mbox{Number of generating stingray duos in}\ g_1^G\times g_2^G  }{ |g_1^G\times g_2^G|  }
\]
 while Theorem~\ref{thm:main} gives a lower bound for $\rhogen(g_1,g_2,G)$ as in \eqref{e:rhogen} (note the different denominators for these proportions). 
 Secondly, Theorem~\ref{thm:main} deals with all field sizes whereas~\cite[Theorem 2]{PSY} requires $q\geq3$ for the orthogonal case. In this regard we note that  \cite[Theorem 2]{PSY}, which was used to analyse the complexity of the algorithm in \cite{DLLOB13},  was stated in \cite{DLLOB13} only for $q>4$ with $q$ even,  since the result in the then-available preprint of \cite{PSY} was proved only for $q>4$, see \cite[Theorem 5.1]{DLLOB13}. However, by the time the paper \cite{PSY} was published, its main result \cite[Theorem 2]{PSY} had been proved for all $q\geq3$.

\subsection{Naturally embedded subgroups for other finite simple groups}\label{s:natAn}

The major results of the paper address the question of constructing a naturally embedded classical subgroup $H$ of a finite $n$-dimensional classical group with $H$ of dimension $O(\log n)$. 
We hope that there might be analogous generation results for subgroups of finite exceptional groups of Lie type which play a similar role to naturally embedded subgroups of finite classical groups, and which could be exploited computationally.

Here we discuss briefly the analogous question of constructing  naturally embedded subgroups $A_k$ of an alternating group $G=A_n$ with $k$ `roughly' $\log n$. Let $\Omega=\{1,\dots,n\}$ denote the set on which $G$ acts naturally.

\medskip
\noindent
\emph{Step 1:}\quad By~\cite[Theorem 1]{GPU}, 
 for $n$ sufficiently large, the proportion of  elements of $G=A_n$ for which some power is a $p$-cycle for a prime $p$ between $\log n$ and $(\log n)^{\log\log n}$, is at least $1-7/\log\log n$. By a $p$-cycle, we mean an element with one cycle of length $p$ and $n-p$ fixed points in $\Omega$. Thus after a constant number of independent random selections we can construct, with high probability, two elements $g$ and $h$ such that $g$ is a $p$-cycle and $h$ is an $r$-cycle, for primes $p, r$ such that $\log n\leq p\leq r\leq (\log n)^{\log\log n}$.  Suppose we have found such elements $g, h$, and let $\Delta, \Delta'$ denote the \emph{supports} of $g, h$ respectively, that is, the sets of points forming the $p$-cycle of $g$ and the $r$-cycle of $h$.

\medskip
\noindent
\emph{Step 2:}\quad If $\Delta, \Delta'$ are not disjoint or, equivalently, if $|\Delta\cup \Delta'|<p+r$, then there is a good chance that $g, h$ generate a naturally embedded alternating group $A_k$, as the next result shows.

\begin{lemma}\label{l:natAn}
 Let $G=A_n$ and suppose that $g, h\in G$ are a $p$-cycle and an $r$-cycle, with supports $\Delta, \Delta'$, respectively, where $p, r$ are primes such that $p\leq r\leq n$. Suppose also that $k:=|\Delta\cup \Delta'|$ satisfies $\max\{p+3,r\}\leq k\leq \min\{p+r-1, n\}$. Then $H:=\langle g,h\rangle \cong A_k$ is naturally embedded in $G$.   
\end{lemma}

\begin{proof}
  Note that $H\leq A_k$ since $g,h\in A_n$, and that $p\geq3$ as $A_n$ has no 2-cycles.
  Since $k\leq p+r-1$, the support intersection $\Delta\cap\Delta'\ne \varnothing$, with $\ell := |\Delta\cap\Delta'|=p+r-k>0$. It follows that $H$ is transitive on its support $\Delta\cup \Delta'$, of cardinality $k$.  Also $k=p+r-\ell< p+r\leq 2r$, and hence $h$ acts primitively on its support $\Delta'$ of size $r>k/2$, and we conclude that $H$ must be primitive on $\Delta\cup \Delta'$. Then, as $k\geq p+3$, the $p$-cycle $g\in H$ fixes at least $3$ points of $\Delta\cup \Delta'$ and hence, by a theorem of Jordan, see for example,~\cite[Theorem~3.3E]{DM}, we conclude that $H=A_k$ is naturally embedded in $G$.
\end{proof}

We note that this result is `sharp' in a number of respects. For example, if $k=p+r$ then $\Delta\cap\Delta'= \varnothing$, and the group $H=\langle g,h\rangle \cong C_p\times C_r$ acts intransitively on $\Delta\cup \Delta'$. Also, if $k=p+2$ with $p=r=2^f-1$ (a Mersenne prime), then it is possible to have $H=\langle g,h\rangle \cong \PSL(2,2^f)$ in its $3$-transitive action on the projective line $\mathrm{PG}(1,2^f)$. 

\medskip
\noindent
\emph{Step 3:}\quad If we are `lucky', we can apply Lemma~\ref{l:natAn} immediately and find the naturally embedded alternating subgroup $H=\langle g,h\rangle = A_k$. However, since the primes $p,r$ in Step~1 are at most $(\log n)^{\log\log n}$, it is likely that the supports $\Delta, \Delta'$ are disjoint so that, as we noted above, the group $H=\langle g,h\rangle \cong C_p\times C_r$. In this case we try to conjugate $g$ by a random element $x\in G=A_n$ to obtain a conjugate $g^x$ with support $\Delta^x$ that intersects $\Delta'$ appropriately. In particular, if $|\Delta^x\cap\Delta'|=  1$ then 
$k = |\Delta^x\cup \Delta'| = p+r-1$, which is at least $\max\{p+3,r\}$ (unless $p=r=3\geq \log n$, but this only holds for $n\leq 20$). Thus, whenever $n>20$, the conditions of Lemma~\ref{l:natAn} would hold for such a conjugate yielding  a naturally embedded alternating subgroup $\langle g^x,h\rangle = A_{p+r-1}$. By Lemma~\ref{l:natAn2} below,  we would find a conjugating element $x$ with this property, with high probability, among $O(n/(\log n)^2)$ independently selected random elements.

\begin{lemma}\label{l:natAn2}
Let $G=A_n$ and suppose that $g, h\in G$ are a $p$-cycle and an $r$-cycle, with supports $\Delta, \Delta'$, respectively, where $p, r$ are primes such that $p\leq r$ and $p + r < n.$ Suppose further that $\Delta\cap\Delta'= \varnothing$.
Then 
\[ \frac{|\{ x\in A_n\mid |\Delta^x\cap\Delta'|=  1 \}|}{|A_n|} 
	>\frac{rp(n-rp)}{(n-p+2)^2}.
    \]
    In particular, if $\log n\leq p\leq r$ and $pr\leq n/2$, then this proportion is greater than $(\log n)^2/2n$. 
\end{lemma}

\begin{proof}
Let $S := \{ x\in A_n\mid |\Delta^x \cap \Delta' | = 1 \}$.
Then $x\in S$ if and only if there is exactly one $i\in \Delta$ such that $i^x \in \Delta'$.

We count the number of such permutations $x\in S$.
There are $p$ choices for the point $i$, and as $i^x\in \Delta'$, there are $r$ choices for $i^x$.
Each of the remaining $p-1$ points in $\Delta$ must be mapped to a point in $\{1,\ldots,n\} \backslash \Delta'.$ Thus there are
\[ 
pr (n-r) (n-r-1) \cdots (n-r-(p-2))
\]
choices for the image of $\Delta$ under $x$.
Each of the points not in $\Delta$ can be mapped to any point of $\{1,\ldots,n\} \backslash \Delta^x.$
Thus, since $x$ must lie in $A_n$, the total number is 
\[
|S| = \frac{1}{2} \cdot rp\left(\prod_{i=0}^{p-2} (n-r-i) \right)\cdot (n-p)!.
\]
In particular,
\begin{align*}
\frac{|S|}{|A_n|} &= \frac{rp}{n-p+1}\prod_{i=0}^{p-2} \frac{n-r-i}{n-i}
	= \frac{rp}{n-p+1}\prod_{i=0}^{p-2}\left(1- \frac{r}{n-i}\right)\\
	&\ge\frac{rp}{n-p+1}\left(1-\sum_{i=0}^{p-2} \frac{r}{n-i}\right)
	\ge\frac{rp}{n-p+2}\left(1-r\frac{p-1}{n-p+2}\right)\\
	&>\frac{rp}{(n-p+2)^2}(n-(r+1)(p-1))
	>\frac{rp(n-rp)}{(n-p+2)^2}
\end{align*}
proving the first assertion. Note that $\frac{rp(n-rp)}{(n-p+2)^2} > rp(n-rp) n^{-2}$, and if $rp\leq n/2$ this is at least $rp/2n$. If also $r\geq p\geq \log n$, then  $|S|/|A_n|>(\log n)^2/2n$.
\end{proof}

\section{Groups and stingray elements}\label{s:sting}

Let $q = p^a $ be a  power of a prime $p$ and integer $a\geq1$, and let $d$ be a positive integer. Throughout, we make the following assumptions about the group $G$.

\begin{hypothesis}\label{hypothesis}
We assume that $G$ is a classical group of type ${\bf X} = {\bf L}, {\bf U}, {\bf Sp}$, or ${\bf O}^\eps$ with $\eps=\pm$, and satisfying
$\OmegaX_d(q) \lhdeq G \le  \GX_d(q)$ acting naturally
on $V = \F^d$, where  $d>8$ and
$\OmegaX_d(q), \GX_d(q), |\F|=q^u$, are as in Table~$\ref{tab:G}$. Note, in particular: 
\begin{itemize}
    \item if ${\bf X} \ne {\bf L}$ then $d$ is even;
    \item if ${\bf X} \ne  {\bf U}$ then $u = 1$, while if ${\bf X} =  {\bf U}$ then $u = 2$ and we view $\GU_d(q)$ as a subgroup of $\GL_d(q^2)$, and hence $\SU_d(q)$ as a subgroup of $\SL_d(q^2)$.
\end{itemize}
\end{hypothesis}

Next we formalise the concepts of stingray elements and stingray duos. In the light of Hypothesis~\ref{hypothesis} we 
consider actions of elements of $\GL_d(q^u)$ on $\F^d_{q^u}$. 

\begin{definition}\label{def:stingray}
  Let $d, e,a $ be integers such that $1\leq e\leq d$ and $q=p^a$ where $p$
  is prime. Let $u \in \{1,2\}$, and consider $\GL_d(q^u)$ acting naturally on the  module  $V=\F^d_{q^u}$.
  \begin{enumerate}
  \item[(a)] A  \emph{primitive prime divisor} of $(q^u)^e-1$ is a prime $r$
    dividing $(q^u)^e-1$ such that $r$ does not divide $(q^u)^i-1$ for any $i$ such that $1\leqslant i<e$;
    we sometimes call such a prime an \emph{$e$-ppd prime}. Note that  $q^u$ has order
$e$ modulo $r$ and hence $e$ divides $r-1$; equivalently $r = ke +1$ for some $k\ge 1.$ 
  \item[(b)]  Let   $r$ be a primitive prime divisor of $(q^u)^e-1$. Then an element $g\in\GL_d(q^u)$ of order
    a multiple of $r$ is called an \emph{$e$-ppd element}. 
  \item[(c)] An element $g\in\GL_d(q^u)$ is called an \emph{$e$-stingray element} if $g$ preserves a
    decomposition $V=U_g\oplus F_g$, where $F_g$ is  the
    fixed point subspace of $g$ and $g$ acts
    irreducibly (and nontrivially) on $U_g$ with $\dim(U_g)=e$.  In particular $|g|\mid (q^u)^e-1$.
    An $e$-stingray element $g$ of order a multiple of an $e$-ppd prime $r$ is called
    an \emph{$e$-ppd stingray element} (and in most of this paper $|g|$ is equal to the $e$-ppd prime $r$). 
   \item[(d)]    For positive integers $e_1, e_2$ such that $e_1 + e_2 = d$, a pair $(g_1,g_2)$  is  an  \emph{$(e_1,e_2)$-stingray duo},  if $e_1 \ge e_2$ and,
for $i=1,2$, the element $g_i\in G$ is an $e_i$-stingray element  and, moreover,
  $U_{g_1}+ U_{g_2} = V$, or equivalently, $U_{g_1}\cap U_{g_2} = 0$. Further, $(g_1,g_2)$  
  is  an  \emph{$(e_1,e_2)$-ppd stingray duo} if, for each $i$, $g_i$ is an $e_i$-ppd stingray element.
 \end{enumerate}   
\end{definition}

\begin{remark}\label{rem:ppd}
 (a) For Theorem~\ref{thm:main}, we are concerned with  $(e_1,e_2)$-ppd stingray duos in a $d$-dimensional classical subgroup of $\GL_d(q^u)$, where $d=e_1+e_2$, so we will be assuming that $(q^u)^{e_i}-1$ has a primitive prime divisor, for each $i=1,2$.   Such primes are known to exist for almost all values of $q^u, e_i$, the exceptions being $(q^u,e_i)=(2,6)$, and the case $e_i=2$ with $q^u+1$ a power of $2$, see \cite[Section 1.13]{BHRD}. Thus, for example, if $e_i=2$ then we will require  $q^u\not\in\{3, 7, 31,127, \dots\}$.    

    (b) In particular if $e\geq 2$ then each primitive prime divisor of  $(q^u)^{e}-1$ is an odd prime.

    (c) For an $e$-ppd element $g\in\GL_d(q^u)$ as in Definition~\ref{def:stingray}(b), if $|g|$ is coprime to $p$ then $V=\F_{q^u}^d$ has the form
    $V=\left(\bigoplus_{i=1}^t U_i\right)\oplus F$, where $g$ fixes $F$ elementwise, and fixes each $U_i$ setwise, so $t\geq 1$. Moreover, if $|g|=r$, then $g$ acts irreducibly (and nontrivially) on $U_i$ and $\dim(U_i)=e$ for each $i$.
    
\end{remark}

For Theorem~\ref{thm:main}, the property of interest is whether or not a given $(e_1,e_2)$-ppd stingray duo $(g,g')$ in a $d$-dimensional classical group $G$
is `generating', see Definition~\ref{e:gen}. If $(g,g')$ is not a generating stingray duo then we say that  $(g,g')$ is \emph{non-generating}. Our objective is to obtain a positive lower bound for the proportion $\rhogen(g_1,g_2,G)$ of generating   stingray duos in $g_1^G\times g_2^G$, as defined in \eqref{e:rhogen}. Equivalently, we seek  an upper bound strictly less than~$1$ for the quantity
\begin{equation}\label{e:rhonongen}
   \rhonongen(g_1,g_2,G) \coloneq   
   \frac{\mbox{Number of non-generating stingray duos in}\ g_1^G\times g_2^G   }{ \mbox{Number of stingray duos in} \ g_1^G\times g_2^G}.
\end{equation}

Note that $\rhogen(g_1,g_2,G)+\rhonongen(g_1,g_2,G)=1$. Moreover, in Proposition~\ref{prop:om-class}(a) we show that the conjugacy classes $g_i^G$ are independent of the group $G$ satisfying $\OmegaX_d(q) \lhdeq G \le  \GX_d(q)$, and hence also the proportions of generating and non-generating stingray duos are independent of $G$. We therefore may (and usually do) write
\begin{equation}\label{e:rhoX}
\mbox{$\rhogen(g_1,g_2,\bfX):= \rhogen(g_1,g_2,G)$\quad and \quad $\rhonongen(g_1,g_2,\bfX):= \rhonongen(g_1,g_2,G)$.}    
\end{equation}
Our strategy, which is discussed in detail in Section~\ref{s:strategy}, is to find an upper bound for  $\rhonongen(g_1,g_2,\bfX)$, and we observe that $\rhonongen(g_1,g_2,\bfX) = \rhonongen(g_1',g_2',\bfX)$ for any choice of $g_i'\in g_i^G$. For the remainder of this section we study properties of stingray elements and stingray duos in classical groups.

\subsection{Properties of  Stingray Elements}

We are primarily interested in stingray duos, but first we examine properties of individual stingray elements. Parts (a) and (b) of the first result were proved in \cite[Lemma 3.7]{GNP2}. (Note that, although the hypothesis ${\bf X}\ne {\bf L}$ is implicit in the hypotheses of \cite[Lemma 3.7]{GNP2}, this is not used in the proof.) The proof of part (c) follows easily from the fact that, since $V=U_g\oplus F_g$, each subspace $W$ satisfying $U_g\leq W\leq V$ is equal to $W=U_g\oplus (F_g\cap W)$, and hence is $g$-invariant.

\begin{lemma}\label{lem:unique} 
Let $g\in \GL(V)$ be an $e$-stingray   element, with $U_g, F_g$ as in Definition~$\ref{def:stingray}$.   Then
\begin{enumerate}[{\rm (a)}]
    \item
    $0\ne v^g-v\in U_g$, for all $v\in V\backslash F_g$; and hence $U_g$ is the unique $\langle g\rangle$-invariant submodule of $V$ on which $g$ acts non-trivially and irreducibly;
    \item
    if $Z$ is a $\langle g\rangle$-invariant submodule of $V$ then either $Z\leq F_g$, or $U_g\leq Z$, and in the latter case the restriction $g|_Z$ is an $e$-stingray element of $\GL(Z)$;
    \item
    if $U_g\leq W\leq V$ then $W$ is $\langle g\rangle$-invariant.
\end{enumerate}
\end{lemma}

The next lemma identifies a cyclic torus associated with an $e$-stingray element.  

\begin{lemma}\label{lem:uperp}
  Let $G$ be a  group as in Hypothesis~$\ref{hypothesis}$ of type $\bf X$,
  and let $g\in G$ be an $e$-stingray   element with 
$U_g, F_g$ as in Definition~$\ref{def:stingray}$,  where $1\leq e\leq d-1$.

  \begin{enumerate}[{\rm (a)}]
  \item 
  If ${\bf X}\ne {\bf L}$, then $U_g$ and $F_g$ are
  non-degenerate, $V = U_g\perp F_g$, and  $U_g^\perp = F_g$.
 
    \item 
    Either $e = 1$ and  ${\bf X}\in\{ {\bf L}, {\bf U}, {\bf O}^\eps \}$, or $e \ge 2$ and the parity of $e$, and types of $U_g$ and $F_g$ are as  given in
  Table~$\ref{tab:one}$; in particular, if ${\bf X}={\bf O}^\eps$,  then  $U_g$ has minus type.
    
    \item 
     Suppose that $({\bf X}, \mbox{$e$-parity})\ne ({\bf O}^\eps, \mbox{even})$. Then there is a unique cyclic torus $T$ in $\GX(V)$ containing $g$
  such that  $T$ preserves the decomposition $V=U_g\oplus F_g$, $F_g$ is
  the fixed point space of $T$, and $|T|$ is given in Table~$\ref{tab:one}$. Moreover, 
  \[
  C_{\GX(V)} (g) = C_{\GX(V)} (T) = \left\{\begin{array}{ll}
     T\times  \GX(F_g) & \mbox{if ${\bf X}\ne{\bf O}^\pm$} \\
     T\times \GO^{-\eps}(F_g)  & \mbox{if ${\bf X}={\bf O}^\eps$, with $\eps=\pm$}
  \end{array}\right. 
  \]
Also $g^{\GX(V)}\cap T$ has size $e$, and $N_{\GX(V)}(T)$ permutes this set transitively by conjugation.
  \end{enumerate}
\end{lemma}

\begin{table}
  \caption{Conditions for Lemma~\ref{lem:uperp} when $e\geq2$.}
\begin{tabular}{lclll}
  \toprule
  {\bf X} & $e$-parity & type of $U_g$ & type of $F_g$ & $|T|$\\
  \midrule
  {\bf L} & any & {\bf L} & {\bf L} & $q^e-1$\\
  {\bf U} & odd & {\bf U} & {\bf U} & $q^e+1$\\
  {\bf Sp} & even &  {\bf Sp} & {\bf Sp} & $q^{e/2}+1$ \\
  {\bf O}$^\circ$ & even &  {\bf O}$^-$ & {\bf O}$^\circ$ & $q^{e/2}+1$\\
  {\bf O}$^\eps$, $\eps=\pm$  & even  &  {\bf O}$^-$ & {\bf O}$^{-\eps}$ & $q^{e/2}+1$\\
\bottomrule
\end{tabular}
\label{tab:one}
\end{table}

\begin{proof}
Part (a) is proved in \cite[Lemma 3.8(a)]{GNP2}. If $e=1$, then the possibilities for $\bf X$ follow from part (a). If $e\geq2$, then the assertions in part (b) follow from \cite[Lemma 3.8(b)]{GNP2}. 

(c) If $e=1$ then by part (b), $\bfX= \bfL$ or $\bfU$, and $T:= C_{\GX(V)} (g)|_{U_g} \cong C_{q^e-1}$ or $C_{q^e+1}$ respectively. The order $|T|$ may be read from \cite[Table 1]{Bereczky}. If $e\geq2$, then  the existence and uniqueness of $T$ follows from the proof of the result \cite[Lemma~3.1(ii)]{PSY} by
    replacing $t, W, U_g, n$ by $g, U_g, F_g, e$, as the proof there only used the
    fact that $t\mid_W$  (or $g\mid_{U_g}$ in our case) is irreducible.
    The arguments also extend to the case  $G_0=\SO^\circ_d(q)$ as $U_g$
    is nondegenerate of minus type.  Finally we examine $C_{\GX(V)} (g)$. Set $Y\coloneq\GX(V)$. Since $T$ leaves invariant the decomposition $V=U_g\oplus F_g$, $C_{\GX(V)} (g)\leq Y_{U_g, F_g} = Y_1\times Y_2$, where 
    $Y_1, Y_2$ are the isometry groups of $U_g, F_g$ (with types as in Table~\ref{tab:one}). Moreover $1\times Y_2\leq C_{\GX(V)} (g)$ and so $C_{\GX(V)} (g) = C_{Y_1} (g) \times Y_2$. Now $T$ is a self-centralising (irreducible, cyclic) subgroup of $Y_1$, and as $g\in T$ and $T$ is cyclic, we have $C_{Y_1}(T)=T\leq C_{Y_1} (g)$. However since $g|_{U_g}$ is irreducible, $C_{\GL(U_g)}(g)$ is a cyclic torus $\hat T$ of order $q^{u e}-1$, and therefore $\hat T\cap Y_1$ is a cyclic torus of $Y_1$ containing $T$. It follows that $\hat T\cap Y_1=T=C_{Y_1} (g)$, and hence $C_{\GX(U_g)}(g)=C_{\GX(U_g)}(T)$ is as in part (c). 
    
    Finally, as in the proof of \cite[Lemma~3.1(ii)]{PSY}, the normaliser
$N_{\GX(V)}(T)$ leaves invariant $U_g$ and $F_g$, and $|N_{\GX(V)}(T)| = e\cdot|C_{\GX(V)}(T)|$; and moreover $g^G \cap T$ has size $e$,  and $N_{\GX(V)}(T)$ permutes this set transitively by conjugation.
\end{proof}

We use Lemma~\ref{lem:uperp} to derive the important consequence that the $G$-conjugacy class $g^G$ of an $e$-stingray element $g$ is independent of the choice of $G$ in Hypothesis~\ref{hypothesis}, and moreover  a conjugacy class of $e$-ppd stingray elements is in bijection with a conjugacy class of $e$-stingray elements of $e$-ppd prime order.

  \begin{proposition}\label{prop:om-class}
    Let $G$ be a  group as in Hypothesis~$\ref{hypothesis}$ of type $\bf X$, so $\OmegaX_d(q)\lhdeq G\leq \GX_d(q)$,
  and let $g\in G$ be an $e$-stingray   element as in Definition~$\ref{def:stingray}$ such that $1\le e\le d-1$ and if $\bfX\in\{ \bfSp, \mathbf{O^\varepsilon}\}$ then $e$ is even.
    Then 
    \begin{enumerate}[{\rm (a)}]
        \item 
        $G=C_G(g)\,\OmegaX_d(q)$, and $g^{\OmegaX_d(q)} = g^G = g^{\GX_d(q)}$ is independent of $G$.
        \item 
        If $e\geq2$ and $|g|$ is divisible by a primitive prime divisor  $r$ of $(q^u)^{e}-1$, and if $m$ is as in Table~$\ref{tab:m}$ then, letting $n\in\{m, |g|/r\}$ and $h:=g^n$, the following hold:
\begin{enumerate}[{\rm (i)}]
    \item $h$ is an $e$-ppd stingray element, $h\in\OmegaX(V)$, $U_h=U_{g}$, $F_h=F_{g}$;
    \item $C_{\GX(V)}(h)=C_{\GX(V)}(g)$, and the map $x\to x^{n}$ defines a  bijection $g^G\to h^G$.
\end{enumerate}
    \end{enumerate}
 \end{proposition}

   \begin{table}
  \caption{The value of $m$ for Proposition~\ref{prop:om-class}}
\begin{tabular}{l|cccc}
  \toprule
  {\bf X} &  {\bf L} &  {\bf U} &  {\bf Sp} & {\bf O}$^\pm$\\
  $m$       & $q-1$ & $q+1$ & $1$ &$2$\\
\bottomrule
\end{tabular}
\label{tab:m}
\end{table}

  \begin{proof}
  (a) Suppose first that  $\GX(V)=C_{\GX(V)}(g)\,\OmegaX(V)$. Then each element $x\in G$ is of the form $x=yz$ for some $y\in C_{\GX(V)}(g)$ and $z\in\OmegaX(V)$. Since $G$ contains $\OmegaX(V)$ the element $y=xz^{-1}\in G$ and hence $y\in G\cap C_{\GX(V)}(g)=C_G(g)$. Thus $G=C_G(g)\,\OmegaX_d(q)$. Moreover, we have $g^{\OmegaX_d(q)} \subseteq g^G \subseteq g^{\GX_d(q)} = g^{C_{\GX(V)}(g)\,\OmegaX(V)} = g^{\OmegaX_d(q)}$, and so equality holds.
  Therefore to prove part (a) it suffices to prove that $G=C_G(g)\,\OmegaX_d(q)$ in the case where $G=\GX(V)$. So assume that  $G=\GX(V)$.
  
 For ${\bf X}={\bf L}$, ${\bf U}$ or ${\bf Sp}$, the kernel of the determinant map $\det\colon\GX_d(q)\to\F_{q^u}^\times$ is $\OmegaX(V)=\SX(V)$. 
Lemma~\ref{lem:uperp}(c) shows that $\det(C_{\GX(V)}(g)) =\det(\GX(F_g))= \det(\GX(V))$, and so
    \begin{align*}
    |C_{\GX(V)}(g)\SX(V)|&= \frac{|C_{\GX(V)}(g)|\cdot |\SX(V)|}{|C_{\GX(V)}(g)\cap \SX(V)|}
    = |\det(C_{\GX(V)}(g))|\cdot |\SX(V)|\\
    &= |\det(\GX(V))|\cdot |\SX(V)| = |\GX(V)|.
    \end{align*}
   Thus $\GX(V)=C_{\GX(V)}(g)\,\SX(V)$ as claimed. So assume that ${\bf X}={\bf O}^\eps$, with $\eps=\pm$ and $d$ even (see Hypothesis~\ref{hypothesis}). Here $\OmegaX(V)=\Omega^\eps(V)$ and $e$ is even by assumption. Also by~\cite[Lemma~4.1.1(i)]{KL} and Lemma~\ref{lem:uperp}, 
   \[
   1\times\GO^{-\eps}(F_g)\leqslant C_{\GO^\eps(V)}(g)=T\times\GO^{-\eps}(F_g)\leq \GO^{-}(U_g)\times\GO^{-\eps}(F_g) = \GO^\eps(V)_{U_g},
   \]
    with $|T|=q^{e/2}+1$, and by~\cite[Lemma~4.1.1(ii)]{KL}, $(1\times\GO^{-\eps}(F_g))\cap\Omega^\eps(V)=1\times\Omega^{-\eps}(F_g)$, so 
    \[
      \frac{|\GO^\eps(V)|}{|\Omega^\eps(V)|}\ge\frac{|C_{\GO^\eps(V)}(g)\Omega^\eps(V)|}{|\Omega^\eps(V)|}
      \ge\frac{|(1{\times}\GO^{-\eps}(F_g))\Omega^\eps(V)|}{|\Omega^\eps(V)|}=\frac{|1{\times}\GO^{-\eps}(F_g)|}{|1{\times}\Omega^{-\eps}(F_g)|}
      =\frac{|\GO^{-\eps}(F_g)|}{|\Omega^{-\eps}(F_g)|}.
    \]
    By~\cite[Table~2.1.C]{KL}, $|\GO^{-\eps}(V)/\Omega^{-\eps}(V)|=|\GO^{-\eps}(F_g)/\Omega^{-\eps}(F_g)|$ is~$2$ if $q$ is even, and $4$ 
    if $q$ is odd. Since this holds for each $\eps=\pm$, it follows that equality holds for all the inequalities in the above display, and in particular  $|\GO^\eps(V)|=|C_{\GO(V)}(g)\Omega^\eps(V)|$, so $\GO^\eps(V)=C_{\GO^\eps(V)}(g)\,\Omega^\eps(V)$. Thus part (a) is proved.
      
  (b) The group  $\GX_d(q)/\OmegaX_d(q)$ has exponent $m$ for $m$ as in Table~\ref{tab:m}.
  (To see this for case ${\bf O}$ it suffices to consider $q$ odd. As $\GO^\eps(V)$ is generated by
  reflections~\cite[Corollary 11.42]{Taylor}, the same is true for $\GO^\eps(V)/\Omega^\eps(V)$, so the exponent is $2$.)  Hence $g^m\in\OmegaX_d(q)$.
  Also, since $e\ge2$, the primitive prime
  divisor  $r$ of $(q^u)^{e}-1$ does not divide $q^u-1$ and hence $r$ does not divide $m$ in cases ${\bf L}$ or ${\bf U}$. In addition, by Definition~\ref{def:stingray}(a),  $r\geq e+1\geq 3$, and so $r$ does not $m$ in cases ${\bf Sp}$ or ${\bf O^\varepsilon}$ either. Thus in all cases $r$ is
  coprime to $m$ and hence $m$ divides $|g|/r$. This implies that $g^{|g|/r}\in \OmegaX_d(q)$.
  Thus we have proved that the element $h=g^n\in \OmegaX_d(q)$, and $|h|$ is divisible by $r$.
  
  Since $h\in\langle g\rangle$, $h$ fixes $F_g$ pointwise, and leaves $U_g$ invariant. Also since $|h|$ is divisible by $r$, it follows that $h$ acts irreducibly on $U_g$, so $h$ is an $e$-ppd stingray element, and $U_h=U_{g}$, $F_h=F_{g}$. Moreover, since  $h\in\langle g\rangle$, the  unique cyclic torus containing $h$ given by Lemma~\ref{lem:uperp}(c) is the cyclic torus $T$ containing $g$ and hence $C_{\GX(V)}(g)=C_{\GX(V)}(h)$ by Lemma~\ref{lem:uperp}(c). Thus using part (a), $|g^G|=|g^{\GX(V)}|=|\GX(V):C_{\GX(V)}(g)|=|h^{\GX(V)}|=|h^G|$. The map $x\to x^n$ clearly defines a surjection $g^G\to h^G$, and since $|g^G|=|h^G|$, this map is a bijection.  
     \end{proof}

 In our analysis we frequently need to know the number of conjugates of certain overgroups of stingray elements. The final lemma of this subsection is very helpful for this.
 
\begin{lemma}\label{lem:countM}
Suppose that $G=\GX_d(q)$ of type $\bfX$ with a proper subgroup $M$ such that $N_{G}(M)=M$. If $g\in M$ is a stingray element such that $g^{G}\cap M$ is the single $M$-conjugacy class $g^M$, then
\[
|\{ L\in M^G\mid g\in L  \}| = \frac{|C_{G}(g)|}{|C_M(g)|}
\]
with $C_{G}(g)$ given by Lemma~$\ref{lem:uperp}{\rm (c)}$. 
\end{lemma}

\begin{proof}
Let $Y$ denote the number of conjugates $L\in M^{G}$ containing $g$, and note that $|M^{G}|=|G|/|M|$ (since $N_{G}(M)=M$ by assumption). 
Then, the number of pairs $(L,g')$ such that  $L\in M^G$, $g'\in g^{G}$ and $g'\in L$, is equal to the number $|g^{G}|$ of choices for $g'$ times~$Y$. Also the number of these pairs is equal to the number $|M^{G}|=|G|/|M|$ of choices for the group $L$ times the number $|g^{G}\cap L|=|g^{G}\cap M|=|g^M|$ of choices for $g'$, for a given group $L$. Thus 
\[
Y=\frac{|M^{G}|\cdot |g^M|}{|g^{G}|} = \frac{|G|}{|M|}\cdot \frac{|M|}{|C_M(g)|}\cdot \frac{|C_{G}(g)|}{|G|} = \frac{|C_{G}(g)|}{|C_M(g)|}
\]
and we recall that  $C_{G}(g)$ is given by Lemma~\ref{lem:uperp}(c).
\end{proof}

\begin{table}
  \caption{Conditions for $(e_1,e_2)$-stingray duos in classical groups.}
\begin{tabular}{llll}
  \toprule
  {\bf X} & $e_i$-parity & $e_i$-range & $d=e_1+e_2$ and $U_i=V(g_i-1)$\\
  \midrule
  {\bf L} & any & $1 \le e_2\le e_1 \le d-1$ & $d$ odd or even\\
    {\bf U} & odd & $1 \le e_2\le e_1 \le d-1$ & $d$ even\\
  {\bf Sp} & even & $2 \le e_2\le e_1 \le d-2$ & $d$ even \\
{\bf O}$^\pm$ & even & $2 \le e_2 \le e_1\le d-2$ & $d$ even, each $U_i$ has minus type\\
\bottomrule
\end{tabular}
\label{tab:stingraycond}
\end{table}

\subsection{Stingray duos in classical groups}

Let $(g_1,g_2)$ be an $(e_1,e_2)$-stingray duo, as in Definition~\ref{def:stingray}, for a classical group  of type $\bf X$ contained in $\GL_d(q^u)$, so $d=e_1+e_2$ and $e_1\geq e_2\geq1$. (See our comments later in this paragraph for why we allow $e_2=1$, especially in analysing the Aschbacher class $\Asch_3$.)  We  assume that $\OmegaX_d(q)\leq G\leq \GX_d(q)$ and that $g_1, g_2$ are stingray elements in $\GX_d(q)$, not necessarily lying in $G$.
By Lemma~\ref{lem:uperp} the value $e_i=1$ may occur only for ${\bf X}\in\{ {\bf L}, {\bf U}, {\bf O}^\pm\}$; 
and if ${\bf X}=\bf O^\pm$ then we  assume that both of the $e_i$ are even so that, in particular, $d=e_1+e_2$ is even (and we do not consider $(e_1,e_2)$-stingray duos in $\GO_d^\circ(q)$ with $d$ odd). For convenience we summarise the parity and range of the $e_i$ in Table~\ref{tab:stingraycond}, based on these comments and Table~\ref{tab:one}.
In Theorem~\ref{thm:main} we assume further that $d=e_1+e_2$ with each $e_i\geq2$, and $d>8$. However in parts of our analysis we need to deal with classical groups of smaller dimension and duos with smaller $e_i$, for example, when considering classical subgroups over extension fields. Thus we assume that $e_1\geq e_2$ in the discussion below, and often allow the possibility of $e_2=1$ for types  ${\bf X}\in\{ {\bf L}, {\bf U}\}$.

Our general aim is to show that, with high probability, an $(e_1,e_2)$-stingray duo generates  a group containing $\OmegaX_d(q)$, but, as we discussed in the paragraph before Definition~\ref{e:gen}, the case of symplectic groups with $q$ even is exceptional in this regard. The next lemma is a generalisation of \cite[Lemma 3.3]{PSY}, and justifies our definition of generating stingray duo in Definition~\ref{e:gen}.

\begin{lemma}\label{lem:spgen}
Let $G=\Sp_d(q)=\GSp_d(q)$, where $q$ is even and $d=e_1+e_2$ with each $e_i$ even, and let $(g_1,g_2)$ be an $(e_1,e_2)$-stingray duo in $G$ as in Definition~$\ref{def:stingray}$. Then either $\langle g_1, g_2\rangle$ is reducible, or $\langle g_1, g_2\rangle\leq \SO^\varepsilon_d(q)$ for some $\varepsilon=\pm$. 
\end{lemma}

\begin{proof}
Let $H:=\langle g_1, g_2\rangle$, and for each $i$ let $U_i:=U_{g_i}$ and $F_i:=F_{g_i}$ as in Definition~$\ref{def:stingray}$, so that $V=U_i\oplus F_i$ and $V=U_1\oplus U_2$.
Since $q$ is even, $G=\Sp_d(q) \cong \Omega_{d+1}(q)$. Let $\widehat{V}$ be the associated orthogonal space of dimension $d+1$ for $G$ with nonsingular $1$-dimensional radical $R$ and $V=\widehat{V}/R$. Then for each $i$, $g_i$ acts as an $e_i$-stingray element on $\widehat{V}$, and by Lemmas~\ref{lem:unique} and~\ref{lem:uperp}, $g_i$ acts irreducibly on an  $e_i$-dimensional subspace $\widehat{U}_i$ of $\widehat{V}$ and has an $(d-e_i+1)$-dimensional fixed point subspace $\widehat{F}_i$ containing $R$ (recall that $e_i\geq2$). Thus $U_i=(\widehat{U}_i+R)/R$ for each $i$, and $W\coloneq\widehat{U}_1 + \widehat{U}_2$ has dimension at most $e_1+e_2=d$. Since $V=U_1\oplus U_2$ it follows that  $\widehat{V}=  (\widehat{U}_1 + \widehat{U}_2)\oplus R$, and $W=\widehat{U}_1 \oplus \widehat{U}_2$ is a hyperplane of $\widehat{V}$ which intersects $R$ trivially. 
Moreover, by Lemma~\ref{lem:unique}(c), $W$ is invariant under both $g_1$ and $g_2$, and hence $H$ is contained in the stabiliser in $G$ of $W$. If $W$ is a nondegenerate subspace of $\widehat{V}$, then the stabiliser of $W$ is an orthogonal group and we have $H\leq \GO^\varepsilon_d(q)$ for some $\varepsilon=\pm$, and since  $g_1,g_2\in\Sp_d(q)\leq\SL_d(q)$ the subgroup $H=\langle g_1, g_2\rangle\leq\SO_d^\eps(q)$. On the other hand if $W$ is degenerate then $H$ leaves invariant the nonzero totally isotropic subspace $((W\cap W^\perp)+R)/R$ of $V$.
\end{proof}

We first record some basic properties about ppd stingray duos which allow us to reduce to the case where these elements have prime orders and lie in $\OmegaX_d(q)$, and where the group $G$ is $\GX_d(q)$.

\begin{lemma}\label{lem:redOm}
Let $G$ be a classical group of type \textup{\bf X} satisfying
  $\OmegaX_d(q) \lhdeq G \le  \GX_d(q)$ as in Hypothesis~$\ref{hypothesis}$, and
   let $(g_1,g_2)$ be an $(e_1,e_2)$-ppd stingray duo in $\GX_d(q)\times \GX_d(q)$ with $e_1, e_2$ as in Table~$\ref{tab:stingraycond}$ with $e_1\geq e_2\geq 2$. For each $i$, let $|g_i|$ be a multiple of an $e_i$-ppd $r_i$ of $q^{ue_i}-1$, and let $h_i=g_i^{|g_i|/r_i}$. Then the following hold.
\begin{enumerate}[{\rm (a)}]
      \item 
     The proportions defined in \eqref{e:rhogen} and \eqref{e:rhonongen} satisfy 
     \begin{align*}
     \rho_{gen}(g_1,g_2,\OmegaX_d(q))&=\rho_{gen}(g_1,g_2,G) =\rho_{gen}(g_1,g_2,\GX_d(q))\\
     \rho_{nongen}(g_1,g_2,\OmegaX_d(q))&=\rho_{nongen}(g_1,g_2,G) =\rho_{nongen}(g_1,g_2,\GX_d(q)),
     \end{align*}
     so we write these proportions as $\rho_{gen}(g_1,g_2,\bfX)$ and $\rho_{nongen}(g_1,g_2,\bfX)$, respectively. 

      \item 
      Both $h_i$ lie in $\OmegaX_d(q)$ where $h_i^{\OmegaX_d(q)} = h_i^G=h_i^{\GX_d(q)}$, and $(h_1,h_2)$ is an $(e_1,e_2)$-ppd stingray duo.

      \item 
      Moreover if $(h_1,h_2)$ is a generating stingray duo, then $(g_1,g_2)$ is also a generating stingray duo, so
      \[
       \rhogen(g_1,g_2,\bfX)  \geq  \rhogen(h_1,h_2,\bfX ) \quad\mbox{and}\quad
        \rhonongen(g_1,g_2,\bfX)   \leq  \rhonongen(h_1,h_2,\bfX).
      \]
\end{enumerate}
\end{lemma}

\begin{proof}
    Applying Proposition~\ref{prop:om-class}(a) to both $g_1$ and $g_2$, the numerators and the denominators of the fractions defining $\rho_{gen}$ and $\rho_{nongen}$ in \eqref{e:rhogen} and \eqref{e:rhonongen} are independent of $G$. This implies equality of the proportions asserted in part (a). 

    Part (b) follows immediately from Proposition~\ref{prop:om-class}. Moreover we see from Proposition~\ref{prop:om-class} that, for each $i$, the map $x\to x^{|g_i|/r_i}$ is a bijection $g_i^G\to h_i^G$. It follows that the map  
 $(x,y)\to (x^{|g_i|/r_1}, y^{|g_i|/r_2})$ is a bijection $g_1^G\times g_2^G\to h_1^G\times h_2^G$ that maps stingray duos to stingray duos; and, further, the domain and range are independent of the choice of the group $G$.

 Suppose that $(h_1,h_2)$ is a generating stingray duo, so by Definition~\ref{e:gen} either $\langle h_1,h_2\rangle =\Omega X_d(q)$, or  $\bfX=\bfSp$ and $q$ is even, and $\langle h_1,h_2\rangle$ is an orthogonal subgroup of $\Sp_d(q)$. Then also $\langle g_1,g_2\rangle$ contains $\Omega X_d(q)$, or  $\langle g_1, g_2\rangle$ contains an orthogonal subgroup, respectively, so $(g_1,g_2)$ is a generating stingray duo. Finally, the assertions about the proportions follow since the numbers of stingray duos in $g_1^G\times g_2^G$ and in $h_1^G\times h_2^G$ are equal. 
\end{proof}

\section{Strategy for proving Theorem~\ref{thm:main}}\label{s:strategy}

Let $G\leq \GL_d(q^u)$ be a classical group of type $\bf X$ acting on the vector space $V=(\F_{q^u})^d$, as in Hypothesis~\ref{hypothesis}, and let $(g_1,g_2)$ be an $(e_1,e_2)$-ppd stingray duo for  $G$, as in Definition~\ref{def:stingray}. 
As discussed in Section~\ref{s:sting}, to prove Theorem~\ref{thm:main}, our objective is to obtain  an upper bound strictly less than $1$ for the proportion of non-generating   stingray duos in $g_1^G\times g_2^G$, namely for the quantity $\rhonongen(g_1,g_2,\bfX)$ defined in \eqref{e:rhoX} (see also \eqref{e:rhonongen} and Lemma~\ref{lem:redOm}(a)). 
 Clearly these proportions are independent of the choices of the elements $g_1, g_2$ in their $G$-conjugacy classes $g_1^G, g_2^G$, and by Proposition~\ref{prop:om-class}(a),  these conjugacy classes are independent of the group $G$ in the range $\OmegaX_d(q)\lhdeq G\leq\GX_d(q)$. Also, if $|g_i|$ is a multiple of an $e_i$-ppd $r_i$ of $(q^u)^{e_i}-1$, then by Lemma~\ref{lem:redOm}, it is sufficient to prove the theorem, replacing each $g_i$ by $h_i:=g_i^{|g_i|/r_i}$, an element of  $e_i$-ppd prime order $r_i$ lying in $\OmegaX_d(q)$. Thus we may (and shall) henceforth assume the following.

\begin{hypothesis}\label{hyp}
\begin{enumerate}[{\rm (a)}]
    \item 
    $V=(\F_{q^u})^d$, and $G=\GX_d(q)$ of type $\bfX$, with $\bfX, u, \GX_d(q)$ as in one of the lines of Table~$\ref{tab:G}$.
    \item 
    $d=e_1+e_2$ with $2\leq e_2\leq e_1$ as in Table~$\ref{tab:stingraycond}$ (usually $d>8$ but this will often be specified explicitly). For $i=1,2$, $(q^u)^{e_i}-1$ has a primitive prime divisor $r_i$.
    \item 
    For $i=1,2$, $g_i\in G$ is an $e_i$-ppd stingray element of order $r_i$, in particular $g_i\in\OmegaX_d(q)$.
    \item 
    We usually use $(g, g')$  to denote a typical non-generating stingray duo with $g\in g_1^G, g'\in g_2^G$, so $H\coloneq\langle g,g'\rangle$ is a proper subgroup of $\OmegaX_d(q)$, and $H$ is not an orthogonal group if  $\bfX= \bfSp$ with $q$ even, see Lemma~$\ref{lem:spgen}$.
\end{enumerate}    
\end{hypothesis}

Our strategy for obtaining a useful upper bound for $\rhonongen(g_1,g_2,\bfX)$ is to 
 identify a collection $\mathcal M(\bfX)$ of proper subgroups of $G=\GX_d(q)$, usually (but not always) maximal subgroups, such that, for each non-generating stingray duo $(g,g')$, $H =\langle g,g'\rangle\leq M\cap\OmegaX_d(q)$ for at least one $M\in {\mathcal M}(\bfX)$. Then we have
 \begin{equation}\label{e:aschnongen0}
    \rhonongen(g_1,g_2,\bfX)  =  \frac{\mid \bigcup_{M\in {\mathcal M}(\bfX)}
    \{ \mbox{stingray duos $(g,g')$ in}\ g_1^G\times g_2^G \ \mbox{with $\langle g,g'\rangle \leq M$} \}| }{ \mbox{Number of stingray duos in} \ g_1^G\times g_2^G }
\end{equation}
 
 By Aschbacher's Theorem \cite{Asch}, maximal subgroups of $G$  not containing $\OmegaX_d(q)$ 
belong to one of nine `categories' (families of subgroups) denoted $\Asch_1, \ldots , \Asch_9$.
Our standard reference is the book of Kleidman and Liebeck~\cite{KL} (where the exact 
conditions are given for a subgroup in category $\Asch_i$ to be maximal in $G$, for dimension $d\geq 13$). 
Sometimes the categories $\Asch_i$ are enriched to contain some non-maximal subgroups (sometimes all the subgroups of the maximal ones) or, for example in the case  $\bfX\ne\bfL$ and $e_1=d/2$, $\Asch_1$ might include the stabilisers of all totally isotropic or non-degenerate $d/2$-subspaces even though the stabiliser $M$ of a $d/2$-dimensional non-degenerate subspace $U$  might be a proper subgroup of the stabiliser of the decomposition $V=U\oplus U^\perp$, so $M$ has index two in $N_G(M)$.

\begin{remark}\label{r:MX-Sp}
\begin{enumerate}[{\rm (a)}]
    \item 
    In summary, the subgroups in $\cM(\bfX)$ will be maximal among the elements of $\Asch_i$, for various $i$, but are not necessarily maximal in $G$. 
    Thus the collection ${\mathcal M}(\bfX)$ has a natural partition as a disjoint union ${\mathcal M}(\bfX)=\cup_{i=1}^9 {\mathcal M}_i(\bfX)$ with ${\mathcal M}_i(\bfX)\subseteq \Asch_i$ for each $i$.  
    \item 
     The case $\bfX=\bfSp$ with $q$  even  (so $G=\Sp_d(q)=\OmegaX_d(q)$) deserves special mention. By Lemma~\ref{lem:spgen},  $H\leq M$, where either $M$ is a maximal reducible subgroup of $G$ or $M$ is a maximal (proper) subgroup of some subgroup $\Omega_d^\pm(q)$ of $G$. In particular, it follows from Aschbacher's Theorem \cite{Asch} in this case that $H\leq M$ for some $M\in \cup_{i\ne 8}\Asch_i$, and hence we may assume that ${\mathcal M}_8(\bfSp)=\emptyset$ if $q$ is even. Also in this case, ${\mathcal M}_9(\bfSp)$ will consist of the $\Asch_9$-subgroups of an orthogonal subgroup $\Omega_d^\pm(q)$ of $G$. 
     \end{enumerate}    
\end{remark}

We often approximate the contribution 
\[
\frac{\mid \bigcup_{M\in {\mathcal M}_i(\bfX)}
    \{ \mbox{stingray duos $(g,g')$ in}\ g_1^G\times g_2^G \ \mbox{with $\langle g,g'\rangle \leq M$} \}| }{ \mbox{Number of stingray duos in} \ g_1^G\times g_2^G }
\]
to $\rhonongen(g_1,g_2,\bfX)$ in \eqref{e:aschnongen0} from subgroups in $\cM_i(\bfX)$ by the upper bound
\[
\sum_{M\in {\mathcal M}_i(\bfX)}\frac{\mid 
    \{ \mbox{stingray duos $(g,g')$ in}\ g_1^G\times g_2^G \ \mbox{with $\langle g,g'\rangle \leq M$} \}| }{ \mbox{Number of stingray duos in} \ g_1^G\times g_2^G}.
\]
We use several strategies to reduce this upper bound to avoid too much over-counting. Firstly, we try to choose the subsets ${\mathcal M}_i(\bfX)$ as small as possible subject to the following: if $H$ is contained in some subgroup lying in $\Asch_i$, for some $i$, then we must ensure 
that there exists some $M\in \cup_{j=1}^i {\mathcal M}_j(\bfX)$ with $H\leq M$. However, the set 
${\mathcal M}_i(\bfX)$ will not usually contain all the subgroups in $\Asch_i$, not even all those containing $H$. For example, if $i=1$ and $\bfX\ne\bfL$, then we will see in Lemma~\ref{lem:dim} that the only maximal subgroups in $\Asch_1$ containing an $(e_1,e_2)$-ppd stingray duo are stabilisers of nondegenerate $e_j$-subspaces (for $j=1,2$). Thus we may take 
${\mathcal M}_1(\bfX)$ to consist of the stabilisers of non-degenerate subspaces of dimensions $e_1$ and $e_2$ (and no other subgroups). Secondly, we analyse the categories in order so that, by the time we consider cases where $H$ is contained in a subgroup of $\Asch_i$, we may assume that $H$ is not contained in any subgroup in $\cup_{1\leq j<i}\Asch_j$.  
To facilitate choosing the ${\mathcal M}_i(\bfX)$ with this property, we define:
\begin{equation}\label{e:aschnongen00}
   {\mathcal M}_1^\ast(\bfX) =\emptyset, \quad  \mbox{and for $1<i\leq 9$,}\quad  {\mathcal M}_i^\ast(\bfX) = \cup_{1\leq j<i} {\mathcal M}_j(\bfX). 
\end{equation} 
Moreover, for $1\leq i\leq 9$, we let 
\begin{equation}\label{e:s1}
 \mathcal{S}_i(g_1,g_2,\bfX)=   \bigcup_{M\in {\mathcal M}_i(\bfX)}
    \left\{\mbox{duos $(g,g')$ in  $g_1^G\times g_2^G$}\ \middle|
    \begin{array}{lll}
     &\mbox{$\langle g,g'\rangle \leq M$ but $\langle g,g'\rangle \nleq L$}\\
    &\quad\qquad\mbox{for any $L\in\cM_i^\ast(\bfX)$}
    \end{array}
    \right\}
\end{equation}
and  we define 
 \begin{equation}\label{e:asch2}
{\rm Prob}_i(g_1,g_2,\bfX) =    \frac{\mid 
\mathcal{S}_i(g_1,g_2,\bfX)\mid}{ \mbox{Number of stingray duos in} \ g_1^{G}\times g_2^{G}}.
\end{equation}
We therefore have
\begin{equation}\label{e:aschnongen}
    \rhonongen(g_1,g_2,\bfX)  = \sum_{i=1}^9 {\rm Prob}_i(g_1,g_2,\bfX).
\end{equation}

We will find upper bounds for each ${\rm Prob}_i(g_1, g_2, \bfX)$ as a function of $d$ and $q$. We make a few comments about the analysis. 
\begin{itemize}
    \item For some  subgroups $M\in\cM_i(\bfX)$, it is possible to count the numbers of stingray duos they contain, leading to expressions for the numerator in  \eqref{e:asch2} which allow some cancellation with the expressions for the denominator. For example, this  occurs for $\Asch_1$ with $\bfX\ne \bfL$ in the proof of Proposition~\ref{prop:redUSO}. 
    
    \item For some $i>1$,  the extra restriction that  $\langle g,g'\rangle$ does not lie in any subgroup from $\cM_j(\bfX)$ for $j<i$ is useful. For example, for $\Asch_2$, we will see in Lemma~\ref{lem:imprim1} that the family $\cM_2(\bfX)$ can be restricted to maximal imprimitive subgroups of $G$ preserving a decomposition of $V$ as a direct sum of $d$ subspaces of dimension $1$.
    
    \item In all cases we use Lemma~\ref{lem:fixg1} and Propositions~\ref{p:Lbds} and~\ref{p:Xbds} to estimate the denominators in \eqref{e:aschnongen0} and  \eqref{e:asch2} (see also \eqref{e:asch3} below). 
    
    \item The last technique which we now describe is useful in cases where we can specify the subgroups $M\in\cM_i(G)$ containing a given $g$, see for example Proposition~\ref{prop:redUSO} (for $\Asch_1$) or Proposition~\ref{prop:imprim} (for $\Asch_2$). 
 For each  $i$, the family $\cM_i(\bfX)$ is a union of $G$-conjugacy classes of subgroups, and hence the set $\mathcal{S}_i(g_1,g_2,\bfX)$ in \eqref{e:s1} is closed under conjugation so that the number of stingray duos in $\mathcal{S}_i(g_1,g_2,\bfX)$ with first entry  $g\in g_1^{G}$ is independent of the choice of $g$.
Therefore, for $M\in\cM_i(\bfX)$ such that $g\in  g_1^G\cap M$, setting
\begin{equation}\label{e:siM}
\mathcal{S}_i(g, M) =   
 \left\{g'\in g_2^{G}\cap M\ \middle| 
    \begin{array}{lll}
    &\mbox{$(g,g')$ a stingray duo, $\langle g,g'\rangle \leq M$, but }\\ 
    &\quad\quad\quad\mbox{$\langle g,g'\rangle \not\le L$ for any $L\in\cM_i^\ast(\bfX)$}
    \end{array} 
    \right\}
\end{equation}
the cardinality of $\mathcal{S}_i(g_1,g_2,\bfX)$ as given in \eqref{e:s1} satisfies, for a fixed $g\in g_1^G$, 
\begin{equation}\label{e:asch4}
    |\mathcal{S}_i(g_1,g_2,\bfX)| = |g_1^{G}|\cdot \left|  \bigcup_{M\in\cM_i(\bfX), g\in M} \mathcal{S}_i(g,M)\right|\leq |g_1^{G}|\cdot \sum_{M\in\cM_i(\bfX), g\in M} \left|\mathcal{S}_i(g,M)\right|.
\end{equation}
To relate this back to estimating ${\rm Prob}_i(g_1,g_2,\bfX)$ we use the following lemma.
\end{itemize}

\begin{lemma}\label{lem:fixg1}
 Let $G, \bfX, e_1, e_2, g_1, g_2$ be  as in Hypothesis~$\ref{hypothesis}$, and choose a fixed $g\in g_1^G$. Then 
\openup 2pt
\begin{align}\label{e:asch3}
    {\rm Prob}_i(g_1,g_2,\bfX) &= 
    \frac{\left|\bigcup_{M\in\cM_i(\bfX), g\in M} \mathcal{S}_i(g,M)\right|}{\left|N(d,q^u,e_1,\bfX)\right|}
      \leq   
\sum_{M\in\cM_i(\bfX), g\in M}
\frac{\left|\mathcal{S}_i(g,M)\right|}{\left|N(d,q^u,e_1,\bfX)\right|},
\end{align}
with $\mathcal{S}_i(g,M)$ as in \eqref{e:siM} and
 \begin{equation}\label{e:ndex}
    N(d,q^u,e_1,\bfX) \coloneq \{g'\in g_2^{G} \mid (g,g') \ \mbox{is a stingray duo  in $g_1^{G}\times g_2^{G}$}\}.
\end{equation}
\end{lemma}

    \begin{proof}
  Let $x\in g_1^G$, let $A(x), B(x)$ be the set of all elements $y\in g_2^G$ such that $(x,y)\in \mathcal{S}_i(g_1,g_2,\bfX)$, or $(x,y)$ is a stingray duo in $g_1^G\times g_2^G$, respectively. Then,  with $\mathcal{S}_i(x,M)$ as in \eqref{e:siM}, we have $A(x)= \bigcup_{M\in\cM_i(\bfX), x\in M} \mathcal{S}_i(x,M)$. 
  Also for each $z\in G$ it is clear that $A(x)^z=A(x^z)$ and $B(x)^z=B(x^z)$, and hence the sizes of these subsets do not depend on the choice of $x\in g_1^G$. Now fix $g\in g_1^G$. Then $B(g)=N(d,q^u,e_1,\bfX)$ as in \eqref{e:ndex} and it follows from \eqref{e:asch2} that
  \[
   {\rm Prob}_i(g_1,g_2,\bfX) = \frac{|g_1^G|\cdot |A(g)|}{|g_1^G|\cdot |B(g)|}
  = \frac{ |A(g)|}{|B(g)|} = 
    \frac{\left|\bigcup_{M\in\cM_i(\bfX), g\in M} \mathcal{S}_i(g,M)\right|}{\left|N(d,q^u,e_1,\bfX)\right|}
  \]
and the assertions in \eqref{e:asch3} follow.       
    \end{proof}


\section{Bounds on numbers of subspaces and duos}\label{s:explicit}

For our analysis we need explicit bounds on the numbers of (nondegenerate) $e$-subspaces of $V$ for the various types $\bfX$, and on related quantities. Since  the parity of $e$ is restricted as in Table~\ref{tab:one}, we refine bounds available from, for example, \cite{PSer}. To facilitate use of the result in \cite{PSer}, we define, for integers $k, d, q$ with $1\leq k\leq d$ and $q\geq2$, and $\tau=\pm1$, 
\begin{equation}\label{e:OmDel1}
    \omega(k,d;\tau q) \coloneq\prod_{i=k}^d (1-(\tau q)^{-i}),
\end{equation}
and, if $k<d$,
\begin{equation}\label{e:OmDel2}
    \Delta(k,d;\tau q)\coloneq \frac{\omega(1,d;\tau q)}{\omega(1,k;\tau q)\omega(1,d-k;\tau q)}
    = \frac{\omega(k+1,d;\tau q)}{\omega(1,d-k;\tau q)}.
\end{equation}

\begin{lemma}\label{lem:Delbds}
Let $d\geq2$, $q\geq2$, and let $k$ satisfy $1\leq k\leq d-1$, and let $\Delta(k,d;q)$ be as in \eqref{e:OmDel2}. Then
\[
1<\Delta(k,d;q)< (1-q^{-1}-q^{-2})^{-1};
\]
and  
\[    
    \begin{array}{clllcl}
(1+q^{-1})^{-1} &<&\Delta(k,d;-q)&<& 1 + q^{-k-1}& \mbox{if $k$ is even,}\\
(1-q^{-k-1})(1+q^{-1})^{-1}&<&\Delta(k,d;-q)&<& 1&\mbox{if $k$ is odd.} 
    \end{array}
\]
Moreover, the bounds in Table~$\ref{t:Del}$ hold, and also $\Delta(k,d; q)< 81/71<8/7$ if $q\geq 9$.
\end{lemma}

\begin{table}
  \caption{Bounds for $\Delta(k,d;\pm q)$ for Lemma~\ref{lem:Delbds}}
\begin{tabular}{ccccccccc}
  \toprule
  $q=2$ & $q=3$ & $q\geq4$ & $<$& $\Delta$ &$<$&$q\geq4$ & $q=3$ & $q=2$ \\
  \midrule
  1 & 1 & 1 &$<$& $\Delta(k,d;q)$&$<$& $16/11$ & $9/5$ & $4$\\
$2/3$& $3/4$ & $4/5$ &$<$&$\Delta(k,d;-q)$&$<$& $65/64$ & $28/27$ & $9/8$\\
  &&&&$k$ even&&&&\\
$5/8$ & $20/27$ & $51/64$ &$<$& $\Delta(k,d;-q)$ & $<$ & $1$ & $1$ & $1$\\
  &&&& $k\geq3$, odd&&&&\\
\bottomrule
\end{tabular}
\label{t:Del}
\end{table}

\begin{proof}
The bounds for $\Delta(k,d; q)$ and $\Delta(k,d;-q)$ are proved in \cite[Lemmas 3.1(b) and 3.2(b)]{PSer}. Applying these bounds yields the entries in Table~\ref{t:Del}. Finally, if $q\geq 9$, then using the upper bound for $\Delta(k,d; q)$ we have $\Delta(k,d; q)< (1-q^{-1}-q^{-2})^{-1}\leq (1-9^{-1}-9^{-2})^{-1} =  81/71<8/7$.
\end{proof}

Let $V=\F_{q^u}^d$ be the natural module for $\GX_d(q)$. Let $\Binom{V}{e}$ denote the set of all $e$-subspaces of $V$ if $\bfX=\bfL$, and the set of all nondegenerate $e$-subspaces of $V$ if $\bfX\in\{\bfSp,  \bfU\}$. If $\bfX=\bfO^\eps$, let
$\Binom{V}{e}^\eps_\eta$ denote the set of all nondegenerate $e$-subspaces of $V$ of type $\eta\in\{-,\circ,+\}$.
For a positive integer $e<d$ and a $\GX_d(q)$-conjugacy class $\cC$ of $e$-ppd stingray elements, the estimates use the following sets
\begin{equation}\label{e:ude}
\cU(d, q^u,e,\bfX)\coloneq \left\{\begin{array}{ll}
	\Binom{V}{e} &\mbox{if $\bfX\in \{ \bfL, \bfSp,  \bfU\}$,}\\
	\Binom{V}{e}^\eps_-	&\mbox{if $\bfX= \bfO^\eps$.}
\end{array}\right. 
\end{equation}
Also, let
\begin{equation}\label{e:dde}
\cD(d,q^u,e,\bfX)\coloneq
	\{(U,U')\mid U\in \cU(d,q^u,e,\bfX), U'\in \cU(d,q^u,d-e,\bfX), V=U\oplus U' \}.
\end{equation}
The subspace pairs in $\cD(d,q^u,e,\bfL)$ and subspaces in $\cU(d,q^u,e,\bfX)$ may be used to partition the elements of $\cC$ as follows:
\begin{equation}\label{e:cCux}
  \cC(\delta,\bfX)\coloneq \left\{\begin{array}{ll}
	\{g\in\cC\mid (U_g,F_g)=\delta\}	&\mbox{if $\bfX= \bfL$ and $\delta\in\cD(d,q^u,e,\bfL)$,}\\
	\{g\in\cC\mid U_g=\delta\}	&\mbox{if $\bfX\ne  \bfL$ and $\delta\in\cU(d,q^u,e,\bfX)$.}\\
\end{array}\right.  
\end{equation}
This partition leads to expressions for $|\cC|$, which we then extend to expressions for counting stingray duos from a pair of conjugacy classes.
Some of the estimates we need follow from work in \cite{GNP3} for type $\bfL$, so we treat this type separately from the other classical types.

\subsection{Estimates for the linear type \texorpdfstring{$\bfL$}{L} }

\begin{table}
  \caption{Values for $\bfa(q, \bfX), \bfb(q, \bfX)$ for Propositions~\ref{p:Lbds} and~\ref{p:Xbds}}
\begin{tabular}{cccc|cccc}
  \toprule
  $\bfX$ & $q$ & $\bfa(q, \bfX)$ & $\bfb(q, \bfX)$& $\bfX$ &$q$&$\bfa(q, \bfX)$ & $\bfb(q, \bfX)$  \\
  \midrule
  $\bfL$ & $2$ & $1$ & $4$& 
  $\bfU$ &$2$  &$5/8$ & $1$  \\
     & $\geq3$ & $1$ & $9/5$&     &$\geq3$&$20/27$ & $1$  \\
  \midrule
  $\bfSp$ & $2$ & $1$ & $16/11$& 
  $\bfO^\pm$ &$2$  &$5/22$ & $12/11$  \\
     & $\geq3$ & $1$ & $8/7$&     &    $3$&$20/61$ & $54/71$  \\
    &           &   &       &     &$\geq4$&$47/128$ & $160/239$  \\
  \midrule
  $\bfO^-$ & $2$ & $1/2$ & $12/11$& 
  $\bfO^+$ &$2$  &$5/22$ & $8/11$  \\
     & $3$ & $1/2$ & $54/71$&     &    $3$&$20/61$ & $81/142$  \\
    &  $\geq4$ & $1/2$& $160/239$&     &$\geq4$&$47/128$ & $128/239$  \\
\bottomrule
\end{tabular}
\label{t:abLU}
\end{table}

\begin{proposition}\label{p:Lbds}
Assume that Hypothesis~$\ref{hyp}$ holds with $\bfX=\bfL$, but without the assumption that $d>8$ and allowing $e_1, e_2\geq1$. Take $G=\GL_d(q)$, and write $\cC_i=g_i^G$ so we may use the  expression in \eqref{e:cCux} for each $i$. Then the following hold, where $\omega(1,d;q)$ is as in \eqref{e:OmDel1}.
\begin{enumerate}[{\rm (a)}]
    \item  For $\delta\in \cD(d,q,e_i,\bfL)$, and with $\cC_i(\delta,\bfL)$ as in \eqref{e:cCux}, 
    \[
    c(d,q,e_i,\bfL) \coloneq |\cC_i(\delta,\bfL)| = \frac{|\GL_{e_i}(q)|}{q^{e_i}-1} = \frac{q^{e_i^2}\omega(1,e_i;q)}{q^{e_i}-1}.
    \]
    \item Recalling that $G=\GL_d(q)$, for each $i$, 
    \[
    |C_{G}(g_i)| = q^{(d-e_{i})^2} (q^{e_i}-1) \omega(1,d-e_i;q) \quad \mbox{and}\quad  |g_i^G|=  \frac{q^{e_i^2+2e_1e_2}\omega(1,d;q)}{(q^{e_i}-1) \omega(1,d-e_i;q)}.
    \]
    \item  For $g\in g_1^G$ and $N(d,q,e_1,\bfL)$ as in \eqref{e:ndex}, 
    \[
    |N(d,q,e_1,\bfL)| = \frac{q^{e_2^2+2e_1e_2}\omega(1,e_2;q)}{q^{e_2}-1} = q^{2e_1e_2}\cdot c(d,q,e_2,\bfL),
    \]
    and
    \[
    \frac{|C_{G}(g_1)|}{|N(d,q,e_1,\bfL)|} = \frac{(q^{e_1}-1)(q^{e_2}-1)}{q^{2e_1e_2}}.
    \]
    \item For  $\cD(d,q,e_1,\bfL)$ as in \eqref{e:dde}, $|\cD(d,q,e_1,\bfL)|= q^{2e_1e_2}\frac{\omega(1,d;q)}{\omega(1,e_1;q)\omega(1,e_2;q)} = q^{2e_1e_2} \Delta(e_1,d;q)$. Further, for $\bfa(q, \bfL),\bfb(q, \bfL)$ as in Table~$\ref{t:abLU}$,  then $\bfa(q, \bfL)\leq |\cU(d,q, e_1, \bfL)|\leq \bfb(q, \bfL)$ and
      \[
      \bfa(q, \bfL)\cdot q^{2e_1e_2}\leq |\mathcal{D}(d,q, e_1, \bfL)|\leq \bfb(q, \bfL)\cdot q^{2e_1e_2}.
      \]
\end{enumerate}
    
\end{proposition}

\begin{proof}
    (a) By \cite[Theorem 3.7(b)]{GNP3}, $|\cC_i(\delta,\bfL)| = \frac{|\GL_{e_i}(q)|}{q^{e_i}-1}$, and part (a) follows, since $|\GL_{e_i}(q)|=q^{e_i^2}\omega(1,e_i;q)$.
    
    (b) By Lemma~\ref{lem:uperp}(c), $|C_{G}(g_i)| = (q^{e_i}-1)\cdot |\GL_{d-e_i}(q)|$, and the first equality follows. The second equality, follows from the fact that $|g_i^G|= |G|/|C_{G}(g_i)|$ and $d=e_1+e_2$, see also \cite[Theorem 3.7(a)]{GNP3}.

    (c) Since each element $g\in g_1^G$ lies in the same number of stingray duos in $g_1^G\times g_2^G$, it follows from \eqref{e:ndex} that the total number of stingray duos is $|g_1^G|\cdot |N(d,q,e_1,\bfL)|$. Further, by \cite[Theorem 7.2]{GNP3}, the number of stingray duos is equal to $|g_1^G|\cdot |g_2^G|\cdot \omega(1,e_1;q)\omega(1,e_2;q)/\omega(1,d;q)$. Therefore
    \[
    |N(d,q,e_1,\bfL)| = |g_2^G|\cdot \frac{\omega(1,e_1;q)\omega(1,e_2;q)}{\omega(1,d;q)}
    \]
    so the first expression for $|N(d,q,e_1,\bfL)|$ follows from part~(b), and the second from part~(a). Finally, using part~(b),
    \[
     \frac{|C_{G}(g_1)|}{|N(d,q,e_1,\bfL)|} =    q^{e_{2}^2} (q^{e_1}-1) \omega(1,e_2;q) \cdot\frac{q^{e_2}-1}{q^{e_2^2+2e_1e_2}\omega(1,e_2;q)}    = \frac{(q^{e_1}-1)(q^{e_2}-1)}{q^{2e_1e_2}}.
    \]

    (d) The cardinality $|\cD(d,q,e_1,\bfL)|$ is the number of ordered pairs $(U,F)$ such that $\dim(U)=e_1, \dim(F)=e_2$, and $V=U\oplus F$. As $G=\GL_d(q)$ is transitive on these pairs, this number is equal to the index $|\GL_d(q)|/(|\GL_{e_1}(q)|\cdot |\GL_{e_2}(q)|)$ and the claimed expression for $|\cD(d,q,e_1,\bfL)|$  follows. Similarly,
    $|\cU(d,q,e_2,\bfL)|=\Delta(e_1,d;q)$. Finally, the inequalities $\bfa(q, \bfL)\leq\Delta(e_1,d;q)\leq\bfb(q, \bfL)$ used in Table~\ref{t:abLU} follow from Lemma~\ref{lem:Delbds}.
\end{proof}

\subsection{Estimates for all classical types \texorpdfstring{$\bfX\ne \bfL$}{X not L} }

Now we prove an analogue of Proposition~\ref{p:Lbds} for the other classical types. The estimates for the orthogonal type $\bfO^\varepsilon$ require a rather delicate analysis of a quantity  depending on $\varepsilon=\pm1$, which we give in the preliminary Lemma~\ref{lem:gam}.

\begin{lemma}\label{lem:gam}
For $d=e_1+e_2$ and $q\ge 2$ as in Hypothesis~$\ref{hyp}$, and $\varepsilon \in \{-1,1\}$, define
 \begin{equation}\label{d:gam}
 \Gamma(\varepsilon) = 1 - \frac{q^{-e_1/2} +\varepsilon q^{-e_2/2}}{1+\varepsilon q^{-d/2}}.     
 \end{equation}
Then 
 $1\leqslant \Gamma(-1)< 1+\frac{1}{q}$, and if $d\geqslant 8$ then also $1-\frac{1}{q}- \frac{1}{q^3}\leqslant \Gamma(1)< 1$. 
 Moreover, for $d\ge 10$ we have $\frac{5}{11} \le \Gamma(1) < 1$ for $q=2$, and $\frac{40}{61} \le \Gamma(1) < 1$ for $q=3$.
\end{lemma}

\begin{proof}
Since $e_2\le e_1$, the quantity
\[
1-\Gamma(\varepsilon)=\frac{q^{-e_1/2} +\varepsilon q^{-e_2/2}}{1+\varepsilon q^{-d/2}}
= \frac{q^{(d-e_1)/2} +\varepsilon q^{(d-e_2)/2}}{q^{d/2}+\varepsilon}
= \frac{q^{e_2/2} +\varepsilon q^{e_1/2}}{q^{d/2}+\varepsilon}
\]
which is positive for $\varepsilon=1$
and non-positive for $\varepsilon=-1$. In particular, 
$\Gamma(1)  < 1$ and $1\le \Gamma(-1)$. 
Consider first $\varepsilon=-1.$ Then 
$1-\Gamma(-1)=\frac{f_1(e_1)}{q^{d/2} -1}$, where  
$f_{1}(x) = q^{(d-x)/2}-q^{x/2}$. As  $d/2\leq e_1\leq d-2$, and 
as $f_1(x)$ is a decreasing function on the interval
$d/2\leq x\leq d-2$, it follows that
\[
1-\Gamma(-1)\geq \frac{f_1(d-2)}{q^{d/2}-1} = \frac{q-q^{(d-2)/2}}{q^{d/2}-1}
=-\frac{1}{q} \cdot \frac{q^{d/2}-q^2}{q^{d/2}-1} > -\frac{1}{q}. 
\]
Therefore $1\leq \Gamma(-1)<1+1/q$ as claimed.

Now let $\varepsilon=1.$ 
Then $1-\Gamma(1)=\frac{f_2(e_1)}{q^{d/2} +1}$, where 
$f_{2}(x) = q^{(d-x)/2}+q^{x/2}$.
 Now  $d/2\leq e_1\leq d-2$, and the function $f_2(x)$ is increasing on the interval $[d/2, d-2]$. 
Hence, 
\begin{equation}\label{eq:gammin}
1-\Gamma(1) = \frac{f_2(e_1)}{q^{d/2}+1}
\leq
\frac{q^{(d-2)/2}+q} {q^{d/2}+1} = \frac{1}{q}\left( 1 +\frac{q^2-1}{q^{d/2}+1} \right).
\end{equation}
As the right-hand side of \eqref{eq:gammin} is less than $
\frac{1}{q}+\frac{1}{q^3}
$ for $d\ge 8,$
we find $1-\frac{1}{q} - \frac{1}{q^3} < \Gamma(1) < 1 $, as claimed.
 For $d\ge 10$ and  $q\in \{2,3\}$
we evaluate the upper bound in \eqref{eq:gammin} to obtain for $q=2$ the bound
$1-\Gamma(1) \le \frac{1}{2} \cdot
\left(1 + \frac{3}{33}\right) = \frac{6}{11}. $ For $q=3$ we obtain
$1-\Gamma(1) \le \frac{1}{3} \cdot
\left(1 + \frac{8}{244}\right) = \frac{21}{61}. $
\end{proof}

\begin{proposition}\label{p:Xbds}
Assume that Hypothesis~$\ref{hyp}$ holds with $\bfX\ne \bfL$ but without the assumption that $d>8$ and allowing $e_1, e_2\geqslant 1$  with parity restrictions as in  Table~$\ref{tab:one}$. Take $G=\GX_d(q)$, and let $\cC_1=g_1^G$ and $\cC_2=g_2^G$, let $T$ be the torus with $|T|$ as in Table~$\ref{tab:one}$, and let   $\mathcal{U}(d,q^u, e_1, \bfX)$ be as in \eqref{e:ude} and $\mathcal{D}(d,q^u, e_1, \bfL)$ be as in \eqref{e:dde}. Then the following hold. 

\begin{enumerate}[{\rm (a)}]
    \item For $U\in \cU(d,q^u,e_i,\bfX)$ and with $\cC_i(U,\bfX)$ as in \eqref{e:cCux}, for each $i$, 
    \[
    c(d,q^u,e_i,\bfX)\coloneq|\cC_i(U,\bfX)|
    =\begin{cases}\openup=5pt
    \frac{|\GX_{e_i}(q)|}{|T|} & \mbox{if  $\bfX=\bfU$ or $\bfSp$,}\\
    \frac{|\GO_{e_i}^{-}(q)|}{|T|} & \mbox{ if $\bfX=\mathbf{O^\varepsilon}$.}
    \end{cases}
    \]
     \item For each $i$, $|g_i^G| =  c(d,q^u,e_i,\bfX)\cdot |\cU(d,q^u,e_i,\bfX)|$, and
      \[
    |C_G(g_i)|= \begin{cases}
    |T|\cdot |\GX_{d-e_i}(q)| & \mbox{if  $\bfX=\bfU$ or $\bfSp$,}\\
    |T|\cdot |\GO_{d-e_i}^{-\varepsilon}(q)| & \mbox{if $\bfX=\mathbf{O^\varepsilon}$, with $\varepsilon =\pm$.}
    \end{cases}     
      \]
      \item For $g\in g_1^G$, and  $N(d,q^u,e_1,\bfX)$ as in \eqref{e:ndex}, 
\[
|N(d,q^u,e_1,\bfX)|=     \mathbf{k}(d,q^u, e_1,\bfX)\cdot|\cU(d,q^u,e_1,\bfX)|\cdot c(d,q^u,e_2,\bfX)
\]
and setting  $(\alpha, \beta)=(1,1), (1/2, 1), (1/2,\varepsilon)$ for  $\bfX=\bfU, \bfSp, \bfO^\varepsilon$, respectively,
\[
\frac{|C_G(g_1)|}{|N(d,q^u,e_1,\bfX)|} = 
     \frac{(q^{\alpha e_1}+1)(q^{\alpha e_2}+\beta)}{\mathbf{k}(d,q^u, e_1,\bfX)\cdot|\cU(d,q^u,e_1,\bfX)|}
\]
and provided 
$(\bfX,e_1, e_2,q)\ne (\bfU,1,1, 2)$,  for each $e\in\{e_1,e_2\}$, 
        \[
    |\cD(d,q^u,e,\bfX)| = 
	\mathbf{k}(d,q^u,e,\bfX)\cdot |\cU(d,q^u,e,\bfX)|^2
\]    
where $1-3/(2q^u)\leq \mathbf{k}(d,q^u,e,\bfX)<1$, for all  $d, e$.

     \item Moreover,  
        \[
        \begin{array}{cl}
      |\mathcal{U}(d,q^u, e_1, \bfX)|  &= \left\{ \begin{array}{ll}
     q^{2e_1e_2}\cdot  \Delta(e_1,d; -q)&\mbox{if $\bfX=\bfU$}  \\
     q^{e_1e_2}\cdot  \Delta(e_1/2,d/2; q^2)&\mbox{if $\bfX=\bfSp$}  \\
     q^{e_1e_2}\cdot\frac{1}{2}\cdot  \Delta(e_1/2,d/2; q^2)\cdot \Gamma(\varepsilon)&\mbox{if $\bfX=\mathbf{O^\varepsilon}$}  \\
\end{array}
\right.
        \end{array}
        \]
    with $\Gamma(\varepsilon)$ as in \eqref{d:gam}. Also,  for $\bfa(q, \bfX)$ and $\bfb(q, \bfX)$ as in Table~$\ref{t:abLU}$, and  provided $(\bfX, d, e_1,q)\ne (\bfU, 2, 1,2)$, {and that $d$ is at least $3,4,10$ for $\bfX= \bfU, \bfSp, \mathbf{O^\varepsilon}$ respectively,} we have
  \[
        \bfa(q, \bfX)\cdot q^{u e_1e_2}\leq |\mathcal{U}(d,q^u, e_1, \bfX)|\leq \bfb(q, \bfX)\cdot  q^{u e_1e_2}.
        \]
\end{enumerate}

\end{proposition}

We remark on the final assertion in part (c): it was shown in~\cite[Theorem 1.1 and Table~1]{GNP1}   that  the proportion of subspace pairs $(U,U')$ in $\cU(d,q^u,e,\bfX)\times \cU(d,q^u,d-e,\bfX)$ with the property  $V=U\oplus U'$ is at least $1-c/q$, where $c$ is $5/3, 2, 43/16$ for $\bfX=\bfSp, \bfU, \bfO^\eps$, respectively.
However, the arguments in \cite{GNP1} for 
  ${\bf X}={\bf O}^\eps$ required $q\geq3$. Recently Glasby, Ihringer and Mattheus~\cite{GIM} proved that the proportion is at least $1-3/(2q^u)$ all $q\geq2$ and all types $\bfX\ne\bfL$ apart from $(\bfX,e_1,e_2,q)=(\bfU,1,1,2)$ (see also \cite{DBVV}). We will use this bound repeatedly.

\begin{proof}
  (a) and  (b)\quad By Lemma~\ref{lem:uperp}(c), for $g\in g_i^G$, $|C_G(g)|$ is as stated in part (b), and  $C_{G}(g)$ is contained in the stabiliser in $G$ of the nondegenerate $e_i$-subspace $U_g$, and hence stabilises also the decomposition $V=U_g\oplus F_g$. Thus the subset $\cC_i(U,\bfX)$, with $U=U_g\in \cU(d,q^u,e_i,\bfX)$, is a block of imprimitivity for the transitive conjugation action of $G$ on $g_i^G$. The number of blocks of imprimitivity is $|G:G_U|$ which is $|\cU(d,q^u,e_i,\mathbf{X)}|$, and the size of a block is $|G_U:C_{G}(g)|$ which, by Lemma~\ref{lem:uperp}(c), is $|\GX_{e_i}(q)|/|T|$ if  $\bfX\ne \mathbf{O^\varepsilon}$, and 
  $|\GO_{e_i}^{-}(q)|/|T|$ if  $\bfX=\mathbf{O^\varepsilon}$, 
  with $|T|$ as in Table~$\ref{tab:one}$. This proves both part (a) and part (b).

  (c) First we evaluate $|\cD(d,q^u,e,\bfX)|$. By \eqref{e:dde}, $|\cD(d,q^u,e,\bfX)|$ is the number  of pairs $(U, U')\in \cU(d,q^u,e,\bfX)\times \cU(d,q^u,d-e,\bfX)$, such that $V=U\oplus U'$. This number was proved in \cite[Theorem 1.1]{GIM}, provided $(\bfX,e_1, e_2,q)\ne (\bfU, 1,1,2)$,  to equal
$\mathbf{k}(d,q^u,e,\bfX)\cdot |\cU(d,q^u,e,\bfX)|\cdot |\cU(d,q^u,d-e,\bfX)|$, where $1-3/(2q^u)\leq \mathbf{k}(d,q^u,e,\bfX)<1$, for all $\bfX, d, e$. This yields the third equality in part (c) on noting that $|\cU(d,q^u,e,\bfX)|= |\cU(d,q^u,d-e,\bfX)|$. Now we prove the other two equalities.

Fix $g\in g_1^G$, and recall that $N(d,q^u,e_1,\bfX)$ (see \eqref{e:ndex}) is the set of $g'\in g_2^G$ such that $(g,g')$ is a stingray duo, or equivalently $(U_g, U_{g'})$ forms a decomposition of $V$, that is to say, $(U_g,U_{g'})\in \cD(d,q^u,e_1,\bfX)$. Since $g$ is given, so too is the subspace $U=U_g$, hence the number of $g'$ producing a stingray duo is equal to the number of $e_2$-subspaces $U'\in \cU(d,q^u,e_2,\bfX)$ such that $(U,U')\in \cD(d,q^u,e_1,\bfX)$, times the number $|\cC_2(U',\bfX)|=c(d,q^u,e_2,\bfX)$ of choices for $g'$ such that $U_{g'}=U'$. The number of such $U'$ is $|\cD(d,q^u,e_1,\bfX)|/|\cU(d,q^u,e_1,\bfX)|$, which we have just shown to be   $\mathbf{k}(d,q^u,e,\bfX)\cdot |\cU(d,q^u,e_2,\bfX)|$. Thus $|N(d,q^u,e_1,\bfX)| = \mathbf{k}(d,q^u,e,\bfX)\cdot |\cU(d,q^u,e_2,\bfX)| \cdot c(d,q^u,e_2,\bfX)$, as asserted in part (c).   

Finally, by Lemma~\ref{lem:uperp}(c), $C_{G}(g_1)=T\times \GX_{e_2}(q)$ if  $\bfX\ne\mathbf{O^\varepsilon}$ and $T\times \GO_{e_2}^{-\varepsilon}(q)$ if  $\bfX=\mathbf{O^\varepsilon}$, with $|T|=q^{e_1}+1, q^{e_1/2}+1$, or $q^{e_1/2}+1$, for $\bfX= \bfU, \bfSp,$ or $\mathbf{O^\varepsilon}$, respectively (see Table~\ref{tab:one}). By part (a),  $c(d,q^u,e_2,\bfX)$ is  $|\GX_{e_2}(q)|/|T'|$ if  $\bfX\ne \mathbf{O^\varepsilon}$ or 
  $|\GO_{e_2}^{-}(q)|/|T'|$ if  $\bfX=\mathbf{O^\varepsilon}$, 
  with $|T'|$ as in Table~$\ref{tab:one}$ (with $e=e_2$). The value for $|C_{G}(g_1)|/|N(d,q^u,e_1,\bfX)|$ now follows immediately if $\bfX= \bfU$ or $\bfSp,$ while if $\bfX=\mathbf{O^\varepsilon}$, then this quantity is 
  \[
  \frac{(q^{e_1/2}+1)\cdot |\GO_{e_2}^{-\varepsilon}(q)|\cdot  (q^{e_2/2}+1)}{\mathbf{k}(d,q^u, e_1,\mathbf{O^\varepsilon})\cdot|\cU(d,q^u,e_1,\mathbf{O^\varepsilon})|\cdot |\GO_{e_2}^{-}(q)|} =  \frac{(q^{e_1/2}+1)(q^{e_2/2}+\varepsilon)}{\mathbf{k}(d,q^u, e_1,\mathbf{O^\varepsilon})\cdot|\cU(d,q^u,e_1,\mathbf{O^\varepsilon})|} 
  \]
  as asserted, and the proof of part (c) is complete. 

(d) Let $\mathcal{U}\coloneq \mathcal{U}(d,q^u, e_1, \bfX)$. Suppose first that $\bfX= \bfU$, so $d\geq3$. Since both $e_1, e_2$ are odd (see Table~$\ref{tab:one}$) and  $e_1\geqslant e_2$, it follows that $e_1\geqslant 3$. 
Now $G$ is transitive on $\mathcal{U}$. Let $U\in \mathcal{U}(d,q^u, e_1, \bfU)$, so $G_{U}$  stabilises the decomposition  $V=U\perp U^\perp$ and hence $|\mathcal{U}|=|\GU_d(q)|/(|\GU_{e_1}(q)|\cdot |\GU_{e_2}(q)|)$.  Therefore, by the computation of this quantity given in \cite[proof of Proposition 3.3, page 529 ]{PSer}, we have $|\mathcal{U}| = q^{2e_1e_2}\Delta(e_1,d;-q)$, as required.  Also  the bounds on $|\cU|$ follow from the bounds in Table~\ref{t:Del} since $e_1\geq3$.

Now assume that $\bfX\in\{\bfSp, \bfO^\varepsilon\}$. Then the $e_i$ are even (see Table~$\ref{tab:one}$) so also $d$ is even and $\mathcal{U}$ is $\Binom{V}{e_1}$ or $\Binom{V}{e_1}_-^\varepsilon$, with $d$ at least $4, 10$, for $\bfX =\bfSp, \bfO^\varepsilon$, respectively.
For $U\in \mathcal{U}$, again $G_U$ is the stabiliser of the decomposition $V=U\perp U^\perp$, and $|\mathcal{U}|$ is $|\Sp_d(q)|/(|\Sp_{e_1}(q)|\cdot |\Sp_{e_2}(q)|)$ or $|\GO_d^\varepsilon(q)|/(|\GO_{e_1}^-(q)|\cdot |\GO_{e_2}^{-\varepsilon}(q)|)$, respectively. If $\bfX =\bfSp$, then the computation in \cite[proof of Proposition 3.3, p.\,530]{PSer} shows that $|\mathcal{U}| = q^{e_1e_2}\Delta(e_1/2,d/2; q^2)$, as required, and the bounds on $|\cU|$ follow from the bounds in Table~\ref{t:Del} and the fact noted in Lemma~\ref{lem:Delbds} that $\Delta(e_1/2,d/2; q^2)< 8/7$ for $q^2\geq 9$. 
Finally consider the case $\bfX = \bfO^\varepsilon$, so $d\geq10$. Here the proof of \cite[Proposition 3.3, pp. 530--531]{PSer} shows that
\begin{equation}\label{e:U-O}
|\mathcal{U}| = q^{e_1e_2}\cdot \frac{1}{2}\cdot\Delta(e_1/2,d/2; q^2)\cdot \Gamma(\varepsilon)
\end{equation}
with $\Gamma(\varepsilon)$ as in \eqref{d:gam}, and it remains to obtain the bounds on $|\cU|$.

If $\varepsilon = -1$ then, using the bounds in Lemma~\ref{lem:Delbds} and Lemma~\ref{lem:gam}, we have 
\[
\frac{1}{2} \cdot 1 \cdot 1 < \frac{1}{2}\cdot\Delta(e_1/2,d/2; q^2)\cdot \Gamma(-1)\le 
\frac{1}{2}\cdot \frac{1}{1-q^{-2}-q^{-4}}\cdot \left(1+\frac{1}{q}\right)
\]
which yields $\bfa(q,\bfO^-)\cdot q^{e_1e_2}\leqslant |\cU|\leqslant \bfb(q,\bfO^-)\cdot q^{e_1e_2}$ with $\bfa(q,\bfO^-), \bfb(q,\bfO^-)$ as in Table~\ref{t:abLU}. Similarly if $\varepsilon = 1$ then, since $d\geqslant10$, the bounds in Lemmas~\ref{lem:Delbds} and~\ref{lem:gam} give  
\[
\frac{1}{2} \cdot 1\cdot \left(1-\frac{1}{q}- \frac{1}{q^3}\right)  < \frac{1}{2}\cdot\Delta(e_1/2,d/2; q^2)\cdot \Gamma(1)\le 
\frac{1}{2}\cdot \frac{1}{1-q^{-2}-q^{-4}} \cdot 1.
\]
This yields $\bfa(q,\bfO^+)\cdot q^{e_1e_2}\leqslant |\cU|\leqslant \bfb(q,\bfO^+)\cdot q^{e_1e_2}$. These inequalities give $\bfb(q,\bfO^+)$ as in Table~\ref{t:abLU}, and also $\bfa(q,\bfO^+)$ for $q\geq4$. 
In the cases $q=2$ and $q=3$ we use the improved lower bounds on $\Gamma(1)$ from Lemma~\ref{lem:gam}, namely $5/11$ and $40/61$ to obtain the entries $5/22$ and $20/61$, respectively, for $\bfa(q,\bfO^+)$  in Table~\ref{t:abLU}. This completes the proof.
\end{proof}

\section{\texorpdfstring{$\Asch_1:$}{Asch1} Stabilisers of subspaces}\label{s:ac1}

We use the notation and assumptions from Hypothesis~\ref{hyp}.
In this section we find upper bounds for the proportion  ${\rm Prob}_1(g_1,g_2,\bfX)$ as in \eqref{e:asch2} of stingray duos $(g,g')\in g_1^G\times g_2^G$ for which $H=\langle g,g'\rangle$ is reducible on $V$. We call such duos  \emph{reducible stingray duos}, and otherwise they are called  \emph{irreducible stingray duos}. We deal with type $\bfL$ first and then all the other types. 

\subsection{Reducible duos for type \texorpdfstring{$\bfL$}{}}\label{sub:ac1-L}

Here we take $\cM_1(\bfL)$ as the set of all subspace stabilisers, as in \eqref{e:M1L}, 
so for each reducible stingray duo $(g,g')\in g_1^G\times g_2^G$, the subgroup $\langle g,g'\rangle$ is contained in at least one of these stabilisers. The proportion we need to estimate in this case is, by \eqref{e:asch2} with $i=1$, 
 \begin{align*}
  {\rm Prob}_1(g_1,g_2,\bfL) &=    \frac{\mbox{Number of reducible stingray duos in} \ g_1^{G}\times g_2^{G}}{ \mbox{Number of stingray duos in} \ g_1^{G}\times g_2^{G}} \\
  &= 1 - \frac{\mbox{Number of irreducible stingray duos in} \ g_1^{G}\times g_2^{G}}{ \mbox{Number of stingray duos in} \ g_1^{G}\times g_2^{G}}.   
 \end{align*}
Using a graph theoretic model, this proportion has been determined exactly and a usable upper bound  found in \cite{GNP3}. The bound holds for all positive $e_1, e_2$ with $d\geq 4$ (so we state this stronger result, even though we are assuming that $d>8$ for the rest of our analysis). 

\begin{proposition} \label{prop:redL} \cite{GNP3}
Assume that Hypothesis~$\ref{hyp}$ holds with $\bfX=\bfL$, and relax the conditions on the $e_i$ so that $1\leq e_2\leq e_1$ and $d=e_1+e_2\geq 4$.  Then 
\[
{\rm Prob}_1(g_1,g_2,\bfL)
< \frac{1}{q} + \frac{1}{q^2}.
\]
Moreover, we take 
\begin{equation}\label{e:M1L}
\mbox{$\cM_1(\bfL)$ as 
the set of subspace stabilisers, for all proper nontrivial subspaces of $V$.}   
\end{equation}
\end{proposition}

\begin{proof}
The set $\cM_1(\bfL)$ comprises all the possible maximal reducible subgroups that contain elements from $g_1^G$. We now derive the probablity bound.  It follows from \cite[Theorem 5.1]{GNP3} that ${\rm Prob}_1(g_1,g_2,\bfL)$ is equal to a certain quantity $1-P(e_1, e_2)$, and by \cite[Theorem 1.2]{GNP3}, $1-P(e_1, e_2) < \frac{1}{q} + \frac{1}{q^2}$.
\end{proof}

\subsection{Reducible duos for classical types \texorpdfstring{$\bfX\ne\bfL$}{}}\label{sub:ac1-USO}

First we restrict the subspace stabilisers that can contain a reducible stingray duo. (This lemma does not require $d>8$.)

\begin{lemma}\label{lem:dim}
  Let $G, \bfX, e_1, e_2, g_1, g_2, g, g'$ be  as in Hypothesis~$\ref{hyp}$ with $\bfX\ne \bfL$, 
  and suppose that $H=\langle g, g'\rangle$ leaves invariant  a proper non-trivial  subspace $Z$ of $V$. Then $Z\in\{U_{g}, U_{g'} \}$ and $U_{g}=U_{g'}^\perp$, so $H$ leaves both $U_g$ and $U_{g'}$ invariant.
  \end{lemma}

\begin{proof}
     We use the notation of Definition~\ref{def:stingray}(c) for the subspaces $U_g, F_g$ etc.
     By Lemma~\ref{lem:uperp}, $U_g^\perp = F_g$ and $U_{g'}^\perp=F_{g'}$. Since $(g,g')$ is a stingray duo we have  
  $\{0\}=V^\perp=(U_g\oplus  U_{g'})^\perp =U_g^\perp\cap U_{g'}^\perp=F_g\cap F_{g'}$, so 
  $V=F_g\oplus F_{g'}$. Now suppose that $Z$ is a proper non-trivial  $H$-invariant subspace. By Lemma~\ref{lem:unique}(b), either $Z\subseteq F_{g}$ or $U_{g}\subseteq Z$; and also either $Z\subseteq F_{g'}$ or $U_{g'}\subseteq Z$. 
  Since $F_g\cap F_{g'}=0$, $Z$ cannot be contained in both $F_g$ and $F_{g'}$ since $Z\ne 0$. Thus we must have either $U_{g}\subseteq Z$ or $U_{g'}\subseteq Z$. Suppose first that $U_g\subseteq Z$. If also  $U_{g'}\subseteq Z$, then $V=U_g\oplus U_{g'}\leq Z$ which is a contradiction since $Z$ is a proper subspace. Thus $Z\subseteq F_{g'}$.  The inclusion  $U_g\subseteq Z$ implies that $\dim(Z)\geq e_1$, while the inclusion $Z\subseteq F_{g'}$ implies that $\dim(Z)\leq \dim(F_{g'}) = d-e_{2} = e_1$. Thus $\dim(Z)=e_1$ and $Z=U_g=F_{g'}$, and also $U_{g}=U_{g'}^\perp$ since $U_{g'}^\perp=F_{g'}$. An identical argument in the case where $U_{g}\subseteq Z$ shows that $Z=U_{g'}=U_g^\perp$. 
\end{proof}

It follows from Lemma~\ref{lem:dim} that, if $H=\langle g,g'\rangle$ is reducible, then $U_g=U_{g'}^\perp=F_{g'}$ (which is nondegenerate by Lemma~\ref{lem:uperp}) and $H$ is contained in the stabiliser $G_{U_g}=G_{U_{g'}}$. 
If  $\bfX=\mathbf{O^\eps}$, then by Table~\ref{tab:one}, both $U_g$ and $U_{g'}$ have minus type, and since we have $U_{g}=F_{g'}$ and $V=U_{g'} \perp F_{g'}$, it follows that $\eps=+$ (see Line~5 of Table~\ref{tab:one}). In particular, 
$\mathrm{Prop}_1(g_1,g_2,\mathbf{O^-})=0$, and so we assume in the following discussion that in the orthogonal case $\bfX=\bfO^+$. Thus (as we state in Proposition~\ref{prop:redUSO} below), we may take the set $\cM_1(\bfX)$ of maximal reducible subgroups of $\GX_d(q)$ containing $(e_1,e_2)$-stingray duos to be the empty set if  $\bfX = \bfO^-$ and otherwise to consist of the stabilisers of nondegenerate $e_1$-subspaces (of minus type if $\bfX = \bfO^+$).    

Fix a stingray element $g\in g_1^G$ and set $U=U_g$ and $F=F_g$. Then the unique subgroup in $\cM_1(\bfX)$ containing $g$ is the stabiliser $G_U$, and the numerator for  $\mathrm{Prop}_1(g_1,g_2,\bfX)$ in  \eqref{e:asch3} is the number of elements $g'\in g_2^G$ such that $U_{g'}=F$ and $F_{g'}=U$, since all such elements yield reducible stingray duos $(g,g')$ in $G_U$. (Note that, since $F^\perp = U$, the equality $U_{g'}=F$ holds if and only if $(U_{g'}, F_{g'})=(F,U)$.) Thus the numerator of  $\mathrm{Prop}_1(g_1,g_2,\bfX)$ in \eqref{e:asch3} is the number of elements $g'\in g_2^G$ such that $U_{g'}=F$. By Proposition~\ref{p:Xbds}(a), this number is $c(d,q^u, e_2,\bfX)$, so
 \begin{equation}\label{e:redUSO}
    {\rm Prob}_1(g_1,g_2,\bfX)  = 
    \frac{c(d,q^u, e_2,\bfX)}{|N(d,q^u,e_1,\bfX)|}.
\end{equation}

\begin{proposition}\label{prop:redUSO}
Assume Hypothesis~$\ref{hyp}$ with $\bfX\ne \bfL$ and $d>8$.  Then the probability ${\rm Prob}_1(g_1,g_2,\bfX)=0$ if $\bfX= \mathbf{O^-}$, and for the other types
\[
{\rm Prob}_1(g_1,g_2,\bfX)
\leq  \left\{\begin{array}{ll}
     c_\bfU\, q^{-2e_1e_2}  &\mbox{if $\bfX= \bfU$; with $c_\bfU=2.56$ if $q=2$, and $\frac{81}{50}$ if $q\geq3$}  \\
     c_\bfSp\, q^{-e_1e_2} &\mbox{if $\bfX= \bfSp$;  with $c_\bfSp=4$ if $q=2$, and $2$ if $q\geq3$}\\
	 c_\bfO\, q^{-e_1e_2}  &\mbox{if $\bfX= \mathbf{O^+}$;  with $c_\bfO=17.6$ if $q=2$, and $6.1$ if $q\geq3$.}\\	
\end{array}
\right.
\]
Moreover we take $\cM_1(\bfO^-)=\emptyset$, and for the other types,
\begin{equation}\label{e:M1X}
    \mbox{$\cM_1(\bfX)$ is the stabilisers of all nondegen.\ $e_1$-subspaces (of $-$ type if $\bfX=\bfO^+$).}
\end{equation}
\end{proposition}

\begin{proof}
The choice of $\cM_1(\bfX)$ is justified in the discussion preceding the statement.
By Proposition~\ref{p:Xbds}(c),  
$|N(d,q^u,e_1,\bfX)| = \mathbf{k}(d,q^u,e_1,\bfX)\cdot$ $|\mathcal{U}(d,q^u, e_1,\bfX)|\cdot c(d, q^u,e_2,\bf{X})$ with $\mathbf{k}(d,q^u,e_1,\bfX)\geq 1-3/(2q^u)$.   
Thus, by \eqref{e:redUSO} and Proposition~\ref{p:Xbds}(c) and (d),
\[
  {\rm Prob}_1(g_1,g_2,\bfX)  = 
    \frac{1}{\mathbf{k}(d,q^u,e_1,\bfX)\cdot|\mathcal{U}(d,q^u, e_1,\bfX)|}
    \leq  \frac{1}{(1-3/(2q^u))\cdot q^{ue_1e_2}\cdot \bfa(q,\bfX)}
\]
and the bounds follow from the values of $\bfa(q,\bfX)$ in Table~\ref{t:abLU}. For example if $\bfX=\bfU$, then $(1-3/(2q^u))\cdot \bfa(q,\bfX)$ is $25/64$ if $q=2$ and, is at least $50/81$ if $q\geq3$.  
\end{proof}

\section{ \texorpdfstring{$\Asch_2$}{Asch2}: Stabilisers of decompositions}\label{s:ac2}

We use the notation and assumptions from Hypothesis~\ref{hyp}.
In this section we estimate the proportion  ${\rm Prob}_2(g_1,g_2,\bfX)$ of stingray duos $(g,g')\in g_1^G\times g_2^G$ for which $\langle g,g'\rangle$ is irreducible on $V$  but leaves invariant a direct decomposition $V=\bigoplus_{i=1}^b W_i$ with $b\geq2$. Our main result is  Proposition~\ref{prop:imprim}, and ts proof follows from  Lemma~\ref{lem:imprim}.

\begin{proposition}\label{prop:imprim}
Assume that Hypothesis~$\ref{hyp}$ holds.  Then  
\begin{enumerate}[{\rm (a)}]
    \item  ${\rm Prob}_2(g_1,g_2,\bfX)=0$ if any of the following holds:
    \begin{enumerate}[{\rm (i)}]
        \item  $\bfX=\bfSp$ or $(\bfX, q\ \mbox{parity})=(\bfO^\varepsilon, \mbox{even})$; or 
        \item $q^u=2$, or  $r_i >e_i+1$ for some $i$, or $d$ is odd;
    \end{enumerate}
    \item and otherwise
    $q^u>2$,   $r_i=e_i+1\geqslant3$ for each $i$, $d=e_1+e_2$ is even, and 
 \end{enumerate}   
\[
{\rm Prob}_2(g_1,g_2,\bfX) 
<  \left\{\begin{array}{ll}
    0.5\cdot e_1e_2\cdot q^{-2e_1e_2+e_1+e_2}  &\mbox{if $\bfX= \bfL$}\\
    1.46\cdot e_1e_2\cdot q^{-2e_1e_2+e_1+e_2}  &\mbox{if $\bfX= \bfU$}\\
	4.12\cdot e_1e_2\cdot q^{-e_1e_2+(e_1+e_2)/2}	     &\mbox{if $\bfX= \mathbf{O^\varepsilon}$ with $q$ odd.}\\	
\end{array}
\right.
\]
Moreover, we take $\cM_2(\bfX)=\emptyset$ if any of the conditions in part (a) hold, and otherwise
\begin{equation}\label{e:M2}
\cM_2(\bfX)=
    \left\{ G\cap(\GL_{1}(q^u)\wr S_d)\;\middle| 
    \begin{array}{l}
     \mbox{stabilising $\bigoplus_{i=1}^d W_i$ with $\dim(W_i)=1$,}\\
    \mbox{ and with $W_i$ nondegenerate if $\bfX\ne \bfL$}      
    \end{array}
    \right\}.
\end{equation}

 \end{proposition}

To prove Proposition~\ref{prop:imprim}, we work with the following assumptions.

\begin{hypothesis}[for $\Asch_2$]\label{H:AC2}
Assume that Hypothesis~$\ref{hyp}$ holds and ${\rm Prob}_2(g_1,g_2,\bfX) > 0$. Let $(g,g')$ be a stingray duo in $g_1^G\times g_2^G$ such that $H\coloneq\langle g,g'\rangle$  leaves invariant $V=\bigoplus_{i=1}^b W_i$ with $b\geq2$, and $H\not\leq L$ for any $L\in \cM_1(\bfX)$ as in \eqref{e:M1L} or \eqref{e:M1X}.
\end{hypothesis}

These assumptions imply that the group $H$  is irreducible on $V$, and hence $H$ acts transitively on $\{ W_1,\dots, W_b\}$. Thus $\dim(W_i)=d/b$ for each $i$, and so $H$ is contained in a wreath product $\widehat{M}\coloneq\GL_{d/b}(q^u)\wr S_b < \GL_d(q^u)$ preserving the decomposition $V=\bigoplus_{i=1}^b W_i$. Let $M=G\cap \widehat{M}$, and let $B=\GL_{d/b}(q^u)^b$ be the base group and $T\cong S_b$ the top group of $\widehat{M}$, so  $H\leq M \leq \widehat{M}=B\rtimes T$.

Our first result Lemma~\ref{lem:imprim1} proves most of Proposition~\ref{prop:imprim}(a). 

\begin{lemma}\label{lem:imprim1}
Assume that Hypothesis~$\ref{H:AC2}$ holds for $\Asch_2$. Then
\begin{enumerate}[{\rm (a)}]
    \item $b=d=e_1+e_2 > 8$, $r_1=e_1+1$, $q^u\geq3$, and $g, g'\not\in B$;
   
    \item $\bfX\ne \bfSp$, and  either $\bfX= \bfL$, or the $W_i$ are nondegenerate and $\bfX\in\{\bfU, \mathbf{O^\varepsilon}\}$, with $q$  odd if $\bfX=\mathbf{O^\varepsilon}$;

    \item the $W_i$ are all isometric and, without loss of generality, for each $i$, $W_i=\langle w_i\rangle$ where $w_i$ is the $i^{\rm th}$ standard (row) basis vector in $V=\mathbb{F}_{q^u}^d$; the map $w_i\to w_j$ induces an isometry $W_i\to W_j$; and in particular,  $M$ contains the top group $T=S_d$ of $\widehat{M}=\GL_{d}(q^u)\wr S_d$.
\end{enumerate} 
In particular, all parts of Proposition~$\ref{prop:imprim}${\rm(a)}  are established except for the value of $r_2$.
\end{lemma}

\begin{proof}
(a) Suppose that $g\in B$. Then $r_1$ divides $|B|$ and hence $r_1$ divides $|\GL_{d/b}(q^u)|$, whence $e_1\leq d/a$. 
   However, $e_1\geq d/2$ and it follows that $b=2$ and $e_1=e_2=d/2$. 
   By Definition~\ref{def:stingray}(a), the primitive prime divisor $r_2$ of $(q^{u})^{e_2}-1$ satisfies 
  $r_2= k_2e_2+1 \ge e_2+1\geq 3>b$, and hence also $g'\in B$, so $H\leq B$, which is reducible on $V$ and we have a contradiction. 
  
  Thus $g\not\in B$. 
This implies that the prime $r_1=|g|$ divides $|S_b|=b!$. Hence $r_1\leq b$, and since $r_1 = k_1e_1+1\geq e_1+1>d/2$ and $b$ divides $d$, we conclude that $b=d$ and $r_1=e_1+1$. Recall that $d>8$ by Hypothesis~\ref{hyp}. Now $r_2$ does not divide $q^u-1$ since $e_2\geq 2$, and hence also $g'\not\in B$. Finally, if $q^u=2$ then $H\leq \GL_{1}(2)\wr S_d=S_d$ which is reducible since it leaves invariant the diagonal vector $(1,\dots,1)$. Thus $q^u\geq3$.

(b) By \cite[Table 4.2.A]{KL}, the conditions in part (a) imply that, for $\bfX\ne \bfL$, the direct summands $W_i$ are nondegenerate, and so in particular $\bfX\ne \bfSp$, and if $\bfX= \bfO^\varepsilon$ then $q$ is odd (see, for example \cite[Proposition 2.5.1]{KL}).  

(c) The  assertions about the $W_i$ follow from the transitivity of $H$ on $\{W_1,\dots, W_d\}$ discussed above. Finally,  the fact that, without loss of generality, $M$ contains $T$ follow from \cite[Lemma 4.2.1]{KL} and its proof. 
\end{proof}

Next we examine in detail the actions of $g, g'$ on $V=\bigoplus_{i=1}^d W_i$ (see Hypothesis~\ref{hyp}). 
Set $U=U_{g}, F=F_g, U'=U_{g'}$ and $F'=F_{g'}$, so $V=U\oplus F= U'\oplus F'=U\oplus U'$. 

\begin{lemma}\label{lem:imprim2}
Assume that Hypothesis~$\ref{H:AC2}$ holds for $\Asch_2$. Replacing $(g,g')$ by an  $M$-conjugate if necessary, 
$g\in T$, the top group of 
  $\GL_{1}(q^u) \wr S_{d}$, and permutes $W_1,\dots,W_{e_1+1}$ cyclically, 
  fixing pointwise each of $W_{e_1+2},\dots, W_d$;
\begin{enumerate}[{\rm (a)}]
    \item setting $Y_1=\bigoplus_{i=1}^{e_1+1}W_i$, $X_1=\bigoplus_{i=e_1+2}^{e_1+e_2} W_i$, and $D_1=\langle d_1\rangle$ with $d_1=\sum_{i=1}^{e_1+1}w_i$, we have  $Y_1=U\oplus D_1$, $F=D_1\oplus X_1$,
    $V=Y_1\oplus X_1$, and $X_1\ne 0$; 
    \item $g'$ has a unique nontrivial cycle in its induced action on the $W_i$, say $(W_{j_1}, \dots, W_{j_k})$ with $k=r_2=k_2e_2+1$, and setting $Y_2\coloneq \bigoplus_{i=1}^{k}W_{j_i}$, we have $V=(F'\cap Y_1)\oplus Y_2$ and $X_1, U'\leq Y_2$.
\end{enumerate}
\end{lemma}

\begin{proof}
 Since $e_1>1$, $r_1$ does not divide $|B|=(q^u-1)^d$, and since $r_1>d/2$, it follows that a Sylow $r_1$-subgroup of $M$ is a Sylow $r_1$-subgroup of  $\GL_1(q^u)\wr S_d$, and  
 has order $r_1$. Thus $\langle g\rangle \cong Z_{r_1}$ is a
  Sylow $r_1$-subgroup of $M$, and therefore, replacing $(g, g')$
  if necessary by a conjugate in $M$ we may assume that
  $g\in T=S_d$, and permutes $W_1,\dots,W_{e_1+1}$ cyclically, 
  fixing pointwise each of $W_{e_1+2},\dots, W_d$. 
  
  (a)  It follows from Lemma~\ref{lem:unique}(a) that $U$ is contained in the 
  $(e_1+1)$-subspace $Y_1$ (since $Y_1$ is
  $g$-invariant and $g$ acts on $Y_1$ non-trivially). Also the
  $(e_2-1)$-subspace $X_1\coloneq\bigoplus_{i=e_1+2}^dW_i$ is contained in $F$.
  Moreover $F\cap Y_1$ contains the  `diagonal' subspace
  $D_1$ of $Y_1$ and we have $F=D_1\oplus X_1$. Note that $\dim(X_1)=e_2-1\geq1$ since $e_2\geq2$, so $X_1\ne 0$. Finally since $\dim(Y_1)=e_1+1$ it follows that  $Y_1=U\oplus D_1$, and part (a) is proved.
  
(b) Recall from Definition~\ref{def:stingray}(a) that $r_2= k_2e_2+1 \ge e_2+1$ for some integer $k_2\ge1$. 
We consider the action of $g'$ on $V$. 
By Lemma~\ref{lem:imprim1}, $g'\not\in B$, and so $g'$ induces at least one nontrivial cycle on $\{W_1,\dots,W_d\}$. 
  Let $(W_{j_1}, \dots, W_{j_k})$ be a $g'$-cycle with $k\ge 2$.
  Then $g'$ leaves the $k$-subspace $Y_2\coloneq \bigoplus_{i=1}^{k}W_{j_i}$ invariant,
  and $g'$ acts non-trivially on $Y_2$. Again applying Lemma~\ref{lem:unique}(a), we see first that
  $U'\le Y_2$, and also that $g'$ has a unique nontrivial cycle
  on the $W_i$. Thus $\dim(Y_2)=k=r_2=k_2e_2+1$ since $|g'|=r_2$.
 Now the facts that $V=U\oplus U'$, $U\leq Y_1$ (from part (a)), and $U'\leq Y_2$, imply that $V=Y_1+Y_2$.
  As each of $Y_1, Y_2, X_1$ is a direct sum of some of the $W_i$ and $X_1\cap Y_1=0$ (by part (a)), we
  conclude that $X_1\subseteq Y_2$. 
  Hence $F'\cap Y_1$ is the direct
  sum of each $W_i$ with $i\not\in\{j_1,\dots, j_k\}$, and
  $V=(F'\cap Y_1)\oplus Y_2$. 
\end{proof}

Our next tasks are (i)~to show that $r_2=e_2+1$, {(ii)}~to get a better description of $Y_2$ and $F'$, and {(iii)}~to derive some extra information as preparation for estimating ${\rm Prob}_2(g_1,g_2, \bfX)$.

\begin{lemma}\label{lem:imprim3}
Assume that Hypothesis~$\ref{H:AC2}$ holds for $\Asch_2$ and $X_1, Y_1, Y_2$ as in Lemma~$\ref{lem:imprim2}$. Then
\begin{enumerate}[{\rm (a)}]
    \item $r_2=e_2+1$, $Y_1\cap Y_2=W_i \oplus W_{i'}$ for some $i,i'$ with  $i<i'\leq e_1+1$, and $Y_1=(F'\cap Y_1)\oplus W_i \oplus W_{i'}$, and $Y_2=W_i \oplus W_{i'}\oplus X_1$. In particular $d=e_1+e_2$ is even, and $Y_2=  \bigoplus_{\ell=1}^{r_2} W_{j_\ell}$;
    \item $V=(F'\cap Y_1)\oplus Y_2$,  $g'$ fixes $F'\cap Y_1$ pointwise, and $g'|_{Y_2}=(x_{j_1},\dots, x_{j_{r_2}})(j_1, \dots ,j_{r_2})$, where the $x_{j_i}\in\GL_1(q^u)$ and the product $x_{j_1}\dots x_{j_{r_2}}=1$. 
    \item If $\{i,i'\}=\{j_s,j_t\}$ with $s<t$, then $x_{j_s}\dots x_{j_{t-1}}\ne 1$, so in particular $g'\not\in T=S_d$.
\end{enumerate}
Moreover, Proposition~$\ref{prop:imprim}${\rm(a)}  is proved, and for the subgroup $M=G\cap (\GL_1(q^u)\wr S_d)$ containing $H=\langle g, g'\rangle$, we have $g^M=g_1^G\cap M$.
\end{lemma}

\begin{proof}
  By Lemma~\ref{lem:imprim2}(b), $U' \subseteq Y_2$, and $e_2=\dim(U') <\dim(Y_2)=k=r_2=k_2e_2+1$, so $\dim(F'\cap Y_2)=(k_2-1)e_2+1\geq 1$. 
  Let $0\ne y\in F'\cap Y_2$, say
  $y=\sum_{i=1}^\ell y_{j_i}w_{j_i}$ with the $y_{j_i}\in\mathbb{F}_{q^u}$. Since $y\ne0$, we may assume without loss of generality that $y_{j_1}=c\ne 0$. 
  By Lemma~\ref{lem:imprim2}(b), $g'$ fixes pointwise the first summand of $V=(F'\cap Y_1)\oplus Y_2$, and we consider its action on the second summand $Y_2$, namely $g'|_{Y_2}\le \GL_1(q^u)
   \wr S_k$ which has the form $g'|_{Y_2}=(x_{j_1},\dots, x_{j_k})(j_1, \dots ,j_k)$, where the $x_{j_i}\in\GL_1(q^u)$. 
   Thus we have 
   \[
   y = y^{g_2} = y_{j_k}x_{j_k} w_{j_1} + \sum_{i=1}^{k-1} y_{j_i}x_{j_i}
   w_{j_{i+1}},
   \] 
   and this implies that $c=y_{j_1}=y_{j_k}x_{j_k}$,  and for $2\leq i\leq k$,   $y_{j_i}=y_{j_{i-1}}x_{j_{i-1}}$. The last $k-1$ equations imply that $y_{j_i}=cx_{j_1}\cdots x_{j_{i-1}}$ for all $i>1$, so in particular 
   \[
   c=y_{j_1}=y_{j_k}x_{j_k}=cx_{j_1}\cdots x_{j_{k}}.
   \]
   Since $c\ne 0$, we conclude that each of the entries $x_{j_i}$ in $g'|_{Y_2}$ is non-zero and hence that each of the $y_{j_i}\ne 0$; and also $x_{j_1}\cdots x_{j_{k}}=1$. This argument can be applied for each non-zero $y'\in F'\cap Y_2$, and we see that $y'$ has the same form as $y$ with some possibly different non-zero scalar $c'$ in place of $c$. Thus $y'\in\langle y\rangle$, and this implies that $F'\cap Y_2=\langle y\rangle$ of dimension $1$. Hence $1=\dim(F'\cap Y_2)=(k_2-1)e_2+1$, so that $r_2=e_2+1$. Since $V=Y_1\oplus X_1$ (by Lemma~\ref{lem:imprim2}), and all of $Y_1, Y_2, X_1$ are direct sums of some of the $W_j$, it follows that $Y_2\cap Y_1 = W_i \oplus W_{i'}$ for some $i, i'\in\{j_1,\dots,j_{r_2}\}$ with $i<i'\leq e_1+1$, and 
   $Y_2=W_i \oplus W_{i'}\oplus X_1$ and $Y_1=(F'\cap Y_1)\oplus W_i \oplus W_{i'}$. Since each $r_i=e_i+1$ and $e_i\geq 2$, it follows that each $r_i$ is an odd prime and each $e_i$ is even, so $d=e_1+e_2$ is even. This proves parts (a) and (b).
   
   We may write $\{i,i'\}=\{j_s,j_t\}$ with $s<t$, so that in the previous paragraph $c'\coloneq y_{j_s}=cx_{j_1}\dots x_{j_{s-1}}$ and $y_{j_t}=cx_{j_1}\dots x_{j_{t-1}}$. Suppose that $x_{j_s}\dots x_{j_{t-1}}=1$. Then $y_{j_t}=cx_{j_1}\dots x_{j_{s-1}} =y_{j_s}=c'$. Consider the vector $v\in V$ such that $v|_{Y_2}=y$ as above and $v|_{Y_1}=(c',\dots,c')$ (a constant vector, lying in $D_1$ as in Lemma~\ref{lem:imprim2}(a)). Note that, as $Y_2\cap Y_1 = W_{j_s} \oplus W_{j_t}$ and we have $v_i=v_{i'}=c'$, it follows that $v$ is well defined. Then $v$ is fixed by $g$ (since $F=D_1\oplus X_1$) and by $g'$ (as shown above). Hence $H=\langle g,g'\rangle$ leaves $\langle v\rangle$ invariant, which is a contradiction. Thus $x_{j_s}\dots x_{j_{t-1}}\ne 1$, so in particular $g'$ does not lie in the top group $S_d$. Thus part (c) is proved.

  Proposition~\ref{prop:imprim}(a) follows from  Lemma~\ref{lem:imprim1} and part (a) of this result. Finally,  each element of $g_1^G\cap M$ generates a Sylow $r_1$-subgroup of $M$, and hence is conjugate in $M$ to the element $g$.
  \end{proof}

Finally we obtain the bounds for Proposition~\ref{prop:imprim}  for the types  $\bfX=\bfL, \bfU$, and  $\mathbf{O^\eps}$ (with $\varepsilon=\pm$ and $q$ odd). First we obtain upper bounds for ${\rm Prob}_2(g_1,g_2,\bfX)$ which hold for all $d=e_1+e_2$ with $e_1\geq e_2\geq 1$ and the $e_i$ even in the orthogonal case; the condition $d>8$ is used only to obtain the constant coefficients in the final bounds for the unitary and orthogonal cases.

\begin{lemma}\label{lem:imprim}
Suppose that $\bfX\in\{\bfL, \bfU, \mathbf{O^\varepsilon}\}$ with $q$ odd if $\bfX=\mathbf{O^\varepsilon}$, and suppose that $d>8$. Then, setting $e_0=-2e_1e_2+e_1+e_2$, 
 
\[
 {\rm Prob}_2(g_1,g_2,\bfX) \leq  \left\{ \begin{array}{cll}\displaystyle 
 \frac{(q^{e_1}-1)(q^{e_2}-1)e_1e_2}{2q^{2e_1e_2}} &< 0.5\cdot e_1e_2\cdot q^{e_0}& \mbox{if $\bfX=\bfL$}\\[10pt]
 \displaystyle\frac{(q^{e_1}+1)(q^{e_2}+1)e_1 e_2}{2\cdot \mathbf{k}(d,q^2, e_1,\bfU)\cdot 
 |\mathcal{U}(d, q^2, e_1, \bfU)|} &< 1.452\cdot e_1e_2\cdot q^{e_0}& \mbox{if $\bfX=\bfU$}\\[10pt]
 \displaystyle\frac{(q^{e_1/2}+1)(q^{e_2/2}+1)\cdot  e_1 e_2}{2\cdot \mathbf{k}(d, q, e_1,\mathbf{O^\varepsilon})\cdot |\mathcal{U}(d, q, e_1, \mathbf{O^\varepsilon})|}&< 4.12\cdot e_1e_2\cdot q^{e_0/2} & \mbox{if $\bfX=\mathbf{O^\varepsilon}$}
\end{array}\right.
\]
and Proposition~$\ref{prop:imprim}$ is proved.
\end{lemma}

\begin{proof}
Recall that  $G=\GX_d(q)$ where, if $\bfX=\mathbf{O^\varepsilon}$, then $d$ is even, $q$ is odd, and $G=\GO_d^\varepsilon(q)$ with $\varepsilon\in\{ +,-\}$.  
It follows from \eqref{e:M2}, that  $\cM_2(\bfX)$ consists of a single $G$-conjugacy class, and that each $M\in\cM_2(\bfX)$ is self-normalising in $G$ and satisfies $M=\GX_1(q)\wr S_d$,  see \cite[Table  4.2.A]{KL}. If $\bfX=\bfL$ or $\bfU$, then $\GX_1(q)\cong C_{q-\tau}$, where $\tau=1$ if $\bfX=\bfL$, and $\tau=-1$ if $\bfX=\bfU$, while  if $\bfX=\mathbf{O^\varepsilon}$, then $\GX_1(q)\cong C_{2}$, see for example, \cite[Proposition 2.5.5]{KL}. 

We may choose $M\in\cM_2(\bfX)$ stabilising the decomposition $\bigoplus_{j=1}^d W_j$ (with the $W_j$ as in Lemma~\ref{lem:imprim2}) and $g\in g_1^G\cap M$ of order $r_1=e_1+1$ such that $g$ lies in the top group $T$ of $M$ and induces $(1,2,\dots, r_1)$ on the set of subspaces $\{ W_1,\dots,W_{r_1}\}$ and fixes each $W_j$ pointwise with $j>r_1$ (Lemma~\ref{lem:imprim2}). Thus 
\begin{equation}\label{e:cmg}
    |C_M(g)|=|\GX_1(q)|^{e_2}\cdot r_1\cdot (d-r_1)!
    =|\GX_1(q)|^{e_2}\cdot (e_1+1)\cdot (e_2-1)!
\end{equation}
since $C_T(g)=\langle g\rangle\times S_{d-r_1} \cong C_{r_1} \times S_{d-r_1}$ and $g$ centralises precisely a subgroup of order $|\GX_1(q)|^{e_2}$ of the base group of $M$ (which can be seen from Lemma~\ref{lem:imprim2}(a)). 
By Lemma~\ref{lem:imprim3},  $g^M=g_1^G\cap M$, and hence, by Lemma~\ref{lem:countM}, $g$ lies in precisely $|C_G(g_1)|/|C_M(g)|$ conjugates of $M$ in $G$. Thus, by  \eqref{e:asch3} and \eqref{e:ndex},  
\begin{equation}\label{e:imprimL}
    {\rm Prob}_2(g_1,g_2,\bfX) \leq   
\frac{|C_G(g_1)|}{|C_M(g)|}\cdot 
\frac{|\mathcal{S}_2(g,M)|}{|N(d,q^u,e_1,\bfX)|} = \frac{|C_G(g_1)|}{|N(d,q^u,e_1,\bfX)|}\cdot 
\frac{|\mathcal{S}_2(g,M)|}{|C_M(g)|},
\end{equation}
with $\mathcal{S}_2(g,M)$ as in \eqref{e:siM}, namely the set of all $g'\in g_2^G\cap M$ such that $(g,g')$ is a stingray duo and $\langle g,g'\rangle$ is irreducible on $V$. By Lemma~\ref{lem:imprim3}, $r_2=e_2+1$ and for each such $g'$ there is an $r_2$-subset $J=\{j_1,\dots, j_{r_2}\}\subset \{1,\dots,d\}$ such that $g'$ induces an $r_2$-cycle on the $W_j$ in $Y_2\coloneq\bigoplus_{j\in J} W_j$, and fixes each $W_j$ pointwise where $j\not\in J$. Moreover $J=\{i,i'\}\cup\{r_1+1,\dots,d\}$ for distinct $i,i'\in\{1,\dots, r_1\}$, so that $\{i,i'\}=\{j_s,j_t\}$ with $s<t$, and  the restriction  $g'|_{Y_2} = (x_{j_1},\dots, x_{j_{r_2}})(j_1,\dots, j_{r_2})$ for some $x_{j_i}\in\GX_1(q)$ such that  $x_{j_1}\dots x_{j_{r_2}}=1$ and $x_{j_s}\dots x_{j_t}\ne1$. 

This information allows us to estimate $|\mathcal{S}_2(g,M)|$. (We give an over-estimate as we ignore the second restriction on the $x_j$.) There are $\binom{r_1}{2}$ choices of the pair $\{i,i'\}$, and for each pair the number of $r_2$-cycles $(j_1,\dots, j_{r_2})$ is $(r_2-1)!$, and the number of choices of the tuple  $(x_{j_1},\dots, x_{j_{r_2}})$ such that $x_{j_1}\dots x_{j_{r_2}}=1$ is $|\GX_1(q)|^{r_2-1}$. Thus, since $r_1=e_1+1$ and $r_2=e_2+1$, this gives
\[
|\mathcal{S}_2(g,M)| \leq \binom{r_1}{2}\cdot (r_2-1)! \cdot |\GX_1(q)|^{r_2-1} = \frac{e_1(e_1+1)}{2}\cdot (e_2)!\cdot |\GX_1(q)|^{e_2}.
\]
Using \eqref{e:cmg}, this yields
\[
\frac{|\mathcal{S}_2(g,M)|}{|C_M(g)|} \leq \frac{e_1(e_1+1)/2 \cdot (e_2)!\cdot |\GX_1(q)|^{e_2}}{ |\GX_1(q)|^{e_2}\cdot (e_1+1)\cdot (e_2-1)!} = \frac{e_1e_2}{2}.
\]

Recall that $G=\GX_d(q)$, so the value of $|C_{C}(g_1)|/|N(d,q^u, e_1,\bfX)|$ is given by Propositions~\ref{p:Lbds}(c) and~\ref{p:Xbds}(c). Thus, substituting into \eqref{e:imprimL} we obtain the required bounds as follows. First, if $\bfX=\bfL$, then
\[
 {\rm Prob}_2(g_1,g_2,\bfL) \leq  
 \frac{(q^{e_1}-1)(q^{e_2}-1)}{q^{2e_1e_2}}\cdot \frac{e_1 e_2}{2} < q^{-2e_1e_2+e_1+e_2}\cdot \frac{e_1 e_2}{2}
 \]
as in the statement. Next let $\bfX=\bfU$. Then, using $\mathbf{k}(d,q^2, e_1,\bfU)\geq 1-3/(2q^2)\geq 5/8$, and $|\cU(d,q^2,e_1,\bfU)|\geq (5/8)\cdot  q^{2e_1e_2}$ (see Proposition~\ref{p:Xbds}), we have
\[
 {\rm Prob}_2(g_1,g_2,\bfU) \leq  
 \frac{(q^{e_1}+1)(q^{e_2}+1)}{\mathbf{k}(d, q^2, e_1,\bfU)\cdot|\cU(d, q^2,e_1,\bfU)|}\cdot \frac{e_1 e_2}{2}\leq \frac{e_1e_2}{q^{2e_1 e_2-e_1-e_2}}\cdot \frac{(1+q^{-e_1})(1+q^{-e_2})}{(5/8)\cdot(5/8)\cdot 2}.
 \]
 As $d$ is even and $d > 8$, and as each of the $e_i$ is odd and at least $3$, either $e_1\geq e_2\geq5$, or $e_1\geq7, e_2=3$, and hence
\[
(1+q^{-e_1})(1+q^{-e_2})\leq \max\left\{\frac{33}{32}\cdot\frac{33}{32},\, \frac{129}{128}\cdot\frac{9}{8}\right\} = \frac{1161}{1024} < 1.134. 
\]
Thus 
\[
 {\rm Prob}_2(g_1,g_2,\bfU) <  
  e_1e_2\cdot q^{-2e_1 e_2+e_1+e_2} \frac{1.134}{(5/8)\cdot(5/8)\cdot 2} < 1.452 \cdot  e_1e_2\cdot q^{-2e_1 e_2+e_1+e_2}
 \]
as asserted. Finally consider $\bfX=\mathbf{O^\varepsilon}$, so $q$ is odd and the $e_i$ are even. Here we use $\mathbf{k}(d, q, e_1,\mathbf{O^\varepsilon})\geq 1-3/(2q)\geq 1/2$, and $|\cU(d,e_1,\bfU)|\geq (20/61)\cdot  q^{e_1e_2}$  (see Proposition~\ref{p:Xbds}), and we have
\[
{\rm Prob}_2(g_1,g_2,\mathbf{O^\eps})\leq
 \frac{(q^{e_1/2}+1)(q^{e_2/2}+\varepsilon)}{\mathbf{k}(d, q, e_1,\mathbf{O^\varepsilon})\cdot|\cU(d, q, e_1,\mathbf{O^\varepsilon})|}\cdot \frac{e_1 e_2}{2}
 \leq \frac{e_1e_2 (1+q^{-e_1/2})(1+q^{-e_2/2})}{(1/2)\cdot(20/61)\cdot 2\cdot q^{e_1 e_2-(e_1+e_2)/2}}
\]
As $q\geq3$, $d$ is even and $d > 8$, either $e_1\geq e_2\geq6$, or $e_1\geq8, e_2=2$, or $e_1\geq6, e_2=4$,  and hence
\[
(1+q^{-e_1/2})(1+q^{-e_2/2})\leq \max\left\{\frac{28}{27}\cdot\frac{28}{27},\, \frac{82}{81}\cdot\frac{4}{3},\, \frac{28}{27}\cdot\frac{10}{9}\right\} = \frac{328}{243} =1.349... < 1.35 
\]
Thus
we have 
\[
{\rm Prob}_2(g_1,g_2,\mathbf{O^\eps}) < 
\frac{1.35\cdot q^{-e_1e_2+(e_1+e_2)/2}\cdot  e_1 e_2}{(1/2)\cdot (20/61)\cdot 2} < 4.12\cdot e_1e_2\cdot q^{-e_1e_2+(e_1+e_2)/2}
\]
completing the proof of this lemma, and we observe that Proposition~\ref{prop:imprim} follows immediately from this result and Lemma~\ref{lem:imprim3}.
\end{proof}

\section{\texorpdfstring{$\Asch_3:$}{Asch3} Stabilisers of extension fields}

We use the notation and assumptions from Hypothesis~\ref{hyp}.
In this section we estimate the proportion  ${\rm Prob}_3(g_1,g_2,\bfX)$ of stingray duos $(g,g')\in g_1^G\times g_2^G$ for which $\langle g,g'\rangle$ preserves an `extension field structure' on $V$, that is to say, for some prime $b$ dividing $d$, the space $V$ will be identified with $\mathbb{F}_{q^{u b}}^{d/b}$, and for some such identification the group $\langle g,g'\rangle$ is contained in $G\cap (\GL_{d/b}(q^{u b}).b)$.
Our main result is  Proposition~\ref{prop:extn}, and a formal proof is given at the end of the subsection after several preliminary results.

\begin{proposition}\label{prop:extn}
Assume that Hypothesis~$\ref{hyp}$ holds with $d>8$.   Then 
$ {\rm Prob}_3(g_1,g_2,\bfX)=0$ if and only if one of the lines of Table~$\ref{t:Prob3Zero}$ holds,
 and otherwise
 \[
 {\rm Prob}_3(g_1,g_2,\bfX) < c_{\bfX} \cdot q^{-e_1e_2/\alpha},
 \]
 where $c_{\bfX}$ and $\alpha$ are given in Table~$\ref{t:Prob3Res}$.
Moreover we take $\cM_3(\bfX)=\emptyset$ if the conditions of Table~$\ref{t:Prob3Zero}$ hold, and otherwise
 \begin{quote}
     $\cM_3(\bfX)$ is  the union, over primes $b$ dividing $\gcd(e_1, e_2)$, of the $G$-conjugacy classes
     of groups $M$ occurring in Table~$\ref{t:asch3}$.
 \end{quote}
\end{proposition}

{
 \begin{table}
 \caption{Constants for ${\rm Prob}_3(g_1,g_2,\bfX)$ in Proposition~\ref{prop:extn}}
\begin{tabular}{lllr}
  \toprule
$\bfX$ & conditions & $\alpha$ & $c_{\bfX}$ \\
    \midrule
$\bfL$   & &  $1$& $2.05$\\ 
\midrule
$\bfU$   & & $1$& $0.07$ \\ 
\midrule
$\bfSp$   & $e_1\equiv e_2\equiv 0\pmod{4}$   & $2$ & $9.42$\\ 
  & $e_1\equiv e_2 \equiv 2 \pmod{4}$; {$q$ even} & $2$ & 0.22\\ 
  & $e_1\equiv e_2\equiv 2\pmod{4}$; $ q\ge 5$ odd  & $2$ & $3.0$\\ 
  & $e_1\equiv e_2\equiv 2\pmod{4}$; {$q=3$} & $2$ & $4.02$ \\
\midrule
$\mathbf{O^+}$ &  $e_1\equiv e_2\equiv 2\pmod{4}$;  $e_2\ge 6$ & $2$ & $35.64$\\ 
 &  $e_1\equiv e_2\equiv 2\pmod{4}$;   $q=2$ and $e_2=2$ & $2$ & $25.62$\\ 
 &  $e_1\equiv e_2\equiv 2\pmod{4}$;   $q\geq 4$ and $e_2=2$ & $2$ & $8.02$\\ 
 & $e_1\equiv e_2\equiv 0\pmod{4}$ 
 & $2$ & $19.31$ \\ 
\midrule
$\mathbf{O^-}$ &  $e_1\equiv e_2\equiv 0\pmod{4}$ &$2$ & $10.94$ \\ 
&  $e_1\equiv e_2\equiv 2\pmod{4}$&$2$ & $0.26$ \\ 
\bottomrule
\end{tabular}
\label{t:Prob3Res} 
\end{table}
}

 \begin{table}
\caption{Necessary and sufficient conditions for ${\rm Prob}_3(g_1,g_2,\bfX)=0$}
\begin{tabular}{ll}
  \toprule
$\bfX$ & ${\rm Prob}_3(g_1,g_2,\bfX)=0$ if and only if\\
    \midrule
$\bfL$   & $\gcd(e_1,e_2)=1$  \\
$\bfU$   & $\gcd(e_1,e_2)=1$  \\
$\bfSp$   & $\gcd(e_1,e_2)=2$, and if $\frac{e_1}{2}, \frac{e_2}{2}$ are both odd then $q$ is even  \\
$\mathbf{O^\varepsilon}$ & $\gcd(e_1,e_2)=2$, and if $\frac{e_1}{2}, \frac{e_2}{2}$  are both odd then $\varepsilon = -$ \\ 
\bottomrule
\end{tabular}
\label{t:Prob3Zero}
\end{table}

 \begin{table}
\caption{Possibilities for $M$ in $\cM_3(\bfX)$, for  Proposition~\ref{prop:extn}}
\begin{tabular}{clllll}
  \toprule
Line&$\bfX$ & $M\ge$ & Conditions&& Consequences\\
    \midrule
1&$\bfL$   & $\GL_c(q^b).b$ & & &$c>1$ \\
2&$\bfU$   & $\GU_c(q^b).b$ & $b$, $e_1/b, e_2/b$ odd & &{$d\ge 10$; $d$,} $c$ even, $e_2 \ge 3$\\
3&$\bfSp$   & $\Sp_c(q^b).b$ &  $e_1/b, e_2/b$  even & & {$d\ge 12$;} $c\ge4$ even, $e_2\ge 4$\\
4&$\bfSp$   & $\GU_{d/2}(q).2$  & $b=2$; $d/2$ even, and && $d=2c\geq12$, and $q\ge 5$ or  \\ 
& & & $q$, $e_1/2,e_2/2$ odd && $e_2 \ge 6$ with $q= 3$\\ 
5&$\mathbf{O^{\varepsilon}}$  & $\GO^\varepsilon_c(q^b).b$ &  $\varepsilon=\pm$, $e_1/b, e_2/b$  even &&{$d\ge 12$;} $c\ge4$ even, $e_2\ge 4$\\
6&$\mathbf{O^{+}}$   & $\GU_{d/2}(q).2$ & $b=2$, $d/2$ even, and&& {$d=2c\ge 12$; $(e_2,q) \neq (2,3)$ }\\
&&& $e_1/2, e_2/2$  odd && and {$e_2=2$ or $e_2\ge 6$}\\
\bottomrule
\end{tabular}
\label{t:asch3}
\end{table}

To prove Proposition~\ref{prop:extn}, we work with the following assumptions.

\begin{hypothesis}[for $\Asch_3$]\label{H:AC3}
Assume that Hypothesis~$\ref{hyp}$ holds and 
${\rm Prob}_3(g_1,g_2,\bfX)> 0$. Let $(g,g')$ be a stingray duo in $g_1^G\times g_2^G$ such that $H\coloneq\langle g,g'\rangle$  leaves invariant an extension field structure $V=\mathbb{F}_{q^{u b}}^{c}$, where $d=bc$ for some prime $b$ dividing $d$, and $H\leq M\coloneq G\cap (\GL_{c}(q^{u b}).b)$, while $H\not\leq L$ for any $L\in \cM_1(\bfX)\cup\cM_2(\bfX)$  (as in \eqref{e:M1L}, \eqref{e:M1X}, or \eqref{e:M2}). 
\end{hypothesis}
 
 First we derive some restrictions on these parameters. Note that the proof of the following result uses the fact that $e_2\geq2$.
 
 \begin{lemma}\label{lem:extn1} 
Assume that Hypothesis~$\ref{H:AC3}$ holds. Then 
 \begin{enumerate}[{\rm (a)}]
     \item  $b$ divides $\gcd(e_1, e_2)$, and in particular $b<d$;
     \item $g, g'\in G\cap \GL_c(q^{u b})<M$, and moreover, $(g,g')$ is a ppd $(e_1/b, e_2/b)$-stingray duo in $G\cap\GL_c(q^{u b}) < M$;
     \item the group $M$ satisfies one of the lines of Table~$\ref{t:asch3}$, and in each case, $M$ is self-normalising in $G$,  $g_1^G\cap M = g^M$ and $g_2^G\cap M = (g')^M$. 
 \end{enumerate}
 In particular the content of the set $\cM_3(\mathbf{X)}$ given in Proposition~$\ref{prop:extn}$ is justified.  
 \end{lemma}

 \begin{proof}
 (a) Suppose first that $r_1= d$. Since $e_1\geq d/2$ and $r_1=k_1e_1+1$ for some positive integer $k_1$, we conclude that $r_1=e_1+1=d$. However, this is a contradiction since $e_1\leq d-2$. Hence $r_1\ne d$. Suppose next that $b=d$, so $M\leq \GL_1(q^{u d}).d$. Since $r_1\ne d$, it follows that $g\in\GL_1(q^{u d})$, and hence that $r_1\mid (q^{u d}-1)$. Thus $r_1$ divides $(q^u)^{\gcd(e_1,d)}-1$, and as $r_1$ is a primitive prime divisor of $q^{u e_1}-1$, this implies that $e_1\mid d$.  Since $d=b$ is prime, this means that $e_1=d$, which is a contradiction since $e_1\leq d-2$. Thus $b<d$. 
 
In particular, $b\leq d/2\leq e_1<r_1$, and hence $g\in\GL_c(q^{u b})$. 
This implies that the image $\im(g-1)$ (which is an $e_1$-dimensional $\F_{q^u}$-subspace) is also an $\F_{q^{u b}}$-subspace, and hence $b\mid e_1$. Since $d=e_1+e_2$, we also have $b\mid e_2$, and part (a) is proved. 

(b) Now $r_1\geq e_1+1$ and $b\mid e_1$ by part (a), so $r_1>e_1\geq b$. Hence $g\in G\cap \GL_c(q^{u b})<M$. Similarly $b\mid e_2$, so $r_2>e_2\geq b$, and thus $g'\in G\cap\GL_c(q^{u b})$ as $r_2$ is prime. It follows, from the definition of a primitive prime divisor, that $r_i$ is a primitive prime divisor of $(q^{u b})^{e_i/b}-1$, for each $i$. Hence  $(g, g')$ is  a ppd $(e_1/b, e_2/b)$-stingray duo in $G\cap \GL_c(q^{u b})$, proving part~(b).

 (c) Suppose that the isometry group listed in the second column of~\cite[Table 4.3.A]{KL}
 has type ${\bf Y}\in\{{\bf L},{\bf U},{\bf Sp},{\bf O}\}$ acting on a vector space $(\F_{q^{ub}})^{d/b}$. We begin by discussing the columns of Table~\ref{t:asch3} labeled `$\bfX$' and `$M\ge$'. 
 We show below that Line~6 of~\cite[Table 4.3.A]{KL} does not arise in Table~\ref{t:asch3}; and that in Line~7 of~\cite[Table 4.3.A]{KL} the quantity $d/2$ is even, so the condition $\varepsilon=(-)^{d/2}$ in \cite[Table 4.3.A]{KL} becomes $\varepsilon=+$.   Given these facts, it follows from~\cite[Table 4.3.A]{KL} and the containment $I_\sharp\le I$ 
 on~\cite[p.\,111,\;line\,$-3$]{KL}, that $M$ contains $Y\cdot b$ where $Y={\bf GY}_{d/b}(q^{ub})$ is the full isometry group, and $b$ denotes a field automorphism. It turns out that $Y\cdot b=M$, although
 we will only need $Y\cdot b\le M$. This verifies the `$\bfX$' and `$M\ge$' columns of Table~\ref{t:asch3}.
By part (a), $b<d$ so $c>1$, as noted in particular in Line~1 of Table~\ref{t:asch3}.

Next we deal with the parity conditions claimed for the $e_i/b$, and also we show that Line~6 of \cite[Table 4.3.A]{KL} does not arise. 
The parity conditions  are consequences of the facts that $g, g'$ are $(e_i/b)$-stingray elements of $G\cap\GL_c(q^{u b})$, for the appropriate $e_i$ by part (b), and the parity restrictions given by Lemma~\ref{lem:uperp}(b). We argue as follows. 
The facts that $e_1/b, e_2/b$ are both even in Line~3, and both odd in Lines 2 and 4, of Table~\ref{t:asch3}, follow directly from Lemma~\ref{lem:uperp}(b) applied to $M$. In the remaining lines of Table~\ref{t:asch3}, and also in Line 6 of~\cite[Table 4.3.A]{KL}, we have $\bfX=\mathbf{O^\varepsilon}$, so $e_1,e_2$ are even and so $d$ is even, and hence $\varepsilon=\pm$. In Line~5 of Table~\ref{t:asch3}, and also in Line 6 of~\cite[Table 4.3.A]{KL}, $M$ is an orthogonal group and so, by Lemma~\ref{lem:uperp}(b), either $e_i/b=1$ or $e_i/b$ is even. 
 If $e_i/b$ is odd, for some~$i$, then $e_i/b=1$ and so $b=2$ (since $e_i$ is even) and $e_i=2$. This implies that $i=2$ and $e_2=2$, and also that $e_1/2>1$ (since $d>8$) so $e_1/2$ is even and $d/2=e_1/2+e_2/2$ is odd. Thus for Line 5 of  Table~\ref{t:asch3} the $e_i/b$ are both even, while for Line~6 of \cite[Table 4.3.A]{KL} (which does not appear in Table~\ref{t:asch3}) we do have $M$ of type  $\GO^\circ_{d/2}(q^2).2$ with $(d/2)q$  odd. In this case, since $e_2/b=1$,  $g'$ would be a ppd $1$-stingray element of ${\GO^\circ}_{d/2}(q^2)$ of odd prime order $r_2\geq e_2+1$. Therefore $g'$ would fix pointwise a non-degenerate $\F_{q^2}$-subspace of $\F_{q^2}^{d/2}$ of co-dimension $1$, and hence $g'$ would act as a reflection, and so would have order 2, which is a contradiction. Thus Line~6 of \cite[Table 4.3.A]{KL} does not arise.
Finally in  Line 6 of Table~\ref{t:asch3} (which is  Line~7 of \cite[Table 4.3.A]{KL}), $b=2$ and $M$ is a unitary group so $e_1/2, e_2/2$ are both odd by Lemma~\ref{lem:uperp}(b). Thus the parities of the $e_i/b$ are as stated in Table~\ref{t:asch3}.

In all the Lines 2--6 of Table~\ref{t:asch3}, $e_1/b, e_2/b$ have equal parity,  and hence $c=d/b=e_1/b+e_2/b$ and $d$ are even. In particular in Lines 4 and 6, where $b=2$ we have $d/2$ even, so all assertions in the `Conditions' column are valid. 

Finally we verify the assertions in the `Consequences' column.  We already observed that $c>1$ in Line 1, and that $c=d/b$ is even in all other Lines.  For Line 2, we have $e_2\geq b\geq3$, and since $d>8$ and $d$ is even we also have $d\geq10$.  In Lines 3 and 5, each $e_i/b$ is even and hence $c=e_1/b+e_2/b\geq 4$ and $e_1\geq e_2\geq4$; also since $d>8$, either $b\geq3$ or $b=2$ and $c\geq 6$ (as $c$ is even), and in either case $d\geq12$.     Lastly in Lines 4 and~6 we have $b=2$ and each $e_i/b$ is odd; since $c=d/2$ is even and $d>8$, it follows that either $e_1\geq e_2\geq 6$ or $e_2=2$ with $e_1\geq 10$, and in either case $d\geq12$. Moreover, in Line~4, $q$ is odd and if $e_2=2$ then we must have $q\geq5$ since $3^2-1$ has  no primitive prime divisor. Also in Line 6, $(e_2,q)\ne (2,3)$ since  $3^2-1$ has  no primitive prime divisor. Thus all assertions in the `Consequences' column are valid, and hence all entries in Table~\ref{t:asch3} are valid. This proves the first assertion of part (c).

 Recall that $G=\GX_{bc}(q)$. In each line of Table~\ref{t:asch3}, $Y\coloneq G\cap\GL_c(q^{u b})$ is a classical group, which is normal in  $M$ of index $b$, and $M=N_G(M)=N_G(Y)$. By part (b),  $(g,g')$ is a ppd $(e_1/b, e_2/b)$-stingray duo in $Y$. Consider first the element $g$ (similar comments apply to $g'$). By Lemma~\ref{lem:uperp}(c) applied to $Y$, there is a unique cyclic torus $T_Y$ in $Y$ containing $g$ such that $C_Y(g)=C_Y(T_Y)$, and there is a single $Y$-conjugacy class of such tori. There are $|Y:N_Y(T_Y)|$ of these tori, and each of them contains exactly $|g_1^G\cap M|/|Y:N_Y(T_Y)|$ conjugates of $g_1$.  Now applying Lemma~\ref{lem:uperp} to $G$, there is a unique cyclic torus $T$ of $G$ containing $g$ and $C_G(g)=C_G(T)$, with the form given in Lemma~\ref{lem:uperp}(c). It follows that $T_Y$ is contained in $T$. Again applying Lemma~\ref{lem:uperp}(c) to $G$ and $Y$, $g^G\cap T$ has size $e_1$ and $N_G(T)$ acts transitively on this set by conjugation; and $g^Y\cap T_Y$ has size $e_1/b$ and $N_Y(T)$ acts transitively on this set by conjugation. Now $|N_M(T_Y):N_Y(T_Y)|=b$ and $g^M\cap T_Y$ splits into $b$  orbits of $N_Y(T_Y)$, each of size $e_1/b$, which are permuted cyclically by $N_M(T)$. It follows that $T_Y$ contains $e_1$ conjugates of $g_1$  and hence $g_1^G\cap T = g_1^G\cap T_Y$, and $N_M(T_Y)$ acts transitively by conjugation on this set. It follows that $g_1^G\cap M = g^M$. An analogous argument yields $g_2^G\cap M = (g')^M$.

 Finally, part (c) implies that the content of the set $\cM_3(\bfX)$ given in Proposition~\ref{prop:extn} is justified.
 \end{proof}

For a given $\bfX$, and for each prime $b$ dividing $\gcd(e_1, e_2)$, let $\cM_3(\bfX,b)$ denote the set of subgroups in $\cM_3(\bfX)$ occurring in Table~$\ref{t:asch3}$ for this particular prime $b$. Then  

\begin{equation}\label{e:extn}
\cM_3(\bfX) = \bigcup_b \cM_3(\bfX,b)\quad \mbox{and}\quad  {\rm Prob}_3(g_1,g_2,\bfX) \leq  \sum_b  {\rm Prob}_3(g_1,g_2,\bfX;b),
\end{equation}
where $ {\rm Prob}_3(g_1,g_2,\bfX;b)$ is
the contribution to \eqref{e:asch3} from subgroups $M\in\cM_3(\bfX,b)$ containing the element $g$. We show for each $b$ in Lemma~\ref{lem:extn2}(a), that $\cM_3(\bfX,b)$ contains at most one $G$-conjugacy class.

\begin{lemma}\label{lem:extn2}
Assume that Hypothesis~$\ref{H:AC3}$ holds and let $\cM_3(\bfX,b)$ and ${\rm Prob}_3(g_1,g_2,\bfX;b)$ be as in \eqref{e:extn}. Then
 \begin{enumerate}[{\rm (a)}]
     \item for each $\bfX$ and each prime $b$ dividing $\gcd(e_1, e_2)$, either $\cM_3(\bfX,b)=\emptyset$, or $\cM_3(\bfX,b)$ is a single $G$-conjugacy class; and 
     \item for each $\bfX$, $ {\rm Prob}_3(g_1,g_2,\bfX)=0$ if and only if the conditions in Table~$\ref{t:Prob3Zero}$ hold. 
     \item Moreover, suppose that $\bfX\in\{\bfSp, \mathbf{O^\varepsilon} \}$, $b=2$, and $ {\rm Prob}_3(g_1,g_2,\bfX;2)>0$. Then one of the following holds:
     \begin{enumerate}[\rm (i)]
         \item $e_1\equiv e_2\equiv 0\pmod{4}$, and $\cM_3(\bfX,2)$ is the $G$-conjugacy class of subgroups $M$ in Line~$3$ of Table~$\ref{t:asch3}$ if $\bfX=\bfSp$, or in Line~$5$ of Table~$\ref{t:asch3}$ if  $\bfX=\mathbf{O^\varepsilon}$.

         \item $e_1\equiv e_2\equiv 2\pmod{4}$, and $\cM_3(\bfX,2)$ is the $G$-conjugacy class of subgroups $M$ in Line~$4$ of Table~$\ref{t:asch3}$ with $q$ odd if $\bfX=\bfSp$, or in Line~$6$ of Table~$\ref{t:asch3}$ with $\varepsilon =+$ if  $\bfX=\mathbf{O^\varepsilon}$.
         
     \end{enumerate}
 \end{enumerate}
 In particular,  the first assertion in Proposition~$\ref{prop:extn}$ about when $\mathcal{M}_3(\bfX)$ is empty follows from part~{\rm (b)}.
\end{lemma}

\begin{proof}
(a)\quad We note first that, for each line of Table~\ref{t:asch3}, there is a single $G$-conjugacy class of such subgroups $M$ and, if the conditions for this line all hold then the index $b$ normal classical subgroup of $M$ contains  
 ppd $(e_1/2, e_2/2)$-stingray duos in $g_1^G\times g_2^G$, so ${\rm Prob}_3(g_1,g_2,\bfX;b)>0$.
 Conversely, by Lemma~\ref{lem:extn1}, if ${\rm Prob}_3(g_1,g_2,\bfX)>0$ then there exists a prime $b$ such that $b$ divides $\gcd(e_1, e_2)$ and exactly one of the lines of Table~\ref{t:asch3} holds for each such $b$ (note the different parity conditions on the $e_i/b$ when there are two lines for $\bfX=\bfSp$ or $\mathbf{O^\varepsilon}$).
 In particular part (a) follows from these comments. 
 
(b) and (c).\quad   If there exists an odd prime  $b$ dividing $\gcd(e_1, e_2)$, then all the conditions of Line~1, 2, 3, or 5 of Table~\ref{t:asch3} hold (according to the the type $\bfX$), 
 and so, by Lemma~\ref{lem:extn1},   ${\rm Prob}_3(g_1,g_2,\bfX;b)>0$.  
 In particular if $\bfX=\bfU$, then the $e_i$ are odd, and so we conclude from the previous paragraph that ${\rm Prob}_3(g_1,g_2,\bfX)>0$ if and only if $\gcd(e_1, e_2)>1$, proving part (b) for $\bfX=\bfU$. Now consider $\bfX=\bfL$. We have seen that if $\gcd(e_1, e_2)$ is divisible by an odd prime then ${\rm Prob}_3(g_1,g_2,\bfX)>0$. Suppose next that both $e_i$ are even. Then $b=2$ divides $\gcd(e_1, e_2)$, and the group $M=\GL_{d/2}(q^2).2$ in Line~1 of Table~\ref{t:asch3} yields stingray duos proving ${\rm Prob}_3(g_1,g_2,\bfX;2)>0$. These observations together with the previous paragraph show that  ${\rm Prob}_3(g_1,g_2,\bfL)>0$ if and only if $\gcd(e_1, e_2)>1$, proving part (b) for $\bfX=\bfL$. 

 For the rest of the proof we may assume that $\bfX=\bfSp$ or $\mathbf{O^\varepsilon}$ for some $\varepsilon=\pm$, so the $e_i$ are both even. By the previous paragraph,  ${\rm Prob}_3(g_1,g_2,\bfX)>0$ if $\gcd(e_1, e_2)$ has an odd prime divisor. Hence to complete the proof of part (b) we may assume that $\gcd(e_1, e_2)=2^f$ for some $f\geq1$, and to prove part (c) we {must consider} $b=2$.

Suppose first that $e_1\equiv e_2\equiv 0\pmod{4}$. Then the conditions of Line~3 or 5 of Table~\ref{t:asch3} hold, and so, as observed in the first paragraph of the proof,  ${\rm Prob}_3(g_1,g_2,\bfX;2)>0$, and $\cM_3(\bfX,2)$ is the $G$-conjugacy class of subgroups $M$ in Line~$3$ or $5$ of Table~\ref{t:asch3}. 
Next suppose that $e_1\equiv e_2\equiv 2\pmod{4}$, so both $e_1/2$ and $e_2/2$ are odd. For $\bfX=\bfSp$, if $q$ is even then ${\rm Prob}_3(g_1,g_2,\bfSp;2)=0$  by Lemma~\ref{lem:extn1}, while if $q$ is odd then the conditions of Line~4 of Table~\ref{t:asch3} hold, and as observed in the first paragraph above,  ${\rm Prob}_3(g_1,g_2,\bfSp;2)>0$ and $\cM_3(\bfSp,2)$ is the $G$-conjugacy class of subgroups $M$ in Line~$4$ of Table~\ref{t:asch3}. 
For $\bfX=\mathbf{O^\varepsilon}$, if $\varepsilon = -$ then ${\rm Prob}_3(g_1,g_2,\bfX;2)=0$ as none of the lines of Table~\ref{t:asch3} holds, while if $\varepsilon = +$ then the conditions of Line~6 of Table~\ref{t:asch3} hold, and as observed in the first paragraph above,  ${\rm Prob}_3(g_1,g_2,\bfX;2)>0$ and $\cM_3(\bfX,2)$ is the $G$-conjugacy class of subgroups $M$ in Line~$6$ of Table~\ref{t:asch3}. 
The remaining possibility is that 
$e_1/2, e_2/2$ have different parities; 
in this case we have ${\rm Prob}_3(g_1,g_2,\bfX;2)=0$  by Lemma~\ref{lem:extn1}. This completes the proofs of parts (b) and~(c).
\end{proof}

\begin{lemma}\label{lem:extn3}
Assume that Hypothesis~$\ref{H:AC3}$ holds  with $d>8$, and suppose  that,  for some prime $b$ dividing $\gcd(e_1, e_2)$, 
${\rm Prob}_3(g_1,g_2,\bfX;b)>0$   (see \eqref{e:extn}). Then
\[
 {\rm Prob}_3(g_1,g_2,\bfX;b) <  \left\{ \begin{array}{ll}  
 b\cdot  q^{-2e_1e_2(1-1/b)}& \mbox{if $\bfX=\bfL$}\\
 2.56\cdot b\cdot q^{-2e_1e_2(1-1/b)}& \mbox{if $\bfX=\bfU$}\\
 4.6\cdot b\cdot q^{-e_1e_2(1-1/b)} & \mbox{if $\bfX=\bfSp$ with $\frac{e_1}{b}, \frac{e_2}{b}$ even}\\
{1.45}\cdot b\cdot q^{-e_1e_2/2} & \mbox{if $\bfX=\bfSp$ with $\frac{e_1}{2}, \frac{e_2}{2}, q$  odd, $q\geq5$, $b=2$}\\
{2}\cdot b\cdot q^{-e_1e_2/2} & \mbox{{if $\bfX=\bfSp$ with $\frac{e_1}{2}, \frac{e_2}{2}$  odd, $q=3, e_2\geq6$, $b=2$}}\\
5.34\cdot b\cdot q^{-e_1e_2(1-1/b)} & \mbox{if $\bfX=\mathbf{O^-}$ with $\frac{e_1}{b}, \frac{e_2}{b}$ even}\\
9.43\cdot b\cdot q^{-e_1e_2(1-1/b)} & \mbox{if $\bfX=\mathbf{O^+}$ with $\frac{e_1}{b}, \frac{e_2}{b}$ even}\\
17.6\cdot b\cdot  q^{-e_1e_2/2} & \mbox{if $\bfX=\mathbf{O^+}$ with $\frac{e_1}{2}, \frac{e_2}{2}$ odd, $b=2$.}\\
\end{array}\right.
\]
Moreover, if $\bfX=\bfO^+$ with $e_2=2$, $e_1=d-2$ and $b=2$, then $d\ge 12$, $q\ne 3$, and ${\rm Prob}_3(g_1,g_2,\bfX;2) < 12.585\cdot b\cdot q^{-(d-2)}$ if $q=2$, and  ${\rm Prob}_3(g_1,g_2,\bfX;2) < 4.005\cdot b\cdot q^{-(d-2)}$ if $q\geq 4.$
\end{lemma}

\begin{proof}
By \eqref{e:extn}, 
$ {\rm Prob}_3(g_1,g_2,\bfX)\leq \sum_b  {\rm Prob}_3(g_1,g_2,\bfX;b)$, where the sum is over all primes $b$ dividing $\gcd(e_1, e_2)$. Suppose from now on that  $b$ is a prime  dividing $\gcd(e_1, e_2)$ and $ {\rm Prob}_3(g_1,g_2,\bfX;b)>0$. Then by Lemma~\ref{lem:extn2}, $\cM_3(\bfX,b)$ is a single $G$-conjugacy class, and by Lemma~\ref{lem:extn1}, each $M\in\cM_3(\bfX,b)$ is self-normalising in $G$ and the conditions of exactly one Line~$j$ of Table~\ref{t:asch3} hold; and also for each $i=1, 2$, $g_i^G\cap M$ is a single $M$-conjugacy class.

We use the notation of this section; in particular $g\in g_1^G$  lying in a chosen $M\in\cM_3(\bfX,b)$.  Then by Lemma~\ref{lem:countM}, the number of $L\in\cM_3(\bfX,b)$ such that $g\in L$ is $|C_G(g_1)|/|C_M(g)|$. Thus, by \eqref{e:asch3}, 
\begin{equation}\label{e:extn2}
    {\rm Prob}_3(g_1,g_2,\bfX;b) \leq   
\frac{|C_G(g_1)|}{|C_M(g)|}\cdot 
\frac{|\mathcal{S}_3(g,M)|}{|N(d,q^u,e_1,\bfX)|} =
\frac{|C_G(g_1)|}{|N(d,q^u,e_1,\bfX)|}\cdot \frac{|\mathcal{S}_3(g,M)|}{|C_M(g)|},
\end{equation}
with $|N(d,q^u, e_1,\bfX)|$  as in \eqref{e:ndex}, and  $\mathcal{S}_3(g,M)$ as in \eqref{e:siM}, namely  the set of all $g'\in g_2^G\cap M$ such that $(g,g')$ is a stingray duo and $\langle g,g'\rangle$ is irreducible and primitive on $V$. The first factor in the last expression in \eqref{e:extn2} is determined in Proposition~\ref{p:Lbds}(c) for $\bfX\ne\bfL$ or~\ref{p:Xbds}(c) for $\bfX=\bfL$.
We need to estimate the second factor in this expression.

As in the proof of Lemma~\ref{lem:extn1}(c), let $Y$ be the normal classical subgroup of $M$ of index $b$. Applying Lemma~\ref{lem:uperp} to $Y$ and to $G$, there is a unique cyclic torus $T_Y$ in $Y$ containing $g$ such that $C_Y(g)=C_Y(T_Y)$, and  a unique cyclic torus $T$ in $G$ containing $g$ such that $C_G(g)=C_G(T)$. Moreover (see the proof of Lemma~\ref{lem:extn1}(c)), $T_Y\leq T$ and $C_M(g) = C_Y(T_Y)=C_Y(g)$.

To understand the set $\mathcal{S}_3(g,M)$, we note that, by Lemma~\ref{lem:extn1}, $g_2^G\cap M$ is a single $M$-conjugacy class $(g')^M$, and, as discussed in the proof of Lemma~\ref{lem:extn1}(c), $(g')^M$ splits into $b$ conjugacy classes of $Y$ which are permuted cyclically by $N_M(T_Y)$. Thus the number of $h\in g_2^G\cap M$ such that $(g,h)$ is a stingray duo is $b$ times the number of such elements $h\in (g')^Y$. This latter number is 
$|N(d/b,q^{ub}, e_1/b,\bfY)|$, where $Y$ is classical of type $\bfY$. Notice that, in each line of Table~\ref{t:asch3}, $Y$ is the full isometry group of type $\bfY$, and hence the ratio $|C_Y(g)|/|N(d/b,q^{ub}, e_1/b,\bfY)|$ is given by Proposition~\ref{p:Lbds} or~\ref{p:Xbds} applied to $Y$.
Thus \eqref{e:extn2} becomes
\[
    {\rm Prob}_3(g_1,g_2,\bfX;b) \leq   
\frac{|C_G(g_1)|}{|N(d,q^u,e_1,\bfX)|}\cdot \frac{b\cdot|N(d/b,q^{ub}, e_1/b,\bfY)|}{|C_Y(g)|},
\]
and we substitute the values of the two factors given by Proposition~\ref{p:Lbds} or~\ref{p:Xbds} applied to $G$ and $Y$. If $\bfX=\bfL$, this yields 
\[
 {\rm Prob}_3(g_1,g_2,\bfL;b)\leq  \frac{(q^{e_1}-1)(q^{e_2}-1)}{q^{2e_1e_2}}\cdot
 \frac{b\cdot(q^b)^{2(e_1/b)(e_2/b)}}{ ((q^b)^{e_1/b}-1)((q^b)^{e_2/b}-1)} = b\, q^{-2e_1e_2(1-1/b)}
\]
proving the lemma for type $\bfL$.
Next, consider $\bfX=\bfU$. The bound $5/8\leq 1-3/(2q^2)\leq  \mathbf{k}(d,q^2,e_1,\bfU)$ holds as $d>8$, and $\mathbf{k}(d/b, q^{2b}, e_1/b,\bfU) < 1$ holds for all $b\mid e_1$. Therefore ${\rm Prob}_3(g_1,g_2,\bfU;b)$ is at most
 \begin{align*}
 &\ \ \frac{(q^{e_1}+1)(q^{e_2}+1)}{\mathbf{k}(d,q^2,e_1,\bfU)\cdot |\mathcal{U}(d,q^2,e_1,\bfU|}\cdot
 \frac{b\cdot\mathbf{k}(d/b, q^{2b}, e_1/b,\bfU)\cdot |\mathcal{U}(d/b,q^{2b}, e_1/b,\bfU)|}{ ((q^b)^{e_1/b}+1)((q^b)^{e_2/b}+1)}\\ 
 &<
\frac{b\cdot|\mathcal{U}(d/b,q^{2b}, e_1/b,\bfU)|}{ (5/8)\cdot |\mathcal{U}(d,q^2,e_1,\bfU)|}.
 \end{align*} 
Now the upper and lower bounds in Proposition~\ref{p:Xbds} give
\[
 {\rm Prob}_3(g_1,g_2,\bfX;b) <  =
\frac{b\cdot (q^b)^{2(e_1/b)(e_2/b)}}{ (5/8)\cdot(5/8)\cdot q^{2e_1e_2}} 
= \frac{64}{25}\cdot b\cdot q^{-2e_1e_2(1-1/b)}=2.56\cdot b\cdot q^{-2e_1e_2(1-1/b)}
\]
which is the required bound for type $\bfU$.
 Next we consider $\bfX=\bfSp$ or $\mathbf{O^\varepsilon}$ for the cases where both $e_1/b$ and $e_2/b$ are even. Then,  using the bounds $1/4\leq 1-3/(2q)\leq  \mathbf{k}(d,q,e_1,\bfX) < 1$, and setting $\nu=1$ if $\bfX=\bfSp$ and $\nu=\varepsilon$ if $\bfX=\mathbf{O^\varepsilon}$, Proposition~\ref{p:Xbds}(c) shows that ${\rm Prob}_3(g_1,g_2,\bfX;b)$ is at most
 \begin{align*}
 &\phantom{\leq\;\,}\frac{(q^{e_1/2}+1)(q^{e_2/2}+\nu)}{\mathbf{k}(d,q,e_1,\bfX)\cdot |\mathcal{U}(d,q,e_1,\bfX)|}\cdot
 \frac{b\cdot\mathbf{k}(d/b,q^b,e_1/b,\bfX)\cdot |\mathcal{U}(d/b,q^b, e_1/b,\bfX)|}{ ((q^b)^{e_1/2b}+1)((q^b)^{e_2/2b}+\nu)}\\ 
 &<
\frac{b\cdot |\mathcal{U}(d/b,q^b,e_1/b,\bfX)|}{ (1/4)\cdot |\mathcal{U}(d,q,e_1,\bfX)|}.
 \end{align*}
Then the upper and lower bounds in Proposition~\ref{p:Xbds} show that 
\[
 {\rm Prob}_3(g_1,g_2,\bfX;b) < 
\frac{b\cdot \bfb(q^b,\bfX)\cdot (q^b)^{(e_1/b)(e_2/b)}}{ (1/4)\cdot \bfa(q,\bfX)\cdot q^{e_1e_2}} 
= \frac{4\cdot b\cdot \bfb(q^b,\bfX)}{\bfa(q,\bfX)}\cdot q^{-e_1e_2(1-1/b)}.
\]
By Table~\ref{t:abLU}, using $q\geq2$ and $q^b\geq4$, the factor $4\cdot \bfb(q^b,\bfX)/\bfa(q,\bfX)$ is  $4(8/7)/1=32/7< 4.6$ if $\bfX=\bfSp$, and is 
$4(160/239)/(1/2)<5.34$ if $\bfX=\mathbf{O^-}$, and $4(128/239)/(5/22)  < 9.43$ if $\bfX=\mathbf{O^+}$. Thus the required bounds hold in these cases.

For the remaining two cases (Lines 4 and 6 of Table~\ref{t:asch3}),  $\bfX=\bfSp$ or $\mathbf{O^+}$, $b=2$, both $e_1/2$ and $e_2/2$ are odd, and the classical group $Y=\GU_{d/2}(q)$. Further, by Table~\ref{t:asch3}, $d\geq 12$ and $d/2$ is even.  Now,  using the bounds $1-3/(2q)\leq  \mathbf{k}(d,q,e_1,\bfX) < 1$, and then using the upper and lower bounds in Proposition~\ref{p:Xbds}(c)--(d), shows that  
${\rm Prob}_3(g_1,g_2,\bfX;2)$ is at most 
 \begin{align*}
{} &\ \    \frac{(q^{e_1/2}+1)(q^{e_2/2}+1)}{\mathbf{k}(d,q,e_1,\bfX)\cdot |\mathcal{U}(d,q,e_1,\bfX)|}\cdot
 \frac{2\cdot\mathbf{k}(d/2,q^{2}, e_1/2,\bfU)\cdot |\mathcal{U}(d/2,q^{2}, e_1/2,\bfU)|}{ (q^{e_1/2}+1)(q^{e_2/2}+1)}\\ 
 &<
\frac{2\cdot |\mathcal{U}(d/2,q^{2}, e_1/2,\bfU)|}{ (1-3/(2q))\cdot |\mathcal{U}(d,q, e_1,\bfX)|} 
< \frac{2\cdot  q^{2(e_1/2)(e_2/2)}}{ (1-3/(2q))\cdot \bfa(q,\bfX)\cdot q^{e_1e_2}} \\
&= \frac{2}{(1-3/(2q))\cdot \bfa(q,\bfX)}\cdot q^{-e_1e_2/2}.
 \end{align*} 
For $\bfX=\bfSp$,  if $q\geq 5$, then
\[
\frac{2}{(1-3/(2q))\cdot \bfa(q,\bfSp)}\leq \frac{2}{(7/10)\cdot 1} < 2.9, 
\]
and hence ${\rm Prob}_3(g_1,g_2,\bfSp;2) < 2.9 \cdot q^{-e_1e_2/2}$. On the other hand, in Line~4 of Table~\ref{t:asch3}, we may also have $q=3$ with $e_2\geq 6$, and here we have  
\[
\frac{2}{(1-3/(2q))\cdot \bfa(q,\bfSp)}\leq {\frac{2}{(1/2)\cdot 1}} =4 
\]
yielding ${\rm Prob}_3(g_1,g_2,\bfSp;2) < 4 \cdot q^{-e_1e_2/2}$.

Now we consider $\bfX=\mathbf{O^+}$. Here $d\geq12$ and $q\geq 2$, and so 
\[
\frac{2}{(1-3/(2q))\cdot \bfa(q,\mathbf{O^+})}\leq \frac{2}{(1/4)\cdot (5/22)} = 35.2,
\]
and our general bound is ${\rm Prob}_3(g_1,g_2,\mathbf{O^+};2) < 35.2 \cdot q^{-e_1e_2/2}$ holds. 
For the special case where $e_2=2$, $e_1=d-2$ we note that $q\ne 3$ since $3^2-1$ has no primitive prime divisor, and for $q\ne 3$ we use the  expression for $|\mathcal{U}(d,q,d-2,\bfO^+)|$  in Equation~\eqref{e:U-O}, namely
\[
|\mathcal{U}(d,q,d-2,\bfO^+)| = q^{2(d-2)} \cdot \bfa^*(q,\mathcal{O}^+), \ \mbox{where}\   \bfa^*(q,\mathcal{O}^+)= \frac{1}{2}\cdot \Delta\left( \frac{d{-}2}{2},\frac{d}{2}; q^2\right) \cdot \Gamma(1)
\]
and arguing as above we have 
\[
{\rm Prob}_3(g_1,g_2,\bfO^+;2)< \frac{2}{(1-3/(2q))\cdot \bfa^*(q,\mathbf{O^+})}\cdot q^{-(d-2)}.
\]
Next we use the definition of $\Delta(k,d;q)$ in Equation~\eqref{e:OmDel2}, 
and the definition of $\Gamma(1)$  in Equation~\eqref{d:gam}. Since $d/2$ is even, these yield:
\begin{align*}
\bfa^*(q,\bfO^+) &= \frac{1}{2} \cdot \Delta\left( \frac{d{-}2}{2},\frac{d}{2}; q^2\right) \cdot \Gamma(1)
= \frac{1}{2}\cdot
\frac{\omega(\frac{d}{2},\frac{d}{2};q^2)}{\omega(1,1;q^2)} 
 \cdot
\left(
1 - \frac{q^{-(d-2)/2} +q^{-1}}{1+ q^{-d/2}}\right) \\
&= \frac{1}{2}\cdot 
\frac{(1 - q^{-d})}{(1-q^{-2})}
\cdot
\left( \frac{1-(q-1)q^{-d/2} -q^{-1}}{1+ q^{-d/2}}\right)\\
 &= \frac{1}{2}\cdot
\frac{(1 - q^{-d/2})}{(1-q^{-2})}
\cdot 
 (1-(q-1)q^{-d/2} -q^{-1}).
\end{align*}
Therefore,
\begin{align*}
\frac{2}{(1-3/(2q))\cdot \bfa^*(q,\bfX)} &=
\frac{2}{(1-3/(2q))}\cdot \frac{2(1-q^{-2})}
{ (1 - q^{-d/2}) (1-(q-1)q^{-d/2} -q^{-1})}.
\end{align*}
For $q=2$ this expression is less than $25.17$ and for $q\geq 4$ it is less than $8.01$, and we have proved the final assertions of the lemma.
\end{proof}

We need a technical lemma before we prove Proposition~\ref{prop:extn}.

\begin{lemma}\label{lem:c3boundV2}
Let $q,d, e_1, e_2\in\mathbb{Z}$ satisfy $q\ge2$, $d=e_1+e_2$ and $2\le e_2 \le e_1\le d-2$.  Let $B$ be the set of odd prime divisors of $\gcd(e_1,e_2).$ Then the following are true.
\begin{enumerate}[{\rm (a)}]
\item For  $m\geq 3$, the function 
$h(x) = x q^{m/x}$ is strictly decreasing for $x\in [1, \sqrt{m}].$ 
\item The bound $\sum_{b\in B} b\cdot q^{e_1e_2/b}<
 3.01\cdot q^{e_1e_2/3}$ holds.
\item The bound $q^{-e_1e_2}\sum_{b\in B} b\cdot q^{e_1e_2/b}<q^{-e_1e_2/2}\cdot 3.01\cdot q^{-(d-3)/2}$ holds. Moreover, if either $e_2\ne 3$ or $3\nmid e_1$, then $q^{-e_1e_2}\sum_{b\in B} b\cdot q^{e_1e_2/b}<q^{-e_1e_2/2}\cdot 3.01\cdot q^{-(d-6)}$.
 \end{enumerate}
\end{lemma}

\begin{proof}
(a) The derivative of $h(x)$ is $h'(x) = q^{m/x}(1-\frac{m}{x}\log q)$. Since $q\geq2$ and $m\geq3$, we have
 $\sqrt{m}<m\log q$, and it follows that $h'(x)<0$ for all $x\in [1, \sqrt{m}]$. 
 Therefore, $h(x)$ is decreasing for $x\in [1, \sqrt{m}].$

\noindent
(b) Set $m\coloneq e_1e_2\geq4$. This bound holds if $B=\emptyset$ as the  sum is zero, and if $B=\{b\}$ has one element because $b\ge3$ and $bq^{m/b}\leq 3q^{m/3}$ by part~(a). Suppose now that $B=\{b_1,\dots,b_\ell\}$ where $\ell\ge2$ and $3\leq b_1<\cdots<b_\ell$. Then $15\le b_1b_2\le \gcd(e_1,e_2)\le e_2\leq d/2$ and hence $d\geq30$. Since $b_1\cdots b_\ell\mid\gcd(e_1,e_2)\mid e_2$, we have, using part (a), that $\sum_{b\in B} b\cdot q^{m/b}\leq b_1q^{m/b_1}+(\ell-1)b_2q^{m/b_2}\leq 3q^{m/3}+(\ell-1)5q^{m/5}$. Now $b_1\cdot b_2^{\ell-1}\leq b_1\cdots b_\ell\leq e_2$, so 
$5^{\ell-1}\leq b_2^{\ell-1}\leq e_2/b_1\leq e_2/3$, and hence 
$\ell-1\leq\log_5(e_2/3)$. Therefore
 \[
 \sum_{b\in B} b\cdot q^{m/b}\leq q^{m/3}\left(3+\log_5\left(\frac{e_2}{3}\right)\cdot 5\cdot q^{-2e_1e_2/15}\right)\leq q^{m/3}f(e_2,q)
 \]
 where $f(e_2,q)=3+5\log_5(e_2/3)q^{-2e_2^2/15}$. Now $f(e_2,q)$ is clearly a decreasing function of $q$ so $f(e_2,q)\leq f(e_2,2)$ and we will show that $f(e_2,2)$ is a decreasing function of $e_2$. Part  (b) will follow from this since it implies that $f(e_2,q)\leq f(15,2)=3+5\cdot 2^{-30}<3.01$. 
 
 Now we prove that $f(e_2,2)$ is a decreasing function of $e_2$. Since $\log_5(e_2/3)=\ln(e_2)/\ln(5)-\log_5(3)$ where $\ln$ denotes the natural log, it suffices to show that $h(e_2) :=\ln(e_2)\cdot2^{-2e_2^2/15}$ is a decreasing function of $e_2$.  Setting $c=2^{-2/15}$ (a constant), we have $h(e_2)=\ln(e_2)\cdot c^{e_2^2}$, so  $h'(e_2)=c^{e_2^2}\left(2e_2\ln(e_2)\ln(c)+\frac{1}{e_2}\right)$ and the condition $h'(e_2)<0$ is equivalent to $2e_2^2\ln(e_2)\ln(c)+1<0$ and this is true because $2(15)^2\ln(15)\ln(2^{-2/15})+1=-60\ln(15)\ln(2)+1<0$. In summary, $\sum_{b\in B} b\cdot q^{m/b} 
 \leq 3.01\cdot q^{m/3}$ holds as claimed.
 
 (c)~The stated inequality holds when $e_2=2$ because then $B$ is empty and the sum is zero. So suppose that $e_2\geq3$ and set $m\coloneq e_1e_2$.
 Then by part (b), $q^{-m}\sum_{b\in B} b\cdot q^{m/b} 
 \leq q^{-m/2}\cdot 3.01\cdot q^{-m/6}$. However,  $q^{-m/6}=q^{-e_1e_2/6}\le q^{-(d-3)3/6}=q^{-(d-3)/2}$ which proves the first inequality of part~(c). Now suppose that either $e_2\ne3$ or $3\nmid e_1$, and  consider the second bound. Observe that if $e_2\geq6$, then $e_1e_2/6\ge(d-6)6/6=d-6$ and the claimed bound holds by the argument just given, so we may suppose that $e_2\leq 5$. Also if $|B|=0$, then the sum is zero and the bound holds; while if $|B|\ge2$, then we would have $e_2\geq15$ which is not the case.
 Hence $|B|=1$, say $B=\{b\}$, where $b$ is an odd prime dividing $\gcd(e_1,e_2)$. Since $e_2\leq 5$ it follows that $e_2=b\in\{3,5\}$. As $b\mid e_1$, the possibility $e_2=b=3$ does not arise by our assumptions. Hence $e_2=b=5$, and $\sum_{b\in B} bq^{e_1e_2/b}=5q^{e_1e_2/5}$, 
 so that $q^{-e_1e_2}\sum_{b\in B} bq^{e_1e_2/b}=q^{-e_1e_2/2}5q^{-3e_1e_2/10}$.
 Note that $5q^{-3e_1e_2/10}=5q^{-3(d-5)/2}$ and that the inequality $5q^{-3(d-5)/2}\leq 3q^{-d+6}$ (which is sufficient to obtain the second bound) is equivalent to 
 $q^{(d-3)/2}\geq 5/3$. Since $d\geq2e_2=10$ we have $q^{(d-3)/2}\geq 2^{7/2}> 5/3$, and hence the second bound holds and part (c) is proved.    
\end{proof}

\medskip\noindent
\textit{Proof of Proposition~$\ref{prop:extn}$.}\quad Necessary and sufficient conditions for  ${\rm Prob}_3(g_1,g_2,\bfX)=0$ follow from Lemma~\ref{lem:extn2}(b). Assume that ${\rm Prob}_3(g_1,g_2,\bfX)>0$. Then by \eqref{e:extn},
we have ${\rm Prob}_3(g_1,g_2,\bfX)\leq \sum_b {\rm Prob}_3(g_1,g_2,\bfX;b)$, where the sum is over the primes $b$ dividing $\gcd(e_1,e_2)$. An upper bound for each  ${\rm Prob}_3(g_1,g_2,\bfX;b)$ is given in Lemma~\ref{lem:extn3}, and our task is to bound the sum of these quantities for each type $\bfX$. We begin with the type $\bfX= \bfL$, as it is simplest and illustrates the general method. Here, by Lemma~\ref{lem:extn3}
\[
{\rm Prob}_3(g_1,g_2,\bfL)\leq q^{-2e_1e_2}\ \sum_b b\cdot q^{2e_1e_2/b},
\] 
summing over primes $b$ dividing $\gcd(e_1, e_2).$  By separating out the term for $b=2$ and applying Lemma~\ref{lem:c3boundV2}(c) (with $q^2$ in place of $q$) to the sum over the odd primes $b$ for $d\ge 9$, we find  
\[
{\rm Prob}_3(g_1,g_2,\bfL)\leq q^{-2e_1e_2}\ \sum_{b} b\cdot q^{2e_1e_2/b} 
\leq 
\left(2+ 3.01\cdot  
q^{-(d-3)} \right)\cdot q^{-e_1e_2} <
2.05 \cdot q^{-e_1e_2}. 
\]
 The argument for $\bfX= \bfU$ is almost identical except that in this case the $e_i$ are both odd, so $d\geq10$ and all the primes $b$ are odd, and there is an additional constant factor $2.56$ by Lemma~\ref{lem:extn3}.
 By Lemma~\ref{lem:c3boundV2}(c) (with $q^2$ in place of $q$)  and  $d\ge 10$, 
 \[
 {\rm Prob}_3(g_1,g_2,\bfU)\leq 2.56\cdot 
 3.01 
q^{-(d-3)}
\cdot q^{-e_1e_2}
 \leq 0.07\cdot
 q^{-e_1e_2}.
 \]
 The remaining types are $\bfX= \bfSp$ and $\bfX= \mathbf{O^\varepsilon}$, and Lemma~\ref{lem:extn3} provides an upper bound for ${\rm Prob}_3(g_1,g_2,\bfX;b)$ for each type and for each prime $b$, and in some cases two different upper bounds for ${\rm Prob}_3(g_1,g_2,\bfX;b)$ when $b=2$ depending on the (necessarily equal) parities of $e_1/2$ and $e_2/2$, and the field size $q$, 
  see also Table~\ref{t:asch3}. We define constants $c(\bfX)$ such that 
 ${\rm Prob}_3(g_1,g_2,\bfX;b) < c(\bfX) \cdot b \cdot q^{-e_1e_2(1-1/b)}$ for odd primes $b$, and constants $c(\bfX)'$ (depending on $q$ and the parities of $e_1/2, e_2/2$) such that 
 ${\rm Prob}_3(g_1,g_2,\bfX;2)
 < c(\bfX)'  \cdot q^{-e_1e_2/2}$.
 Thus the upper bounds
 for ${\rm Prob}_3(g_1,g_2,\bfX)$ also depend on $q$ and these parities.
Using Lemma~\ref{lem:extn3}, we set $c(\bfSp)=4.6$, and
\[
c(\bfSp)'= \left\{ \begin{array}{ll}  
4.6 & \mbox{if $e_1\equiv e_2\equiv 0\pmod{4}$} \\
1.45  & \mbox{if $e_1\equiv e_2\equiv 2\pmod{4}$, $q$ is odd, and $q\geq5$} \\
0  & \mbox{if $e_1\equiv e_2\equiv 2\pmod{4}$, $q$ is even} \\
2  & \mbox{if $e_1\equiv e_2\equiv 2\pmod{4}$, $q=3$} \\
\end{array}\right.
\]
Likewise we set  $c(\mathbf{O^+})=9.43$ and $c(\mathbf{O^-})=5.34$,
and  
\[
c(\mathbf{O^\varepsilon})'= \left\{ \begin{array}{ll}  
c(\mathbf{O^\varepsilon}) & \mbox{if $e_1\equiv e_2\equiv 0\pmod{4}$} \\
0  & \mbox{if $e_1\equiv e_2\equiv 2\pmod{4}$, and  $\varepsilon=-$} \\
17.6  & \mbox{if $e_1\equiv e_2\equiv 2\pmod{4}$, $\varepsilon=+$, and  $e_2\geq 6$} \\
12.585  & \mbox{if $e_1\equiv 2\pmod{4}$, $e_2=2$, $\varepsilon=+$, and $q=2$} \\
4.005  & \mbox{if $e_1\equiv 2\pmod{4}$, $e_2=2$, $\varepsilon=+$, and $q\geq4$} \\
\end{array}\right.
\]

 Then we have by Equation~\eqref{e:extn} and Lemma~\ref{lem:extn3}
 \[
{\rm Prob}_3(g_1,g_2,\bfX)\leq 
 \sum_{b} {\rm Prob}_3(g_1,g_2,\bfX;b) \le
c(\bfX)'\cdot 2\cdot q^{-e_1e_2/2}\  + c(\bfX)\cdot q^{-e_1e_2}\ \sum_{b>2} b\cdot q^{e_1e_2/b},
\]
where we sum over odd primes $b$ dividing $\gcd(e_1,e_2).$
We apply Lemma~\ref{lem:c3boundV2}(c) 
  noting that  $d\ge 12$ (see Table~\ref{t:asch3}) and that $e_2 \neq 3$ to obtain
  \begin{align*}
{\rm Prob}_3(g_1,g_2,\bfX)&\leq \left(2\cdot c(\bfX)'
+ c(\bfX) \cdot 3.01 \cdot
q^{- (d-6)} \right)
\cdot q^{-e_1e_2/2}.
\end{align*}
Evaluating the right hand side taking $q^{-(d-1)}\leq q_{\rm min}^{-6}$, where $q_{\rm min}$ is the least value for $q$ in each case, yields the constants in Table~\ref{t:Prob3Res}, completing the proof of Proposition~\ref{prop:extn}.

\section{\texorpdfstring{$\Asch_4$}{Asch4} and \texorpdfstring{$\Asch_7:$}{Asch7}  Stabilisers of tensor decompositions}

In this section we consider together the Aschbacher categories  $\Asch_4$ and $\Asch_7$,  
and we prove that the proportions  ${\rm Prob}_4(g_1,g_2,\bfX) = {\rm Prob}_7(g_1,g_2,\bfX) = 0$ whenever $d>4$. Recall that $G=\GX_d(q)$. For these categories, each maximal subgroup is the stabiliser of a tensor decomposition $V=U_1\otimes U_2\otimes\dots\otimes U_t$, where either (i) $t=2$ and $U_1, U_2$ are not isometric, say of dimensions $d_1, d_2\geq2$, and $d=d_1d_2$, or (ii) $t\geq2$ and the $U_i$ are isometric of the same dimension, say $b\geq2$, and $d=b^t$. In case (i) the stabiliser is $G\cap (\GL_{d_1}(q^u)\circ \GL_{d_2}(q^u))$ and is in category $\Asch_4$, while in  case (ii) the stabiliser is $G\cap ((\GL_{b}(q^u)\circ \dots \circ \GL_{b}(q^u))\rtimes S_t)$ and is in category $\Asch_7$. 

\begin{proposition}\label{p:tensor}
Assume that Hypothesis~$\ref{hyp}$ holds.   Then  $\cM_4(\bfX) = \cM_7(\bfX) =\emptyset$ and 
$$ {\rm Prob}_4(g_1,g_2,\bfX)= {\rm Prob}_7(g_1,g_2,\bfX)=0.$$ 
\end{proposition}

\begin{proof}
Let $V=U_1\otimes U_2\otimes\dots\otimes U_t$ as above, and let $d_i=\dim(U_i)$ for each $i$, so that $d=d_1\cdots d_t$ and each $d_i\geq2$. Let 
$(g, g')$ be an $(e_1,e_2)$-stingray duo in  $g_1^G\times g_2^G$  such that $H=\langle g,g'\rangle$ is irreducible and leaves this tensor decomposition invariant.
We claim that $g\in \GL_{d_1}(q^u)\circ \dots\circ \GL_{d_t}(q^u)$. If this were not the case then $d_1=\dots = d_t=b$, say, so $d=b^t\geq 2^t$, and $g$ projects onto a nontrivial element of $S_t$. Since $r_1=|g|$ is prime, it follows that $r_1\leq t$. Now $r_1=ke_1+1$ for some integer $k$, and so we have $t\geq r_1\geq e_1+1\geq d/2+1 \geq 2^{t-1}+1$, and this is impossible for any $t\geq2$. This proves the claim. 

Since $g\in \GL_{d_1}(q^u)\circ \dots\circ \GL_{d_t}(q^u)$, and $r_1=|g|$ is prime and does not divide $q$, it follows that $r_1$ divides $q^{u i}-1$ for some $i\leq \max\{d_1,\dots,d_t\}$. Further, since $r_1$ is a primitive prime divisor of $q^{u e_1}-1$ we conclude that 
\[
d/2\leq e_1\leq i\leq \max\{d_1,\dots,d_t\}\leq d/(\min\{d_1,\dots,d_t\})^{t-1}\leq d/2.
\]
Thus $t=2$, $\min\{d_1,d_2\}=2$
and $d/2=e_1=i=\max\{d_1,d_2\}$, say $d_1=d/2$ and $d_2=2$. Since $d>4$, we have $d_1\geq3$, and since $r_1$ is a primitive prime divisor of $q^{u d_1}-1$, it follows that $r_1$ does not divide $|\GL_2(q^u)|$, and hence $g$ is of the form $g=h\otimes 1 \in \GL_{d/2}(q^u)\circ \GL_{2}(q^u)$. Also, since $e_2=d-e_1=d/2$, the same argument shows that $g'$ is of the form $g'=h'\otimes 1$. Thus $H=\langle g,g'\rangle$ leaves invariant the $d_1$-dimensional subspace $U_1\otimes 1$, which contradicts the assumption that $H$ is irreducible. 

We conclude that no stingray duos $(g,g')$ with these properties exist, and hence the proportions $ {\rm Prob}_4(g_1,g_2,\bfX)$ and ${\rm Prob}_7(g_1,g_2,\bfX)$ are zero. This also means that the families $\cM_4(\bfX)$ and $\cM_7(\bfX)$ are empty for all types $\bfX$.
\end{proof}

\begin{remark}\label{rem:tensor}
The proof of Proposition~\ref{p:tensor} shows that, if a single ppd $e$-stingray element with $e\ge d/2$ leaves invariant a tensor decomposition $V=U\otimes W$ of $V$ with $\dim(U)\geq \dim(V)$, then $e=d/2=\dim(U)$ and $\dim(W)=2$.  
\end{remark}

\section{\texorpdfstring{$\Asch_5:$}{Asch5} Stabilisers of subfields}

We use the notation and assumptions from Hypothesis~\ref{hyp}, and 
in this section we estimate the proportion  ${\rm Prob}_5(g_1,g_2,\bfX)$ of stingray duos $(g,g')\in g_1^G\times g_2^G$ for which $\langle g,g'\rangle$ preserves an `subfield structure' on $V$, which we explain as follows. We require $ua>1$ (where $q=p^a$ and $G=\GX_d(q)\leq \GL_d(q^u)$).
   There is a prime $r$ dividing $ua$ so, in particular, the field $\F_{q^u}$ has a unique maximal subfield $\F_0$ of order $q_0\coloneq q^{u/r}= p^{u a/r}$, and the $\F_0$-span of a basis for $V$ will be an $\F_0$-space $V_0=\mathbb{F}_0^{d}$. The group $\langle g,g'\rangle$ will be contained in the stabiliser $M$ in $G$ of some such  $V_0$ up to scalars and so, as $\GL_d(q^u)$ is transitive on the set of bases of $V$, $M$ will be conjugate in $\GL_d(q^u)$ to a subgroup of the form $G\cap (Z\circ\GL_{d}(q_0))$, where $Z\cong Z_{q^u-1}$ is the (central) subgroup of scalar matrices of $\GL_d(q^u)$. Each group $M$ has a normal subgroup which is a classical group $Y$ depending on both $\bfX$ and $r$. We show in Lemma~\ref{lem:sub1} that the prime $r$ must divide $a$, and that the possibilities for $Y$ are as in Table~\ref{t:sub1}.   
   The subgroups in $\cM_5(G)$ will 
be maximal subject to preserving a `subfield structure' on $V$.
Our main result is  Proposition~\ref{prop:sub}, and a formal proof is given at the end of the subsection after several preliminary results.

\begin{proposition}\label{prop:sub}
Assume that Hypothesis~$\ref{hyp}$ holds,  with $d>8$ and $q=p^a$ for $p$ a prime and $a\geq 1$. Then $ {\rm Prob}_5(g_1,g_2,\bfX)=0$ if one of {\rm(i)} $a=1$, or {\rm(ii)} each prime divisor of $a$ divides $\lcm(e_1, e_2)$, or {\rm(iii)} $\bfX\ne \bfL$ and $a=2$ or $q<8$. Otherwise  
     \[
 {\rm Prob}_5(g_1,g_2,\bfX) <  \left\{ \begin{array}{rl}  
 4.04 \cdot  q^{-e_1e_2 + (d-1)/2}& \mbox{if $\bfX=\bfL$}\\
2.11\cdot q^{-(4/3)e_1e_2+ (2/3)(d-1)}& \mbox{if $\bfX=\bfU$}\\
1.82\cdot q^{-(2/3)e_1e_2 + d/3} & \mbox{if $\bfX=\bfSp$}\\
2.47\cdot q^{-(2/3)e_1e_2 + d/3} & \mbox{if $\bfX=\mathbf{O^+}$}\\
5.46\cdot q^{-(2/3)e_1e_2 + d/3} & \mbox{if $\bfX=\mathbf{O^-}$.}\\
\end{array}\right.
\]
Moreover we take $\cM_5(\bfX)=\emptyset$ if one of {\rm (i)--(iii)} above holds, and  otherwise we take
\begin{quote}
 \textit{$\cM_5(\bfX)$  as the union, over primes $r$ dividing $a$ such that $r$ does not divide $\lcm(e_1, e_2)$, of the $G$-conjugacy classes of groups $M$ with normal classical subgroups $Y$ as in Table~$\ref{t:sub1}$}.
 \end{quote}

\end{proposition}

 \begin{table}
\caption{Possibilities for the classical group $Y$ for $\cM_5(\bfX)$\\ in  Proposition~\ref{prop:sub}, where $q=p^a$ and the prime $r$ divides $a$}
\begin{tabular}{clccl}
  \toprule
Line&$\bfX$ & $Y$ & Conditions & Consequences\\
    \midrule
1&$\bfL$   & $\GL_d(q^{1/r})$ & &  $q \ge p^r\ge p^2$\\
2&$\bfU$   & $\GU_d(q^{1/r})$ & $r$ odd & $q \ge p^r\ge p^3$ \\
3&$\bfSp$   & $\Sp_d(q^{1/r})$ & $r$ odd& $q \ge p^r\ge p^3$\\ 
4&$\mathbf{O^{\varepsilon}}$, $\varepsilon=\pm$   & $\GO^{\varepsilon}_d(q^{1/r})$ & $r$ odd & $q \ge p^r\ge p^3$ \\ 
\bottomrule
\end{tabular}
\label{t:sub1}
\end{table}

To prove Proposition~\ref{prop:sub}, we work with the following assumptions:

\begin{hypothesis}[for $\Asch_5$]\label{H:AC5}
Assume that Hypothesis~$\ref{hyp}$ holds, where $d>8$ and $q=p^a$ for $p$ a prime and some $a\geq2$, and also that ${\rm Prob}_5(g_1,g_2,\bfX)>0$. Let $(g,g')$ be a stingray duo in $g_1^G\times g_2^G$ such that $H\coloneq\langle g,g'\rangle$  leaves invariant a `subfield structure' on $V$. Thus there is a prime $r$ dividing $ua$ and a basis $B$ of $V$ such that $H$ leaves invariant the $\F_0$-span $V_0$ of $B$ up to scalars, where $q_0=|\F_0|=p^{u a/r}=q^{u/r}$. We may identify $B$ with the standard basis of $V$ so that $H\leq M\coloneq  G\cap (Z\circ\GL_{d}(q_0))$, where $Z\cong Z_{q^u-1}$ is the subgroup of scalar matrices of $\GL_d(q^u)$. We also assume that  $H\not\leq L$ for any $L\in \cM_1(\bfX)\cup\dots \cup\cM_4(\bfX)$.
\end{hypothesis}

 First we obtain some restrictions on the parameters, and on the possible maximal subgroups $M$ in $\Asch_5$ containing $H$.
 
 \begin{lemma}\label{lem:sub1} 
Assume that Hypothesis~$\ref{H:AC5}$ holds. Then 
 \begin{enumerate}[{\rm (a)}]
     \item  the prime $r$ divides $a$ but does not divide $\lcm(e_1, e_2)$, 
     and for $i=1,2$, the prime $r_i$ (see Hypothesis~$\ref{hyp}$) 
     is a primitive prime divisor of $q_0^{e_i}-1$. Further, if $\bfX\ne\bfL$ then $a\geq 3$, and in particular $q=p^a\geq 8$;
     \item the group $M$ has a classical normal subgroup $Y$, with $Y$ as in one of the lines of Table~$\ref{t:sub1}$,  and in each case $M=N_G(Y)$ is self-normalising in $G$; 
     \item $g, g'\in Y$, and moreover, $(g,g')$ is a ppd $(e_1, e_2)$-stingray duo in $Y$, $g_1^G\cap M =g^M=g^Y$, and $g_2^G\cap M = (g')^M= (g')^Y$. 
 \end{enumerate}
  In particular, the content of the set $\cM_5(\mathbf{X)}$ given in Proposition~$\ref{prop:sub}$ is justified.
 \end{lemma}

\begin{proof}
It follows from \cite[Table 4.5.A]{KL} that $M$ has a classical normal subgroup $Y$ such that $M=N_G(M)=N_G(Y)$, and either $Y$ satisfies one of Lines 1--3 of Table~$\ref{t:sub1}$ (with $r$ odd if $\bfX=\bfU$) or 
\begin{enumerate}
    \item[(i)]  $\bfX=\mathbf{O^\varepsilon}$ and $Y=\GO^{\varepsilon'}_d(q^{1/r})$ with $\varepsilon = (\varepsilon')^r$; or
    \item[(ii)] $\bfX=\bfU$, $r=2$, and  $Y=\GO^\varepsilon_d(q)$, for some $\varepsilon=\pm$, or $Y=\Sp_d(q)$.
\end{enumerate}
We shall show that, in case (i), the prime $r$ is odd so that $\varepsilon = \varepsilon'$ and Line~4 of Table~\ref{t:sub1} holds; and we shall show that neither of the possibilities in case (ii) arises. From this it follows from Table~\ref{t:sub1} that in each line $r$ divides $a$; if $\bfX=\bfL$ then $q\geq p^r\geq p^2$, while if $\bfX\ne\bfL$, then the prime $r$ is odd so $a\geq r\geq 3$ and $q\geq p^r\geq p^3\geq 8$ (yielding information in the `Consequences' column of Table~\ref{t:sub1}). 

First we consider the elements $g, g'\in M$. Since $e_1\geq e_2\geq 2$, neither of their orders $r_1, r_2$ divides $|Z|=q^u-1$, and hence $g,g'\in Y\leq \GL_d(q_0)$. Consider the element  $g$ (the situation for $g'$ is identical). Since $g$ is an $e_1$-stingray element on $V$, the characteristic polynomial $c_g(t)=\det(t I_d-g)$ for $g$ on $V$ has the form $f(t)(t-1)^{d-e_1}$, with $f(t)$ irreducible over $\F_{q^u}$ of degree $e_1$. Moreover, since $g\in \GL_d(q_0)<\GL_d(q^u)$ where $q_0=q^{u/r}$, the polynomial $c_g(t)$ has coefficients in $\F_0$, and hence also $f(t)$ has coefficients in $\F_0$. Thus $f(t)$ is irreducible over $\F_0$ and $r_1 = |g|$ divides $q_0^{e_1}-1$. If $r$ divides $e_1$ then $q_0^{e_1}-1=(q^{u})^{e_1/r}-1$, and hence $r_1$ divides $(q^{u})^{e_1/r}-1$, which is a contradiction. Hence $r$ does not divide $e_1$. Now let $j$ be the least positive integer such that $r_1$ divides $q_0^j-1$. Then $j\leq e_1$ since we just proved that $r_1$ divides $q_0^{e_1}-1$. Also $j\geq e_1$ since $r_1$ also divides $(q_0^{j})^{r}-1 = q^{u j}-1$ and by definition $r_1$ is a primitive prime divisor of $(q^u)^{e_1}-1$. Hence $j=e_1$. Thus $r_1$ is a primitive prime divisor of $q_0^{e_1}-1$, and it follows that $g$ is a ppd $e_1$-stingray element of $Y$ acting on $V_0$. An identical argument shows that $r$ does not divide $e_2$, that $r_2$ is a primitive prime divisor of $q_0^{ e_2}-1$, and that $g'$ is a ppd $e_2$-stingray element of $Y$. In particular we have shown that $r\nmid\lcm(e_1,e_2)$.

Next we prove our assertions about the possibilities in (i) and (ii) above. 
If $\bfX=\bfU$, then the $e_i$ are odd (see Table~\ref{tab:one}). However if in this case the prime $r=2$ and $Y$ is an orthogonal or symplectic group, then the $e_i$ should be even by Lemma~\ref{lem:uperp}(b) and Table~\ref{tab:one}, which is a contradiction. Hence the exceptional cases in (ii) above do not arise. Moreover, if $\bfX=\bfSp$ or $\mathbf{O^\varepsilon}$, then both $e_1$ and $e_2$ are even (see Table~\ref{tab:one}) and since $r\nmid\lcm(e_1,e_2)$, the prime $r$ must be odd. In case (i) above this means that $\varepsilon = (\varepsilon')^r= \varepsilon'$ and hence that Line~4 of Table~\ref{t:sub1} holds. This justifies all the entries of Table~\ref{t:sub1}, and, as mentioned above, we now have that $r$ divides $a$. 
The fact that in all cases $M=N_G(Y)$ is self-normalising follows from \cite[Proposition 4.5.1]{KL}. 
This completes the proofs of parts (a) and (b), and justifies the content of the set $\cM_5(\bfX)$ given in Proposition~\ref{prop:sub}.

We showed above that $g,g'\in Y$. By Lemma~\ref{lem:uperp}(c) applied to $g$ as an element of $G$, and of $Y$, we see that there are unique cyclic tori $T$ of $G$, and $T_0$ of $Y$ containing $g$ such that $C_G(g)=C_G(T)$ and $C_Y(g)=C_Y(T_0)$. Moreover $T, T_0$ preserve the decompositions $V=U_g\oplus F_g$ and $V_0=U_{0,g}\oplus F_{0,g}$, with $U_g, F_g$ and $U_{0,g}, F_{0,g}$ as in Definition~\ref{def:stingray} for $g$ as an element of $G, Y$ respectively.
It follows that $F_{0,g}=F_g\cap V_0$ and also that $U_{0,g}=U_g\cap V_0$ by the uniqueness property in Lemma~\ref{lem:unique}(c). Then, since $g\in Y$, we have $g\in T_0<T$. The same properties hold for $g'$ as an element of $G$ and $Y$, with subspaces $U_{0,g'}=U_{g'}\cap V_0$ and $F_{0,g'}=F_{g'}\cap V_0$. Since $(g,g')$ is a stingray duo in $G$ we have $U_g\cap U_{g'}=0$, and hence also $U_{0,g}\cap U_{0,g'}=0$, so $(g,g')$ is a stingray duo in $Y$. Since $r_i$ is a primitive prime divisor of $q_0^{e_i}-1$, for each $i$ by part~(a), $(g,g')$ is a ppd $(e_1,e_2)$-stingray duo in $Y$, and  the first assertion in part (c) is proved.

Finally we consider $g_i^G\cap M$. Now $Y$ is normal in $M$ and the index $|M:Y|$ is coprime to $r_i$ (see \cite[Propositions 4.5.3, 4.5.4, 4.5.10]{KL}). It follows that $g_i^G\cap M=g_i^G\cap Y$. By Lemma~\ref{lem:uperp}(c), the torus $T$ in the previous paragraph contains exactly $e_1$ elements of $g_1^G$, and these form a single orbit under the conjugation action of $N_G(T)$. Applying this result to $g$ as an element of $Y$, we see that the torus $T_0$ contains exactly $e_1$ elements of $g^ Y$ and these form a single orbit under the conjugation action of $N_Y(T_0)$. Since $g^Y\subseteq g_1^G$, it follows that  $g_1^G\cap T=g_1^G\cap {T_0}=g^{N_Y(T_0)}$. Finally, since each element of $g_1^G\cap M$ lies in a cyclic torus $T_0'$ which is conjugate in $Y$ to $T_0$, it follows that $g_1^G\cap M$ forms a single $Y$-conjugacy class, namely $g^Y=g^M$. An identical argument shows that $g_2^G\cap M= (g')^M=(g')^Y$. This completes the proof of part~(c).
\end{proof}

It follows from Lemma~\ref{lem:sub1} that, for each $\bfX$, the set $\cM_5(\bfX)$ is partitioned according to the primes $r$ dividing $a$ and not dividing $\lcm(e_1, e_2)$: let $\cM_5(\bfX,r)$ denote the set of subgroups in $\cM_5(\bfX)$ occurring in line $\bfX$ of Table~$\ref{t:sub1}$ for this particular prime $r$. Then

\begin{equation}\label{e:sub}
\cM_5(\bfX) = \bigcup_r \cM_5(\bfX,r)\quad \mbox{and}\quad  {\rm Prob}_5(g_1,g_2,\bfX) \leq  \sum_r  {\rm Prob}_5(g_1,g_2,\bfX;r),
\end{equation}
where $ {\rm Prob}_5(g_1,g_2,\bfX;r)$ is
the contribution to \eqref{e:asch3} from subgroups $M\in\cM_5(\bfX,r)$ containing the element $g$. We show in Lemma~\ref{lem:sub2}(a) that each non-empty $\cM_5(\bfX,r)$ is a single $G$-conjugacy class.

\begin{lemma}\label{lem:sub2}
Assume that Hypothesis~$\ref{H:AC5}$ holds, and let $\cM_5(\bfX,r)$ and ${\rm Prob}_5(g_1,g_2,\bfX;r)$ be as in \eqref{e:sub},  for each $\bfX$ and each prime $r$ dividing $a$ but not dividing $\lcm(e_1, e_2)$. Then 
 \begin{enumerate}[{\rm (a)}]
     \item the set $\cM_5(\bfX,r)$ is a single $G$-conjugacy class; and 
     
     \item the following upper bounds hold
     \[
 {\rm Prob}_5(g_1,g_2,\bfX;r) <  \left\{ \begin{array}{rl}  
4\cdot q^{-(1-1/r)(2e_1e_2-d+1)}& \mbox{if $\bfX=\bfL$}\\
2.08 \cdot q^{-(1-1/r)(2e_1e_2-d+1)}& \mbox{if $\bfX=\bfU$}\\
 1.8\cdot q^{-(2e_1e_2-d)(1-1/r)/2} & \mbox{if $\bfX=\bfSp$}\\
2.44\cdot q^{-(2e_1e_2-d)(1-1/r)/2} & \mbox{if $\bfX=\mathbf{O^+}$}\\
5.4\cdot q^{-(2e_1e_2-d)(1-1/r)/2} & \mbox{if $\bfX=\mathbf{O^-}$.}\\
\end{array}\right.
\]
\end{enumerate}
\end{lemma}

\begin{proof}
(a) By the discussion at the beginning of this section, for each $r$ as in part (a), the group $\GL_d(q^u)$ acts transitively by conjugation on the subgroups $M$ stabilising a subfield structure over the field of order $q^{u/r}$, so if  $\bfX=\bfL$ then the subgroups $M=N_G(Y)$ as in Line~1 of Table~\ref{t:sub1} form a single $\GL_d(q)$-conjugacy class of subgroups of type $\bfL$. For all other types, and for each prime $r$, the fact that the subgroups $M=N_G(Y)$ in the relevant line of Table~\ref{t:sub1} form a single $G$-conjugacy class of subgroups $N_G(Y)$ with $Y$ classical of type $\bfX$ essentially follows from `Witt's Lemma' on the uniqueness of these geometries; we refer to \cite[Propositions 2.3.1, 2.4.1 and 2.5.3]{KL} for the unitary, symplectic and orthogonal types, respectively. This proves part (a).

(b) We estimate the contribution ${\rm Prob}_5(g_1,g_2,\bfX;r)$ to ${\rm Prob}_5(g_1,g_2,\bfX)$ from the subgroups in $\cM_5(\bfX,r)$. Recall that $q_0=q^{1/r}=p^{a/r}$.  By part (a), these subgroups are all $G$-conjugate, and we choose a subgroup $M\in\cM_5(\bfX,r)$ and an element $g\in g_1^G\cap M$. By Lemmas~\ref{lem:countM} and~\ref{lem:sub1}(c), the number of $G$-conjugates $L\in\cM_5(\bfX,r)$ such that $g\in L$ is $|C_G(g_1)|/|C_M(g)|$.
Thus, by \eqref{e:asch3}, 
\begin{equation}\label{e:sub2}
    {\rm Prob}_5(g_1,g_2,\bfX;r) \leq   
\frac{|C_G(g_1)|}{|C_M(g)|}\cdot 
\frac{|\mathcal{S}_5(g,M)|}{|N(d,q^u,e_1,\bfX)|} =
\frac{|C_G(g_1)|}{|N(d,q^u,e_1,\bfX)|}\cdot \frac{|\mathcal{S}_5(g,M)|}{|C_M(g)|},
\end{equation}
with $|N(d,q^u, e_1,\bfX)|$ as in \eqref{e:ndex}, and  $\mathcal{S}_5(g,M)$ as in \eqref{e:siM}, namely the set of all $g'\in g_2^G\cap M$ such that $(g,g')$ is a stingray duo and $\langle g,g'\rangle$ is not contained in a subgroup of $\cM_j(\bfX)$ for any $j\leq 4$. 
The first factor in the last expression in \eqref{e:sub2} is determined in Proposition~\ref{p:Lbds} or~\ref{p:Xbds}.
We need to estimate the second factor in this expression.
Note that, by Lemma~\ref{lem:sub1}(b), $M=N_G(Y)$ with $Y$ a classical subgroup as in Table~\ref{t:sub1}, and so we have $|C_M(g)|=x(\bfX)\cdot |C_Y(g)|$, where $x(\bfX)$ is $(q-1)/(q_0-1)$ or $(q+1)/(q_0+1)$ for  $\bfX=\bfL$ or $\bfU$, and is at most $2$ for the other types. Also, by Lemma~\ref{lem:sub1}(c), if $g'\in g_2^G\cap M$ then $g'\in Y$ and $g_2^G\cap M = (g')^Y$. Observe that the type $\bfY$ of $Y$ is $\bfX$ by Table~\ref{t:sub1}. Thus
 \[
 |\mathcal{S}_5(g,M)|\leq  \#\{  h\in (g')^Y \mid \mbox{$(g,h)$ is a stingray duo on $V_0$} \} = |N(d, q_0, e_1, \bfX)|, 
 \]
with $N(d, q_0, e_1, \bfX)$ as in \eqref{e:ndex}. Now the conditions of Proposition~\ref{p:Lbds} or~\ref{p:Xbds} hold with $Y, q_0, g, g'$ replacing $G, q, g_1, g_2$, yielding an expression for $|C_Y(g)|/|N(d, q_0, e_1, \bfX)|$. Putting this information together we have
 \begin{equation}\label{e:prob5X}
     {\rm Prob}_5(g_1,g_2,\bfX;r) \leq  \frac{|C_G(g_1)|}{|N(d,q^u,e_1,\bfX)|}\cdot \frac{|N(d, q_0, e_1, \bfX)|}{x(\bfX)\cdot |C_Y(g)|},
 \end{equation}
and we now apply Proposition~\ref{p:Lbds} or~\ref{p:Xbds} 
separately for each type. 
For $\bfX=\bfL$, we have 
\[
 {\rm Prob}_5(g_1,g_2,\bfL;r)\leq  \frac{(q^{e_1}-1)(q^{e_2}-1)}{q^{2e_1e_2}}\cdot
 \frac{q_0^{2e_1e_2}}{(q-1)/(q_0-1)\cdot (q_0^{e_1}-1)(q_0^{e_2}-1)}.
\]
Now $(q^{e_1}-1)/(q_0^{e_1}-1) = 1 + q_0^{e_1} + \dots + q_0^{e_1(r-1)} < 2q_0^{e_1(r-1)} = 2q^{e_1(1-1/r)}$,  and similarly $(q^{e_2}-1)/(q_0^{e_2}-1)< 2 q^{e_2(1-1/r)}$, while $(q-1)/(q_0-1)>q^{1-1/r}$. Hence 
\[
{\rm Prob}_5(g_1,g_2,\bfL;r)< 4 q^{-(1-1/r)(2e_1e_2-e_1-e_2+1)}.
\]

For $\bfX=\bfU$, we have $r$ odd, and so $q^2\geq 2^6$ as $q\geq8$ by Lemma~\ref{lem:sub1}(a). Thus, using the bounds $\frac{125}{128}\leq 1-3/(2q^2)\leq  \mathbf{k}(d,q^u,e_1,\bfU)<1$, we obtain
 \begin{align*}
 {\rm Prob}_5(g_1,g_2,\bfU;r) &\leq 
       \frac{(q^{e_1}+1)(q^{e_2}+1)}{\mathbf{k}(d,q^2,e_1,\bfU)\cdot |\mathcal{U}(d,q^2,e_1,\bfU|}\cdot
 \frac{\mathbf{k}(d,q_0^2,e_1,\bfU)\cdot |\mathcal{U}(d,q_0^2,e_1,\bfU)|}{(q+1)/(q_0+1)\cdot  (q_0^{e_1}+1)(q_0^{e_2}+1)}\\ 
 & <
\frac{|\mathcal{U}(d,q_0^2,e_1,\bfU)|}{ (125/128)\cdot |\mathcal{U}(d,q^2,e_1,\bfU)|}\cdot   \frac{(q^{e_1}+1)(q^{e_2}+1)}{(q+1)/(q_0+1)\cdot  (q_0^{e_1}+1)(q_0^{e_2}+1)}\\
 & <
\frac{{\bf b}(q_0^2,\bfU)\, q_0^{2e_1e_2}}{ (125/128)\cdot {\bf a}(q^2,\bfU)\, q^{2e_1e_2}}\cdot   \frac{(q^{e_1}+1)(q^{e_2}+1)}{(q+1)/(q_0+1)\cdot  (q_0^{e_1}+1)(q_0^{e_2}+1)}.
 \end{align*} 
Next, we note since $r$ is odd  that, for $e\in\{1, e_1,e_2\}$,
\begin{equation}\label{E:e}
\frac{q^{e}+1}{q_0^{e}+1}
=q^{e(1-1/r)} \frac{1+q_0^{-er}}{1+q_0^{-e}} 
< q^{e(1-1/r)}.
 \end{equation}
 Also, taking $e=1$, we see that 
 \[
\frac{q+1}{q_0+1}> q^{1-1/r} \frac{1}{1+q_0^{-1}} \ge \frac23 q^{1-1/r}.
\]
Thus the second term in the upper bound for ${\rm Prob}_5(g_1,g_2,\bfU;r)$ is less than
\[
\frac{q^{e_1(1-1/r)}\cdot q^{e_2(1-1/r)}}{(2/3) q^{1-1/r}}
=\frac32 q^{(1-1/r)(e_1+e_2-1)}.
\]
Then,  using the upper and lower bounds in Proposition~\ref{p:Xbds} with $q^2>q_0^2\geq 4$, we have
\begin{align*}
 {\rm Prob}_5(g_1,g_2,\bfU;r) &< 
\frac{(q^{1/r})^{2e_1e_2}}{ (125/128)\cdot (20/27)\cdot q^{2e_1e_2}}\cdot (3/2)\cdot  q^{(1-1/r)(e_1+e_2-1)}\\
&< 2.08\cdot q^{-(1-1/r)(2e_1e_2-e_1-e_2+1)}.
\end{align*}
In the last two cases $\bfX=\bfSp$ or $\mathbf{O^\varepsilon}$, $r$ is odd, and so $q=q_0^r\ge 8$. Thus, using the bounds   
$13/16\leq 1-3/(2q)\leq  \mathbf{k}(d,q,e_1,\bfX) < 1$ and $1\leq x(\bfX)
\leq 2$, and setting $\nu=1$ if $\bfX=\bfSp$ and $\nu=\varepsilon\cdot 1$ if $\bfX=\mathbf{O^\varepsilon}$ we obtain, in \eqref{e:prob5X},
 \begin{align*}
 {\rm Prob}_5(g_1,g_2,\bfX;r) &\leq 
       \frac{(q^{e_1/2}+1)(q^{e_2/2}+\nu)}{\mathbf{k}(d,q,e_1,\bfX)\cdot |\mathcal{U}(d,q,e_1,\bfX)|}\cdot
 \frac{\mathbf{k}(d, q_0,e_1,\bfX)\cdot |\mathcal{U}(d,q_0,e_1,\bfX)|}{x(\bfX)\cdot (q_0^{e_1/2}+1)(q_0^{e_2/2}+\nu)}\\ 
 &< 
\frac{|\mathcal{U}(d,q_0,e_1,\bfX)|}{ (13/16)\cdot |\mathcal{U}(d,q,e_1,\bfX)|} 
\cdot \frac{(q^{e_1/2}+1)(q^{e_2/2}+\nu)}{(q_0^{e_1/2}+1)(q_0^{e_2/2}+\nu)}.
 \end{align*} 
Replacing $q, q_0$ in~\eqref{E:e} with $q^{1/2}, q_0^{1/2}$ gives $(q^{e/2}+1)/(q_0^{e/2}+1)<q^{(e/2)(1-1/r)}$, for $e\in\{e_1,e_2\}$, 
while  $(q^{e_2/2}-1)/(q_0^{e_2/2}-1)<2 q^{(e_2/2)(1-1/r)}$. Then, setting $y(\bfX)=1$ if $\bfX=\bfSp$ or $\mathbf{O^+}$ (the cases where $\nu=1$) and $y(\bfX)=2$ if $\bfX=\mathbf{O^-}$ (the case where $\nu=-1$),  the second term in the above expression is less than $y(\bfX)\cdot q^{(e_i+e_2)(1-1/r)/2}$. For the first term in the above expression we use the 
inequalities 
in Proposition~\ref{p:Xbds}, and find
\begin{align*}
 {\rm Prob}_5(g_1,g_2,\bfX;r) 
 &< \frac{ \bfb(q_0,\bfX)\cdot q_0^{e_1e_2}\cdot y(\bfX)\cdot q^{(e_1+e_2)(1-1/r)/2}}{ (13/16)\cdot \bfa(q,\bfX)\cdot q^{e_1e_2}} \\
&= \frac{(16/13)\cdot y(\bfX)\cdot \bfb(q_0,\bfX)}{\bfa(q,\bfX)}\cdot q^{-(2e_1e_2-e_1-e_2)(1-1/r)/2}
\end{align*}
and from Table~\ref{t:abLU} with $q\geq q_0^2\geq 4$, we have $(16/13)\cdot y(\bfX)\cdot \bfb(q_0,\bfX)/\bfa(q,\bfX)$ is $\frac{16}{13}\cdot\frac{16}{11}/1 = 256/143< 1.8$ if $\bfX=\bfSp$, and is $\frac{16}{13}\cdot\frac{8}{11}/\frac{47}{128} < 2.44$ if $\bfX=\mathbf{O^+}$, and is $\frac{16}{13}\cdot 2\cdot \frac{12}{11}/\frac{1}{2}<5.4$ if $\bfX=\mathbf{O^-}$. 
\end{proof}

Finally we prove Proposition~\ref{prop:sub}.

\medskip\noindent
\begin{proof}[Proof of Proposition~$\ref{prop:sub}$.]
Recall that $q=p^a$. By Lemma~\ref{lem:sub1},  ${\rm Prob}_5(g_1,g_2,\bfX)=0$ if $a=1$, or if each prime divisor of $a$ divides $\lcm(e_1,e_2)$, or if $\bfX\ne \bfL$ and $a=2$ or $q<8$; and in general we have ${\rm Prob}_5(g_1,g_2,\bfX)\leq \sum_r {\rm Prob}_5(g_1,g_2,\bfX;r)$ as in \eqref{e:sub}, where the sum is over all primes $r$ dividing $a$ such that $r$ does not divide $\lcm(e_1,e_2)$. An upper bound for ${\rm Prob}_5(g_1,g_2,\bfX;r)$ is given by Lemma~\ref{lem:sub2}, and in all cases the upper bound  can be expressed as $Y_\bfX Q_\bfX^{a/r}$, with both $Y_\bfX, Q_\bfX$ positive quantities depending on $p$ and $a$, but not~$r$. In particular we take $Q_\bfL= Q_\bfU= p^{2e_1e_2-d+1}$ and $Q_\bfSp= Q_{\mathbf{O}^\varepsilon}= p^{e_1e_2-d/2}$. Then, for  all $\bfX$, we have, since $d\geq 9$, 
$$Q_\bfX\ge p^{e_1e_2-d/2}\geq 2^{(d-2)2-d/2}=2^{3d/2-4}>2^9,\ \mbox{and}\ \sum_{i=0}^\infty Q_\bfX^{-i}=1/(1-Q_\bfX^{-1})<1.01.
$$
Consider first $\bfX=\bfL$. Here by Lemma~\ref{lem:sub2}, $Y_\bfL= 4 Q_\bfL^{-a}$. Let $r_0$ be the smallest prime dividing $a$, so $a/r\leq a/2$, and we have 
\begin{align*}
{\rm Prob}_5(g_1,g_2,\bfL) &< Y_\bfL \sum_{r|a} Q_\bfL^{a/r}
\leq  Y_\bfL  Q_\bfL^{a/r_0} \sum_{i=0}^{a/r_0} Q_\bfL^{-i}\\
&< Y_\bfL  Q_\bfL^{a/2} \sum_{i=0}^{\infty} Q_\bfL^{-i}  < 1.01 \cdot Y_\bfL  Q_\bfL^{a/2} = 4.04\cdot Q_\bfL^{-a/2}.
\end{align*}
This gives the upper bound ${\rm Prob}_5(g_1,g_2,\bfL) < 4.04\cdot q^{-e_1e_2+(d-1)/2}$ in Proposition~\ref{prop:sub}. For all other types $\bfX$, $r$ is odd by Table~\ref{t:sub1}, so $a/r\leq a/3$, and hence 
\begin{equation}\label{e:prob5x2}
{\rm Prob}_5(g_1,g_2,\bfX) < Y_\bfX \sum_{r|a} Q_\bfX^{a/r} < Y_\bfX  Q_\bfX^{a/3}\sum_{i=0}^{\infty} Q_\bfX^{-i} < 1.01\cdot Y_\bfX  Q_\bfX^{a/3}.    
\end{equation}
If $\bfX=\bfU$ then, by Lemma~\ref{lem:sub2},    $Q_\bfU=Q_\bfL$ and $Y_\bfU= 2.08\cdot Q_\bfL^{-a}$. Hence, by \eqref{e:prob5x2}, ${\rm Prob}_5(g_1,g_2,\bfU) 
<  2.11\cdot Q_\bfL^{-2a/3}=2.11\cdot q^{-(4/3)e_1e_2+ (2/3)(d-1)}$, as in  Proposition~\ref{prop:sub}.
Finally, for $\bfX=\bfSp$, $\bfX=\mathbf{O^+}$,  or $\bfX=\mathbf{O^-}$, we have, by Lemma~\ref{lem:sub2},  $Q_\bfX = p^{e_1e_2-d/2}$ and $Y_\bfX= c_\bfX Q_\bfX^{-a}$, where $c_\bfX = 1.8, 2.44$, or $5.4$, respectively. Hence, by \eqref{e:prob5x2}, ${\rm Prob}_5(g_1,g_2,\bfX) 
<  1.01\cdot c_\bfX\cdot Q_\bfX^{-2a/3}=1.01\cdot c_\bfX\cdot  q^{-(2/3)e_1e_2+ d/3}$, and the bounds in  Proposition~\ref{prop:sub} follow.
\end{proof}

\section{\texorpdfstring{$\Asch_6:$}{Asch6} normalisers of symplectic type \texorpdfstring{$s$}{s}-groups}

We use the notation and assumptions from Hypothesis~\ref{hyp}, and deal with the Aschbacher category $\Asch_6$, where the  dimension $d = s^n$ for some prime divisor $s$ of $q^u -1$ and $n\geq1$. 
In this section we estimate the proportion  ${\rm Prob}_6(g_1,g_2,\bfX)$ of stingray duos $(g,g')\in g_1^G\times g_2^G$ for which $\langle g,g'\rangle$ lies in the normaliser of an extraspecial group of order $s^{1+2n}$, or a symplectic type group of order $2^{2+2n}$ with $s=2$, and exponent $s\cdot\gcd(2,s)$ (see \cite[Chapter 4.6]{KL} or  \cite[Definition 2.2.13]{BHRD}). In particular, we shall see that for a non-zero contribution to $ {\rm Prob}_6(g_1,g_2,\bfX)$, the prime $s$ must be $2$.  Our main result is Proposition~\ref{prop:c6}, and a formal proof is given after proving a crucial technical Lemma~\ref{lem:c6-1}.

\begin{remark}
 {\rm
 Our approach to analysing this case has been influenced by the approach in \cite[Section 10]{PSY}. However, in working through our more general estimation problem for $\Asch_6$ groups, we discovered an error in the proof of \cite[Lemma 10.1]{PSY}. A valid upper bound for the proportion in this case in the situation in \cite[Section 10]{PSY} is given by our Proposition~\ref{prop:c6}(a) (since $s=2$ and $e_1=e_2$ in \cite{PSY}). This error does not affect the main result of \cite{PSY}. (The problem in \cite[Lemma 10.1]{PSY}, where  $\langle g,g'\rangle$ normalises a 2-group $R$ of type $2^{1+2n}$ or $4\circ 2^{1+2n}$, is that the proof implicitly assumed that the stingray elements left invariant each nontrivial $R$-orbit in $V$, and this was not justified.)  
 }   
\end{remark}

\begin{proposition}\label{prop:c6}
Assume that Hypothesis~$\ref{hyp}$ holds. Then 
\begin{enumerate}[{\rm (a)}]
    \item  $ {\rm Prob}_6(g_1,g_2,\bfX)=0$ if any one of the following holds:
    \begin{enumerate}[\rm (i)]
        \item $d\ne 2^n$ for some $n\geq4$,  or $q$ is not an odd prime;
        \item $(e_1,e_2)\ne (d/2,d/2)$ or $(d-2,2)$;
        \item $\bfX=\bfU$ or $\bfX=\mathbf{O^-}$;
    \end{enumerate}
    \item and otherwise $d= 2^n$ for some $n\geq4$, $q$ is an odd prime, $\bfX\in\{\bfL, \bfSp, \mathbf{O^+}\}$, and if $ {\rm Prob}_6(g_1,g_2,\bfX)>0$ then either
     \begin{enumerate}[\rm (i)]
        \item $(e_1,e_2)= (d/2,d/2)$, $r_1=r_2$ is $e_1+1=2^{n-1}+1$ or  $2e_1+1=2^{n}+1$, and 
        \[
 {\rm Prob}_6(g_1,g_2,\bfX) <  \left\{ \begin{array}{ll}  
\frac{45}{1024}\cdot q^{-d^2/4}& \mbox{if $\bfX=\bfL$}\\
\frac{45}{1024}\cdot q^{-d^2/4+23d/10} & \mbox{if $\bfX=\bfSp$ or $\mathbf{O^+}$;}\\
\end{array}\right.
\]
        \item or  $(e_1,e_2)=(d-2,2)$, with $d=2^n\geq 32$, $q\geq5$, and $n, q$ both odd primes as in Lemma~$\ref{lem:c6-1}${\rm(b)(iii)},  $r_1=e_1+1=2^n-1$, $r_2$ is an odd prime divisor of $q+1$ (in particular if $n=5$ then $q\geq 11$), and 
    \[
 {\rm Prob}_6(g_1,g_2,\bfX) <  \left\{ \begin{array}{rl}  
\frac{45}{640}\cdot  q^{-\frac{59}{25}d+8}& \mbox{if $\bfX=\bfL$}\\
5.3\cdot 10^{-3}\cdot   q^{-d+3} & \mbox{if $\bfX=\bfSp$ or $\mathbf{O^+}$ and either $d\ge 128$ or $q\ge 19$}\\
6.9\cdot q^{-d+3} & \mbox{if $\bfX=\bfSp$ or $\mathbf{O^+}$ and $d=32$ with $q\in\{11,13,17\}$.}
\end{array}\right.
\]
    \end{enumerate}  
\end{enumerate}
Moreover we take $\cM_6(\bfX)=\emptyset$ if any of the conditions of part {\rm(a)} hold, or if $r_1, r_2$ are not as in part {\rm(b)}, and otherwise
\begin{quote}
$\cM_6(\bfX)$ is the $G$-conjugacy class of subgroups $M = R.M_0$ occurring in Table~$\ref{t:c6}$.
\end{quote}
\end{proposition}

 \begin{table}
\caption{Possibilities for $R$ and $M$ for $\cM_6(\bfX)$ in Proposition~\ref{prop:c6}}
\begin{tabular}{cllll}
  \toprule
Line&$\bfX$ & $R$ & $M$ & Conditions\\
    \midrule
1&$\bfL$   & $(q-1)\circ 2^{1+2n}$ & $R.\Sp_{2n}(2)$ & $q=p\equiv 1\pmod{4}$\\
2&$\bfSp$   & $2_-^{1+2n}$ & $R.\SO^-_{2n}(2)$ & $q=p\equiv \pm 1\pmod{8}$\\ 
 &             &  & $R.\Omega^-_{2n}(2)$ & $q=p\equiv \pm 3\pmod{8}$\\ 
3&$\mathbf{O^+}$   & $2_+^{1+2n}$ & $R.\SO^+_{2n}(2)$ & $q=p\equiv \pm 1\pmod{8}$\\ 
 &             &  & $R.\Omega^+_{2n}(2)$ & $q=p\equiv \pm 3\pmod{8}$\\ 
\bottomrule
\end{tabular}
\label{t:c6}
\end{table}

To prove Proposition~\ref{prop:c6}, we work with the following assumptions.

\begin{hypothesis}[for $\Asch_6$]\label{H:AC6} Assume that Hypothesis~$\ref{hyp}$ holds, where $d=s^n>8$ for some prime $s$ dividing $q^u-1$ and $n\geq1$, and that ${\rm Prob}_6(g_1,g_2,\bfX)\ne 0$. Let $(g,g')$ be a stingray duo in $g_1^G\times g_2^G$ such that $H\coloneq\langle g,g'\rangle$  is contained in the normaliser $M$ of an extraspecial $s$-group $s^{1+2n}$, or a symplectic type group $2^{1+2n}\circ4$. Assume also that  $H\not\leq L$ for any $L\in \cM_1(\bfX)\cup\dots \cup\cM_5(\bfX)$.
\end{hypothesis}

 First we obtain some restrictions on the parameters, and on the subgroup $M$.
 
 \begin{lemma}\label{lem:c6-1} 
 Assume that Hypothesis~$\ref{H:AC6}$ holds. Then
 \begin{enumerate}[{\rm (a)}] 
     \item 
     the prime $s=2$, $n\geq4$, $q$ is an odd prime, and $M=N_G(R)$ is self-normalising, with $R$ an absolutely irreducible $2$-group as in Table~$\ref{t:c6}$. In particular, $\bfX\ne \bfU, \mathbf{O^-}$. 
     \item 
     One of the following holds for the parameters $e_i$ and $r_i$:
     \begin{enumerate}[{\rm (i)}]
         \item 
         $e_1=e_2=d/2=2^{n-1}$ and $r_1=r_2=2^{n-1}+1$ with $n-1=2^{n_0}\geq4$,
         \item 
         $e_1=e_2=d/2=2^{n-1}$ and $r_1=r_2=2^{n}+1$ with $n=2^{n_0}\geq4$,
         \item 
         $e_1=d-2=2^n-2$, $r_1=e_1+1=2^n-1$ with $n$ an odd prime, $n\geq5$, $q\geq5$, and $e_2=2$ with $r_2$ an odd prime divisor of $q+1$, so $q\ne 3, 7$.
         Moreover (recalling that $q, n$ are odd primes), either $n\ge 7$ with $q\ge5$, or  $n=5$ with $q\ge11$.  
     \end{enumerate}
     \item 
     For all $\bfX\ne \bfU, \mathbf{O^-}$, there is a single $G$-conjugacy class of subgroups $M$ as in Table~$\ref{t:c6}$, and for a fixed $g\in g_1^G\cap M$, the number of $G$-conjugates of $M$ containing $g$ is 
     $
     Y=|C_G(g_1)|\cdot e_1/|N_M(\langle g\rangle)|,
     $
     and $|g_1^G\cap M| = e_1\cdot |M|/|N_M(\langle g\rangle)|$. 
     Moreover, $|N_M(\langle g\rangle)|$ is given by the following table, where $z\coloneq|Z(M)|=|Z(R)|$ as in Table~$\ref{t:c6}$.
     
     \medskip
     \begin{tabular}{c|ccc}
     \toprule
    Part of {\rm (b)}  & {\rm (b)(i)}  & {\rm (b)(ii)}  & {\rm (b)(iii)} \\
     $|N_M(\langle g\rangle)|$ & $48(n-1)r_1 z$
        & $2nr_1z$  & $2nr_1z$ \\
        \bottomrule
        \end{tabular}
 \end{enumerate}
  In particular Proposition~$\ref{prop:c6}${\rm(a)} is proved, and the content of the set $\cM_6(\mathbf{X)}$ given in Proposition~$\ref{prop:c6}$ is justified.
 \end{lemma}

\begin{proof}
(a) If $e_1>d/2$ then, by \cite[Lemma 4.3 and Example 2.5]{GPPS}, it follows that $s=2$. On the other hand, if $e_1=d/2$ then $d=s^n$ is even and again $s=2$. Thus in all cases $s=2$, and hence $n\geq4$ (since $d=2^n>8$) and $q$ is odd (since $s$ divides $q^u-1$). It follows from \cite[Chapter 4.6, Table 4.6.B]{KL} (or see \cite[Table 2.9]{BHRD}) that $M=N_G(R)$ is self-normalising with $R$ and $M$ as in one of the lines of Table~\ref{t:c6}, or $\bfX=\bfU$, $R=(q+1)\circ 2^{1+2n}$, and $M=R.\Sp_{2n}(2)$ with $q=p\equiv 3\pmod{4}$. 
Moreover since $\langle g,g'\rangle$ lies in no subgroup in $\cM_1(\bfX)\cup\dots \cup\cM_5(\bfX)$, it follows in particular that $M$ does not preserve a subfield structure on $V$, and hence $q$ is an odd prime (by \cite[Proposition 4.6.3(ii) and Table 4.6.B]{KL}, since $s=2$). In particular, $\bfX\ne \mathbf{O^-}$, and we will show below that the unitary case does not arise; apart from this part (a) is proved. 

(b) Now for all types $\bfX$, 
the group $M$ is contained in a subgroup  of $\GL_d(q^u)$ of the form $\widehat{M}\coloneq(Z_{q^u-1} \circ 2^{1+2n})\ldotp \Sp_{2n}(2)$ (see~\cite[Section 4.6]{KL}). 
We consider first the ppd $e_1$-stingray element $g$ of $M$. Since $|g|=r_1$ does not divide $q^u-1$ (as $e_1\geq2$), it follows that $r_1$ is odd and  $g$ does not lie in
the normal subgroup $\widehat{R}\coloneq Z_{q^u-1} \circ 2^{1+2n}$ of $\widehat{M}$. Hence $\widehat{R}g$ is a semisimple element of $\widehat{M}/\widehat{R}\cong\Sp_{2n}(2)$, so $r_1$ divides $2^{j}\pm1$ for some least positive integer $j\leq n$. Since $r_1$ is a primitive prime divisor of $(q^u)^{e_1}-1$, also $r_1=ke_1+1$ for some positive integer $k$, and so $2^j+1\geq r_1=ke_1+1\geq kd/2+1\geq 2^{n-1}+1$. In particular $k\leq 2$. Suppose first that $k=2$. Then $r_1=2^n+1$ and $e_1=d/2$, so also $e_2=d/2$. Since $r_1=2^n+1$ is a (Fermat) prime, the exponent $n$ must be a $2$-power $n=2^{n_0}\geq4$. Exactly the same argument applied to $(g', r_2)$ yields that either $r_2=r_1$, so that part (b)(ii) holds, or $r_2=e_2+1=2^{n-1}+1$. However, in the latter case, for $r_2$ to be prime,  we would also require $n-1=2^{n_0}-1$ to be a $2$-power, which is a contradiction since $n\geq 4$. 

Thus we now assume that $k=1$, so $2^j+1\geq r_1=e_1+1\geq 2^{n-1}+1$. Suppose next that $j<n$. Then $r_1=2^{n-1}+1$ and $e_1=d/2$, so also $e_2=d/2$. As before, since $r_1=2^{n-1}+1$ is a (Fermat) prime, the exponent $n-1$ must be a $2$-power, so $n-1=2^{n_0}\geq 3$ and this must be at least 4 for $r_1$ to be prime. Applying this argument to $(g', r_2)$ yields that either $r_2=r_1$, so that part (b)(i) holds, or $r_2=2e_2+1=2^{n}+1$. However in the latter case, for $r_2$ to be prime  we would also require $n+1=2^{n_0}+1$ to be a $2$-power, which is a contradiction. Thus we may assume that $j=n$ so $r_1$ divides $2^n\pm1$. We also have $r_1=e_1+1\leq 2^n-1$ since $e_1\leq d-2$. Thus $2^{n-1}+1\leq r_1\leq 2^n-1$ and $r_1$ divides $2^n\pm 1$. The only possibility is $r_1=2^n-1$ with $e_1=d-2$. For $r_1=2^n-1$ to be prime we require that $n$ is an odd prime and $n\geq5$ (since $n\geq4$). Also $e_2=2$ so $r_2$ divides $q^u+1$. We note at this point that in all cases the $e_i$ are even, and hence $\bfX\ne \bfU$, by Table~\ref{tab:one}, completing the proof of part (a). This implies that $u=1$, and hence in the case covered by (b)(iii), $r_2$ is an odd prime divisor of $q+1$, so $q\ne 3, 7$, that is, either $q=5$ or $q\ge11$ (as $q$ is an odd prime). To complete the proof it remains to show that, in case (b)(iii), the case $q=n=5$ is not possible: we note in this exceptional case that  
$e_1=d-2=30$, and $r_1=31$ with $r_1$ a ppd of $q^{30}-1$. However, the
prime  $31$ already divides $5^3-1$, so the ppd condition on $r_1$ fails. This completes the proof of part (b).

(c) In type $\bfL$, all subgroups $M$ as in Line~1 of Table~\ref{t:c6} are $G$-conjugate by \cite[Proposition 4.6.3(ii) and Lemma 2.10.14]{KL}, while the same is true for types $\bfSp$ and $\mathbf{O^+}$ by  \cite[Proposition 4.6.3(i)]{KL} (since $Z(R) =Z_2$). 

Now we consider $g_1^G\cap M$. By Lemma~\ref{lem:uperp}(c), each $h\in g_1^G$ lies in a unique cyclic torus $T$ and each such $T$ contains exactly $e_1$ elements of $g_1^G$. Moreover these $e_1$ elements all lie in the unique subgroup of $T$ of order $r_1$. Now fix $g\in g_1^G\cap M$. 
By part (b) it follows that the only factor $2^{2j}-1$ of $|\Sp_{2n}(2)|$ which is divisible by $r_1$ has $j=n-1, n, n$ in parts (i), (ii), (iii), respectively; and it follows that in all three cases  $\langle g\rangle\cong Z_{r_1}$ is a Sylow $r_1$-subgroup of $M$. Thus $|g_1^G\cap M|$ is equal to $|M:N_M(\langle g\rangle)|$ (the number of Sylow $r_1$-subgroups of $M$) times $e_1$ (the number of $g_1$-conjugates contained in each of these Sylow subgroups), that is to say (since $M=N_G(M)$),
$
|g_1^G\cap M| = e_1\cdot |M|/|N_M(\langle g\rangle)|.
$
Let $Y$ be the number of $M$-conjugates containing $g$. Then, the number of pairs $(h,L)$ such that $h\in g_1^G, L\in M^G$, and $h\in L$, satisfies
\[
\frac{|G|}{|C_G(g_1)|}\cdot Y = \frac{|G|}{|M|}\cdot |g_1^G\cap M| = \frac{|G|\cdot e_1}{|N_M(\langle g\rangle)|},
\]
and hence $Y=|C_G(g_1)|\cdot e_1/|N_M(\langle g\rangle)|$. Finally we determine $|N_M(\langle g\rangle)|$.

We claim that $|N_M(\langle g\rangle):C_M(g)|=2j$ (with $j=n-1, n, n$ as above). To see this we note that $Rg$ generates a cyclic torus, say $T_0$,  of $M/R$ (which is $\Sp_{2n}(2)$ or $\SO_{2n}^\pm(2)$) of order $r_1= 2^{n-1}+1, 2^n+1, 2^n-1$, respectively. The normaliser of $T_0$ modulo its centraliser is cyclic of order $2j=2(n-1), 2n, 2n$, respectively, and the claim follows. This means that  $|N_M(\langle g\rangle)|=2j|C|$, where $C\coloneq C_M(g)$. To determine $C$ we note first that $C$ contains $Z(M)$, and we have $C/Z(M)\leq M/Z(M) = 2^{2n}\ldotp M_0$ where $M_0=\Sp_{2n}(2)$ or $\SO^\pm_{2n}(2)$ acting naturally on the normal subgroup $2^{2n}$ of $M/Z(M)$. In case (b)(i), (ii), (iii), the element $Z(M)g$ centralises a subgroup of $2^{2n}$ of order $2^2, 1, 1$ respectively, and the centraliser of $Rg$ in $M/R\leq \Sp_{2n}(2)$ has order $r_1|S_3|, r_1, r_1$, respectively. Thus $|C|$ is $24r_1|Z(M)|, r_1|Z(M)|, r_1|Z(M)|$ respectively.

We note that Proposition~\ref{prop:c6}(a) and the validity of $\cM_6(\mathbf{X)}$ follow immediately from parts (a)--(c).
\end{proof}

Next we derive upper bounds $|M|$ for the groups $M\in \cM_6(\mathbf{X)}$. The proof is a simple modification of the argument of \cite[Lemma 4.1]{PSer}.

\begin{lemma}\label{lem:c6-order}
For  $\bfX$ and $M$ as in Table~$\ref{t:c6}$, with $d=2^n\geq16$ and  $q$ an odd prime,  
\[|M| < \begin{cases}
\frac{45}{64}\cdot q^{\frac{9}{5}d} & \textrm{for all } \bfX \textrm{ and } n\ge 4, q\ge 3,\\
\frac{45}{64}\cdot q^{\frac{16}{25}d} & \textrm{for all } \bfX \textrm{ and } n\ge 5, q\ge 11, \textrm{ or } n\ge 7, q\ge 5.\\
\end{cases}
\]
Moreover, if $ \bfX =\bfSp$ or $\mathbf{O^+}$ and the parameters are as in Lemma~$\ref{lem:c6-1}${\rm(b)(iii)}, so  either $n=5$ with $q\geq11$, or $n\geq 7$ with $q\ge5$, we have
\[
|M| < \begin{cases}
5.81\cdot q^{\frac{1}{2}\left(1 - \frac{3}{125}\right)d}&
 \textrm{for  }  n=5 \textrm{ and } 11\le q\le 17,\\
5.81\cdot q^{\frac{2}{5}d}
 & \textrm{for } 
 n\ge 7, q\ge 5, \textrm{ or } n=5, q\ge 19.
\\
\end{cases}
\]
\end{lemma}

\begin{proof} 
Note first that $n\ge 4$ as $d=2^n\ge 16$.  Let $z=|Z(M)|$. 
 It follows from Table~$\ref{t:c6}$ that     $|M|<z\cdot2^{2n}\cdot|\Sp_{2n}(2)|$.  Since $|\Sp_{2n}(2)|<\frac{45}{64}\cdot 2^{2n^2+n}$ (see \cite[Table 3]{PSer} where, as $n\ge 4$,  we use the improved bound $\Omega(1,n,2^2)<(1-\frac{1}{2^2})(1-\frac{1}{2^4})=\frac{45}{64}$),
 and since in all cases $z<q$,  it follows,     that
\[
|M| < z\cdot \frac{45}{64}\cdot 2^{2n^2 +3n}  <  \frac{45}{64}\cdot q^{(2n^2 +3n)/\log_2(q) +1}. 
\]
By induction (on $n$) it is
easy to show  that,  for $n\ge 4$
and $q\ge 3$,
\[ (2n^2 + 3n )/\log_2(q) +1\le (2n^2+3n)/\log_2(3) +1 < \frac{9}{5} 2^n = \frac{9}{5}d,\]
while for either $n\ge 5$ with $q\ge 11$, or $n\ge 7$ with $q\ge 5$,  we have
\[ (2n^2 + 3n )/\log_2(q) +1 < \frac{16}{25} 2^n = \frac{16}{25}d.
\]
This proves the first bounds in the lemma.

Now we assume that $\bfX = \bfSp$ or $\bfO^+$ and either $n=5$ with $q\geq11$, or $n\geq 7$ with $q\ge5$. In this case $z=2$ and
it follows from Table~$\ref{t:c6}$ that     $|M|<2^{2n+1}\cdot|\SO^-_{2n}(2)|$. 
Since 
\[
|\SO^-_{2n}(2)|< \frac{45}{64}\cdot (1+ 2^{-n})\cdot 2^{2n^2-n+1} \le \frac{45}{32}\cdot\left(1+\frac{1}{32}\right)\cdot   2^{2n^2-n} < 1.451 \cdot 2^{2n^2-n}
\] 
(see \cite[Table 3]{PSer}, again using the improved bound $\Omega(1,n-1,2^2)<(1-\frac{1}{2^2})(1-\frac{1}{2^4})=\frac{45}{64}$)
 it follows  that
\[
|M| <  1.451\cdot 2^{2n^2 +n+1}
= 1.451 \cdot 4\cdot 2^{2n^2 +n-1}
< 5.81\cdot  q^{(2n^2 + n-1)/\log_2(q)}. 
\]
For $q\ge 11$ and $n=5$ we have
\[(2n^2+n-1)/\log_2(q)\le (2n^2+n-1)/\log_2(11) < \frac{1}{2}\left(1 - \frac{3}{125}\right)2^n = \frac{1}{2}\left(1 - \frac{3}{125}\right)
d,\] 
giving the bound asserted for $n=5$ with $11\leq q\leq 17$. Further, if either $n\ge 7$ with $q\geq 5$, or  $n=5$ with $q\ge 19$,   we find 
 \[
 (2n^2+n-1)/\log_2(q)< \frac{2}{5} 2^n = \frac{2}{5}d
 \]
 which yields the asserted bounds. 
 \end{proof}

Finally we prove Proposition~\ref{prop:c6}.

\bigskip\noindent
\emph{Proof of Proposition~$\ref{prop:c6}$.}\quad 
Part (a) of Proposition~\ref{prop:c6}, and the form of $\cM_6(\bfX)$  follow from Lemma~\ref{lem:c6-1}. 

Suppose then that Hypothesis~\ref{H:AC6} holds, in particular that  ${\rm Prob}_6(g_1,g_2,\bfX)>0$. The assertions in Proposition~\ref{prop:c6}(b) about $d, q$, the $e_i$ and $r_i$ are proved in Lemma~\ref{lem:c6-1}, and it remains to estimate ${\rm Prob}_6(g_1,g_2,\bfX)$ for the cases in Lemma~\ref{lem:c6-1} (b)(i), (b)(ii) and (b)(iii). In the first two cases 
$e_1=e_2=d/2=2^{n-1}$, while in the third case $e_1=d-2$ and $e_2=2$. Let $g\in g_1^G$ and let $M\in\cM_6(\bfX)$ such that $g\in M$.  By Lemma~\ref{lem:c6-1}~(c), $\cM_6(\bfX)$ is a single $G$-conjugacy class, the number of $G$-conjugates $L\in\cM_6(\bfX)$ such that $g\in L$ is $|C_G(g_1)|\cdot e_1/|N_M(\langle g\rangle)|$, and $|N_M(\langle g\rangle)|=cr_1z$ with $z=|Z(M)|$, and $c=48(n-1), 2n, 2n$ in case (b)(i),  (b)(ii),  or (b)(iii), respectively.
Thus, by \eqref{e:asch3}, 
\begin{equation}\label{e:c6-1}
    {\rm Prob}_6(g_1,g_2,\bfX) \leq   
\frac{|C_G(g_1)|\cdot e_1}{cr_1z}\cdot 
\frac{|\mathcal{S}_6(g,M)|}{|N(d,q,e_1,\bfX)|} =
\frac{|C_G(g_1)|}{|N(d,q,e_1,\bfX)|}\cdot \frac{e_1\cdot |\mathcal{S}_6(g,M)|}{cr_1z},
\end{equation}
with $|N(d,q^u, e_1,\bfX)|$ as in \eqref{e:ndex}, and  $\mathcal{S}_6(g,M)$ as in \eqref{e:siM}, namely the set of all $g'\in g_2^G\cap M$ such that $(g,g')$ is a stingray duo and $\langle g,g'\rangle$ is not contained in a subgroup of $\cM_j(\bfX)$ for any $j\leq 5$. 
The first factor in the last expression in \eqref{e:c6-1} is determined in Propositions~\ref{p:Lbds} and~\ref{p:Xbds}.
As we do not have an easy description of the actions of $M$ on $V$, we over-estimate $|\mathcal{S}_6(g,M)|$ by $|g_2^G\cap M|$ for the first two cases and by $|M|$ in the third case. 

\medskip\noindent
\emph{Cases {\rm(b)(i)} and {\rm(b)(ii)} of Lemma~$\ref{lem:c6-1}$:}\quad
Here $e_2=e_1$, and hence, by Lemma~\ref{lem:c6-1}(c), $|g_2^G\cap M|$ is equal to $|g_1^G\cap M|$, which  is equal to $e_1\cdot |M|/|N_M(\langle g\rangle)|=e_1\cdot |M|/cr_1z$. Thus, since $e_1<r_1$, 
\begin{equation}\label{e:c6-2}
  {\rm Prob}_6(g_1,g_2,\bfX) \leq  \frac{|C_G(g_1)|}{|N(d,q,e_1,\bfX)|}\cdot \frac{e_1^2\cdot |M|}{(cr_1z)^2} < \frac{|C_G(g_1)|}{|N(d,q,e_1,\bfX)|}\cdot \frac{|M|}{4n^2  z^2} 
\end{equation}
where we have used the common upper bound $1/c\leq 1/2n$ for the cases (b)(i) and (b)(ii).
Note that, by Lemma~\ref{lem:c6-order}, $|M|< \frac{45}{64}q^{2d}$.
We now apply Propositions~\ref{p:Lbds} and~\ref{p:Xbds} 
separately for each type.
Note that $(q^{e_i}-1)/(q-1) < 2q^{e_i-1}.$
If $\bfX=\bfL$, noting that $(q^{e_i}-1)/(q-1) < 2 q^{e_i}$, we have 
\[
 {\rm Prob}_6(g_1,g_2,\bfL)\leq  \frac{(q^{e_1}-1)(q^{e_2}-1)}{q^{2e_1e_2}}\cdot
 \frac{\frac{45}{64}q^{\frac{9}{5}d}}{4n^2 (q-1)^2}<  \frac{45}{64 \,n^2}\cdot \frac{q^{e_1+e_2-2}\cdot q^{\frac{9}{5}d}}{q^{2e_1e_2}}
 ,
\]
noting that  $(q^{e_i}-1)/(q-1) <2 q^{e_i-1}.$
Since $e_1=e_2=d/2$, this upper bound is $\frac{45}{64n^2}\cdot q^{-b}$ where $b=d^2/2 -\frac{14}{5}d+2$, and for all $d=2^n\geq 16$ we have $b>d^2/4$, and hence ${\rm Prob}_6(g_1,g_2,\bfL)\leq \frac{45}{1024} q^{-d^2/4}$ as in Proposition~\ref{prop:c6}(b)(i). In the other lines of Table~\ref{t:c6}, $\bfX=\bfSp$ or $\mathbf{O^+}$, and we have 
 $q$ odd, and the bounds $1/2\leq 1-3/(2q)\leq  \mathbf{k}(d,q,e_1,\bfX) < 1$. Note that $(q^{e_1/2}+1)(q^{e_2/2}+1)<2q^{(e_1+e_2)/2}$. Thus using Propositions~\ref{p:Xbds}, \eqref{e:c6-2} yields
 \begin{align*}
 {\rm Prob}_6(g_1,g_2,\bfX) \leq 
      &  \frac{(q^{e_1/2}+1)(q^{e_2/2}+1)}{\mathbf{k}(d,q,e_1,\bfX)\cdot |\mathcal{U}(d,q,e_1,\bfX)|}\cdot
 \frac{|M|}{16n^2} 
 < \frac{q^{(e_1+e_2)/2}\cdot |M|}{ \bfa(q,\bfX)\cdot q^{e_1e_2}\cdot 4n^2}.
 \end{align*} 
For both types, we have $4\bfa(q,\bfX)>1$, by Table~\ref{t:abLU} (since $q$ is odd), so using $e_1=e_2=d/2$, we have 
${\rm Prob}_6(g_1,g_2,\bfX) < \frac{1}{n^2}\cdot q^{-d^2/4 + d/2}\cdot |M|$.  By
Lemma~\ref{lem:c6-order}, $|M|< \frac{45}{64}q^{\frac{9}{5}d}$ and
thus
\[{\rm Prob}_6(g_1,g_2,\bfX) < \frac{45}{1024} \cdot q^{-d^2/4 + 23d/10}.\]

\medskip\noindent
\emph{Cases {\rm(b)(iii)} of Lemma~$\ref{lem:c6-1}$:}\quad Here $e_1=d-2$ and $e_2=2$, the dimension $d=2^n\geq 32$ with $n$ and $q$ odd primes,
and  either $n\ge 7$ with $q\ge5$, or $n=5$ with $q\geq 11$.

We use $|\mathcal{S}_6(g,M)|<|M|$, and the upper bound  for $|M|$
depending on $\bfX$ given by Lemma~\ref{lem:c6-order}. Thus, since $e_1<r_1$ and $c=2n$, \eqref{e:c6-1} becomes 
\[
  {\rm Prob}_6(g_1,g_2,\bfX) \leq  \frac{|C_G(g_1)|}{|N(d,q,e_1,\bfX)|}\cdot \frac{e_1\cdot |M|}{cr_1z} < \frac{|C_G(g_1)|}{|N(d,q,e_1,\bfX)|}\cdot \frac{|M|/z}{2n}. 
\]

Again we apply Propositions~\ref{p:Lbds} and~\ref{p:Xbds}.
If $\bfX=\bfL$, 
then by Lemma~\ref{lem:c6-order}
$|M| < \frac{45}{64}q^{\frac{16}{25}d}$, thus
\[
 {\rm Prob}_6(g_1,g_2,\bfL)\leq  \frac{(q^{d-2}-1)(q^{2}-1)}{q^{4d-8}}\cdot
 \frac{\frac{45}{64}q^{\frac{16}{25}d}}{2n}< \frac{45}{640}q^{-\frac{59}{25}d+8}.
\]

If $\bfX=\bfSp$ or $\mathbf{O^+}$ then 
\[
 {\rm Prob}_6(g_1,g_2,\bfX) \leq 
      \frac{(q^{(d-2)/2}+1)(q+1)}{\mathbf{k}(d,q,d-2,\bfX)\cdot |\mathcal{U}(d,q,d-2,\bfX)|}\cdot\frac{|M|}{2n}.
\]
By Proposition~~\ref{p:Xbds}, $|\mathcal{U}(d,q,d-2,\bfX)|> (47/128)\cdot q^{2d-4}$.
Also  $(q^{(d-2)/2}+1)(q+1)< 2q^{d/2}$ and $ 1-3/(2q)\leq  \mathbf{k}(d,q,e_1,\bfX).$ 
We consider first the cases
$n\ge 7$ with $q\ge 5$, and $n=5$ with $q\ge 19,$ (recall $n, q$ are odd primes).
Here by
Lemma~\ref{lem:c6-order}, $|M| < 5.81\cdot q^{\frac{2}{5}d}$
and hence \[
 {\rm Prob}_6(g_1,g_2,\bfX) < 
\frac{2q^{d/2}\cdot 5.81 \cdot q^{\frac{2}{5}d}}{(1-3/(2q))\cdot (47/128)\cdot q^{2d-4}\cdot 2n} 
< \frac{15.83  }{n(1-3/(2q))\cdot q^{2^n/10-1}} \cdot q^{-d+3}. 
\]
In all of these cases
\[
\frac{15.83}{n(1-3/(2q))\cdot q^{2^n/10-1}} \cdot q^{-d+3}
< 5.3\cdot 10^{-3} \cdot q^{-d+3}.
\]
This leaves the case $n=5$ with
$q\in\{11, 13, 17\}$; here $2n=10$   and 
\[|M| \le 2^{2n+1}\cdot |\SO^-_{2n}(2)|.\]
As in this case $(q^{(d-2)/2}+1)(q+1) < 1.1\cdot  q^{d/2},$ and  $|\mathcal{U}(d,q,d-2,\bfX)|> (227/500)\cdot q^{2d-4}$ by using $q\ge 11$ in the proof of  Proposition~\ref{p:Xbds},
\begin{align*}
{\rm Prob}_6(g_1,g_2,\bfX) &< 
\frac{1.1\cdot q^{d/2}\cdot 2^{2n+1} |\SO^-_{2n}(2)|}{(1-3/(2q))\cdot (227/500)\cdot q^{2d-4}\cdot 10} 
=\frac{1.1\cdot 50\cdot  2^{2n+1}\cdot |\SO^-_{2n}(2)|}{227\cdot (1-3/(2q))\cdot q^{\frac{3}{2}d-4}}\\
&< 
\frac{55\cdot  2048\cdot |\SO^-_{2n}(2)|}{227\cdot (1-3/(2q))\cdot q^{15}} \cdot q^{-d+3}
< 6.9\cdot q^{-d+3}.
\end{align*}
This completes the proof of Proposition~\ref{prop:c6}.

\section{\texorpdfstring{$\Asch_8:$}{Asch8} classical groups}

We use the notation and assumptions from Hypothesis~\ref{hyp}, and in this section we estimate the proportion  ${\rm Prob}_8(g_1,g_2,\bfX)$ of stingray duos $(g,g')\in g_1^G\times g_2^G$ for which $\langle g,g'\rangle$ lies in a classical subgroup $M$ of $G$ lying in 
the Aschbacher category $\Asch_8$, that is, $M$ preserves a non-degenerate form on $V$ other than the form defining $G$ of type $\bfX$. For a non-zero contribution to ${\rm Prob}_8(g_1,g_2,\bfX)$, we do not have $\bfX=\bfSp$ with $q$ even, by Remark~\ref{r:MX-Sp}(b), and taking this into account we record in Table~\ref{t:asch8} the possibilities for $M$ (see \cite[Table 4.8.A]{KL}). 
In particular, such subgroups $M$ exist only if the type $\bfX=\bfL$, and we take 
the family $\cM_8(\bfX)$ as in the main result of this section, Proposition~\ref{prop:classicalmax}.

\begin{proposition}\label{prop:classicalmax}
Assume that Hypothesis~$\ref{hyp}$ holds, and $d>8$. Then $ {\rm Prob}_8(g_1,g_2,\bfX)=0$ if $\bfX \ne \bfL$ or if $d$ is odd, and otherwise 
$$
{\rm Prob}_8(g_1,g_2,\bfL)\le \begin{cases}
     1.46\cdot q^{-e_1e_2+d/2}  &\mbox{if $q=2$} \\
     3.18\cdot q^{-e_1e_2+d/2}  &\mbox{if $q\ge3$.} 
\end{cases}  
$$
Moreover we take $\cM_8(\bfX)=\emptyset$ if $\bfX \ne \bfL$ or if $d$ is odd, and for $d$ even we take $\cM_8(\bfL)$ to consist the classical groups $M$ of type $\bfY$ as in Table~$\ref{t:asch8}$.
\end{proposition}

  \begin{table}
\caption{Maximal subgroups $M$
in $\Asch_8$ containing $e_i$-stingray\\ elements in Lemma~\ref{lem:classicalmax}}
\begin{tabular}{lcll}
  \toprule
$\bfX$ & $M$ & Conditions & $\bfY$\\
    \midrule
$\bfL$   & $\GSp_d(q)$ & $d,e_i$ even & $\bfSp$\\
$\bfL$   & $\GU_d(q_0)$ & $e_i$ odd, $q=q_0^2$ & $\bfU$\\
$\bfL$   & $\GO^{\varepsilon}_d(q)$   & $d,e_i$ even, $q$ odd & $\mathbf{O^\varepsilon}$\\
\bottomrule
\end{tabular}
\label{t:asch8}
\end{table}

To prove Proposition~\ref{prop:classicalmax}, we work with the following assumptions.

\begin{hypothesis}[for $\Asch_8$]\label{H:AC8} Assume that Hypothesis~$\ref{hyp}$ holds with $d>8$, and also that ${\rm Prob}_8(g_1,g_2,\bfX)\ne 0$,  so $\bfX = \bfL$ (by our discussion above). Let $(g,g')$ be a stingray duo in $g_1^G\times g_2^G$ such that $H\coloneq\langle g,g'\rangle\leq M$  for
    some $M\in \cM_8(\bfL)$, while $H\not\leq L$ for any $L\in \bigcup_{i=1}^7\cM_i(\bfL)$.
\end{hypothesis}

We begin with a lemma which determines the $G$-conjugates of a stingray element lying in a subgroup
$M\in \cM_8(\bfL)$.

\begin{lemma}\label{lem:classicalmax}
Assume that Hypothesis~$\ref{H:AC8}$ holds, and let $i\in\{1,2\}$, $e=e_i, r=r_i$ and $h\in g_i^G\cap M$. 
Then the `Conditions' in Table~$\ref{t:asch8}$ hold, $M$ is self-normalising and  $g_i^G \cap M = h^M.$
\end{lemma}

\begin{proof}
The subgroup $M$ is
 self-normalising as $M$ is maximal in $G$. Since 
$h^M \subseteq g_i^G \cap M$, it remains in each case
to show that $g_i^G \cap M\subseteq h^M.$
 Recall that $\bfX=\bfL$,  $G = \GL_d(q)$, and that, by Table~\ref{t:asch8}, 
$M$ is isomorphic to $\GSp_d(q)$, $\GU_d(q)$,  or $\GO^{\varepsilon}_d(q)$. The `Conditions' in Table~$\ref{t:asch8}$ follow from Table~\ref{tab:stingraycond}.
Two elements in $g_i^{G}\cap M$ are conjugate in $M$ if and only if they are conjugate in $G,$ by
 \cite[Theorem~6.1.1]{DeFrancesci}. 
 \end{proof}

\noindent\textit{Proof of Proposition~$\ref{prop:classicalmax}$.} \quad
Let $g\in g_1^{G}$. We need to consider each $M\in\cM_8(\bfL)$ containing $g$. Now each such $M$ is a classical group of type $\bfY$ satisfying the
`Conditions' in some (unique) line  of Table~\ref{t:asch8}. 
 If we denote by $\cM_8(G,\bfY)$ the set of maximal subgroups of $G$ of type $\bfY$, then we can write \eqref{e:asch3} as 
\begin{align}\label{e:asch3-8}
    {\rm Prob}_8(g_1,g_2,\bfL) &\leq   \sum_{\bfY\in\{\bfSp, \bfU, \mathbf{O^\varepsilon}\}} \left(
\sum_{M\in\cM_8(G,\bfY), g\in M}
\frac{\left|\mathcal{S}_8(g,M)\right|}{\left|N(d,q,e_1,\bfL)\right|}\right),
\end{align}
 with $\mathcal{S}_8(g,M)$ as in \eqref{e:siM}. By Lemma~\ref{lem:classicalmax}, 
 $g_2^{G}\cap M$ is a single $M$-conjugacy class, and hence $\mathcal{S}_8(g,M)$ is a subset of 
 \[
 \{g'\in  g_2^{M} \mid \mbox{$(g,g')$ is a stingray duo in } M\}.
 \]
 Since $M\cong \GY_d(q)$, it follows from \eqref{e:ndex} that, if $M\in \cM_8(G,\bfY)$ contains $g$, then $|\mathcal{S}_8(g,M)|$ is at most  $|N(d,q,e_1,\bfY)|$, (noting that,  if $\bfY=\bfU$, then $M=\GU_d(q_0)$  and the second entry $q=q_0^2$). Moreover, since by Lemma~\ref{lem:classicalmax}, $N_G(M)=M$ and  $g_1^{G}\cap M$ is a single $M$-conjugacy class, it follows from Lemma~\ref{lem:countM} that the number of $M\in \cM_8(G,\bfY)$ containing $g$ is $|C_G(g_1)|/|C_M(g)|$; and both $C_G(g_1)$ and $C_M(g)$ are given by Lemma~\ref{lem:uperp}(c) (for types $\bfL$ and $\bfY$ respectively).  Thus by \eqref{e:asch3-8}, 
 \begin{align*}
{\rm Prob}_8(g_1,g_2,\bfL) &\leq \sum_{\bfY\in\{\bfSp, \bfU, \mathbf{O^\varepsilon}\}}  \frac{|C_G(g_1)|}{|C_M(g)|}\cdot \frac{|N(d,q,e_1,\bfY)|}{|N(d,q,e_1,\bfL)|}\\
    &=  \frac{|C_G(g_1)|}{|N(d,q,e_1,\bfL)|}\cdot \sum_{\bfY\in\{\bfSp, \bfU, \mathbf{O^\varepsilon}\}} \frac{|N(d,q,e_1,\bfY)|}{|C_M(g)|}.
 \end{align*}
   The first factor (and also each of the summands) is determined in Proposition~\ref{p:Lbds}.
In particular,
\begin{equation}\label{e:L}
\frac{|C_G(g_1)|}{|N(d,q,e_1,\bfL)|} = \frac{(q^{e_1}-1)(q^{e_2}-1)}{q^{2e_1e_2}}< q^{-2e_1e_2+d}.    
\end{equation}

For $\bfY=\bfU$, we have zero contribution unless $q=q_0^2$ is a square and the $e_i$ are both odd, and in this case it follows from Proposition~\ref{p:Xbds} that 
 \[
\frac{|N(d,q,e_1,\bfU)|}{|C_M(g)|} =  
\frac{\mathbf{k}(d,q_0^2,e_1,\bfU)|\cU(d,q_0^2,e_1,\bfU)|}{(q_0^{e_1}+1)(q_0^{e_2}+1)}
 < \frac{|\cU(d,q_0^2,e_1,\bfU)|}{(q_0^{e_1}+1)(q_0^{e_2}+1)}
< \frac{q_0^{2e_1e_2}}{q_0^{d}} = q^{e_1e_2-d/2}
 \]
 and the contribution to ${\rm Prob}_8(g_1,g_2,\bfL)$ is less than $q^{-2e_1e_2+d}\cdot q^{e_1e_2-d/2} = q^{-e_1e_2+d/2}$.

 For $\bfY=\bfSp$ or $\mathbf{O^\varepsilon}$, we have zero contribution unless the $e_i$ are both even, and in this case it follows from Proposition~\ref{p:Xbds}, (taking $\varepsilon=1$ if $\bfY=\bfSp$) that 
 \[
\frac{|N(d,q,e_1,\bfY)|}{|C_M(g)|} < \frac{|\cU(d,q,e_1,\bfY)|}{(q^{e_1/2}+1)(q^{e_2/2}+\varepsilon)}
< \frac{\bfb(q,\bfY)\cdot q^{e_1e_2}}{q^{e_1/2}(q^{e_2/2}+\varepsilon)}
 \]
 with $\bfb(q,\bfY)$ as in Table~\ref{t:abLU}. If $\bfY=\bfSp$ then $\bfb(q,\bfY)\leq 16/11$ if $q=2$ and $\bfb(q,\bfY)\leq 8/7$ if $q\geq3$. Hence the contribution to ${\rm Prob}_8(g_1,g_2,\bfL)$ is less than $q^{-2e_1e_2+d}\cdot (16/11)\cdot q^{e_1e_2-d/2} = (16/11)\cdot q^{-e_1e_2+d/2}$ if $q=2$, or $(8/7)\cdot q^{-e_1e_2+d/2}$ if $q\geq3$. If  $\bfY=\mathbf{O^\varepsilon}$ then $q$ is odd (see Table~\ref{t:asch8}) so $\bfb(q,\bfY)\leq 54/71$ and to estimate the contribution to ${\rm Prob}_8(g_1,g_2,\bfL)$ we use \eqref{e:L} to see that it is less than 
 \[
\frac{(q^{e_1}-1)(q^{e_2}-1)}{q^{2e_1e_2}} \cdot \frac{(54/71)\cdot q^{e_1e_2}}{q^{e_1/2}(q^{e_2/2}-1)}
< q^{-e_1e_2+d/2}\cdot (54/71)\cdot (1+q^{-e_2/2}).
 \]
Since $q\geq3$ and $e_2\geq2$, the final factor  $1+q^{-e_2/2} \leq 4/3$, and the contribution is therefore less than $1.015\cdot q^{-e_1e_2+d/2}$. 
 
 Finally we put these estimates together to obtain an upper bound for ${\rm Prob}_8(g_1,g_2,\bfL)$ of the form $k\cdot q^{-e_1e_2+d/2}$.  If at least one of the $e_i$ is odd, then the only non-zero contribution is from $\bfY=\bfU$ so we have $k=1$, and Proposition~\ref{prop:classicalmax} follows for all $q$ in this case. Suppose then that both of the $e_i$ are even. Then we can only have $\bfY=\bfSp$ or $\mathbf{O^\varepsilon}$. If $q=2$   then the only non-zero contribution is from $\bfY=\bfSp$ so we have $k=16/11 < 1.46$. On the other hand if $q \ge3$ then we may have contributions from $\bfY=\bfSp, \bfO^+$, and  $\bfO^-$, and we obtain $k=(8/7)+ 1.015 +  1.015 < 3.18$.
 This completes the proof of Proposition~\ref{prop:classicalmax}.

\section{\texorpdfstring{$\Asch_9:$}{Asch9} almost simple groups}\label{S:C9}

We use the notation and assumptions from Hypothesis~\ref{hyp}, and in this section we estimate the proportion  ${\rm Prob}_9(g_1,g_2,\bfX)$ of stingray duos $(g,g')\in g_1^G\times g_2^G$ for which $\langle g,g'\rangle\leq M$ where $M$ lies   in 
the Aschbacher category $\Asch_9$. In most cases these subgroups $M$ are maximal subgroups of $G$ lying in 
 $\Asch_9$ which can conceivably contain such stingray duos, while in the exceptional case where $\bfX=\bfSp$ and $q$ is even,  these are $\Asch_9$-subgroups of orthogonal subgroups of $G$, as described in Remark~\ref{r:MX-Sp}(b). Now $\Asch_9$-subgroups $M$ are projectively almost simple, that is, $M\cap Z(G)<N\lhdeq M$ such that $S \coloneq  N/(M\cap Z(G))$ is a nonabelian simple group and $M/(M\cap Z(G))\leq \Aut(S)$. 
Moreover, $N$ is absolutely irreducible on $V$.  We make the usual assumption $d>8$. The main result of the section is Proposition~\ref{prop:c9}; in particular, the bounds obtained there improve on those in \cite[Lemma 13.1]{PSY}  for the $\Asch_9$ case with $e_1=e_2=d/2$, see also \cite[Theorem 1.4]{GNP4}. 

\begin{proposition}\label{prop:c9}
Assume that Hypothesis~$\ref{hyp}$ holds with $d>8$.
\begin{enumerate}[{\rm (a)}]
    \item If $q^u\ne p$,  or if $\bfX \in\{ \bfL, \bfU\}$, or if $\bfX=\bfSp$ with $q$ odd, then $\textup{Prob}_9(g_1,g_2,\bfX)=0$.

    \item If $q^u=p$ and $\bfX= \bfO^\varepsilon$ for some $\varepsilon=\pm$, or $({\bfX},q)=(\bfSp,2)$, then 
\[
\textup{Prob}_9(g_1,g_2,\bfX)\le
\left\{\begin{array}{rc}
     8.24\cdot p^{-e_1e_2+d/2}& \mbox{ if $\bfX=\mathbf{O^\varepsilon}$  and $q^u=p\geq3$} \\
    28.1\cdot2^{-e_1e_2+d/2} & \mbox{ if $\bfX=\mathbf{O^\varepsilon}$  and $q^u=p=2$}\\
    6.38\cdot2^{-e_1e_2+d/2} & \mbox{ if $\bfX=\bfSp$  and $q^u=p=2$.}
\end{array}\right.
\]

\end{enumerate}
\end{proposition}

To prove Proposition~\ref{prop:c9}, we work with the following assumptions.

\begin{hypothesis}[for $\Asch_9$]\label{H:AC9} Assume that Hypothesis~$\ref{hyp}$ holds, with $d>8$, and also that ${\rm Prob}_9(g_1,g_2,\bfX)\ne 0$. Let $(g,g')$ be a stingray duo in $g_1^G\times g_2^G$ such that $H\coloneq\langle g,g'\rangle\leq M$,  for some $M\in \cM_9(\bfX)$, while $H\not\leq L$ for any $L\in \bigcup_{i=1}^8\cM_i(\bfX)$. Thus $M\cap Z(G)<N\lhdeq M$ with $S \coloneq  N/(M\cap Z(G))$ a nonabelian simple group, $M/(M\cap Z(G))\leq \Aut(S)$, and $N$ absolutely irreducible on $V$. 
\end{hypothesis}

The situation where $S=A_n$ and $V$ is the fully deleted permutation module for the natural action of $S$ is exceptional, and is considered separately in Subsection~\ref{sub:fdpm}. It arises only for  $\bfX=\mathbf{O^\varepsilon}$, and bounds on the contribution  $\textup{Prob}_9(g_1,g_2,\bfX;\textup{f.d. perm mod})$ from groups of this type to $\textup{Prob}_9(g_1,g_2,\bfX)$ are obtained in Proposition~\ref{prop:c9-1}. For this reason we write
\begin{equation*}
  \textup{Prob}_9(g_1,g_2,\bfX) = \textup{Prob}_9(g_1,g_2,\bfX;\textup{f.d. perm mod}) + \textup{Prob}_9^*(g_1,g_2,\bfX)  
\end{equation*}
where $\textup{Prob}_9^*(g_1,g_2,\bfX)$ is the contribution from all other kinds of subgroups in $\Asch_9$. 
We are able to show in Subsection~\ref{sub:c9-diff} that $\textup{Prob}_9^*(g_1,g_2,\bfX)=0$ when $d>8$. The proof in the case $e_1>d/2$ uses the classification of the $\Asch_9$-subgroups containing an $e_1$-ppd stingray element given in \cite[Main Theorem]{GPPS}.
On the other hand, if $e_1=e_2=d/2$, then we use a new result obtained in \cite[Theorem 1.2]{GNP4} which shows, for $d>8$, that no $\Asch_9$-subgroups contain a $d/2$-ppd stingray element apart from the groups $A_n, S_n$ acting on fully deleted permutation modules.

\subsection{Fully deleted permutation modules}\label{sub:fdpm}  
We estimate $\textup{Prob}_9(g_1,g_2,\bfX;\textup{f.d. perm mod})$ in Proposition~\ref{prop:c9-1}.

\begin{proposition}\label{prop:c9-1}
 Assume that Hypothesis~$\ref{H:AC9}$ holds. Then:
\begin{enumerate}[{\rm(a)}]
    \item  if $q\ne p$, or if $\bfX \in\{ \bfL, \bfU\}$, or if $\bfX =\bfSp$ with $q$ odd, then 
    \[
    \textup{Prob}_9(g_1,g_2,\bfX;\textup{f.d. perm mod})=0;
    \]
    \item if $\bfX= \bfO^\varepsilon$ for some $\varepsilon=\pm$ with $q=p$,  or if $\bfX =\bfSp$ with $q=2$, then
\[
\textup{Prob}_9(g_1,g_2,\bfX;\textup{f.d. perm mod})\le
\left\{\begin{array}{rc}
     8.24\cdot p^{-e_1e_2+d/2}& \mbox{ if $\bfX=\mathbf{O^\varepsilon}$  and $q=p\geq3$} \\
    28.1\cdot2^{-e_1e_2+d/2} & \mbox{ if $\bfX=\mathbf{O^\varepsilon}$  and $q=p=2$}\\
   6.38\cdot 2^{-e_1e_2+d/2} & \mbox{ if $\bfX=\bfSp$  and $q=p=2$.}
\end{array}\right.
\]
\end{enumerate}
\end{proposition}

Recall that $q=p^a$ where $p$ is prime and $a\geq1$. Let $S_n$ act on $Y=(\F_{q^u})^n$ by permuting the basis vectors $y_1,\dots,y_n$ according to the natural representation of $S_n$. The $S_n$-module $Y$ has submodules of dimensions 1 and $n-1$, namely 
\[
W\coloneq\langle y_i-y_{i+1}\mid 1\le i\le n-1\rangle = \left\{ \sum_{i=1}^n x_iy_i \mid \sum_{i=1}^n x_i=0 \right\} \quad \mbox{and}\quad D\coloneq\langle y_1+\cdots+y_n\rangle.
\]
We call $V=(W+D)/D\cong W/(W\cap D)$ the \emph{fully deleted permutation module}. Note that $W\cap D=D$ if $p\mid n$ and $W\cap D=0$ if $p\nmid n$. Let $\delta=1$ if $p\nmid n$ and $\delta=2$ if $p\mid n$ (see \cite[(5.3.6)]{KL}). Then:
\begin{enumerate}
  \item[(PM1)] The following are equivalent: $p\nmid n$,\quad $\delta=1$,\quad
  $Y=W\oplus D$,\quad
  $V\cong W\cong Y/D$;
  \item[(PM2)] The following are equivalent: $p\mid n$,\quad $\delta=2$,\quad
  $D\le W$,\quad $V=W/D$.
\end{enumerate}
Thus $\dim(V)=n-\delta$. The $S_n$-module $V$ is absolutely irreducible, and remains absolutely irreducible when restricted to $A_n$. It can be written over the prime field $\F_p$ of $\F_{q^u}$, and since $d=n-\delta>8$, the normaliser $L\coloneq  N_{\GL(V)}(A_n)= S_n\times Z$, where $Z$ is the subgroup of scalar matrices of $\GL(V)$, see \cite[p. 187]{KL}. We collect together in the next result these and other observations concerning $N_G(A_n)=G\cap L$ when this group is not contained in any subgroup of $\bigcup_{i=1}^8\mathcal{M}_i(\bfX)$. 

\begin{lemma}\label{lem:c9-2} 
Assume that Hypothesis~$\ref{H:AC9}$ holds, and let $V, n,\delta, Z$, $L\coloneq  N_{\GL(V)}(A_n)$ be as above. Then:
\begin{enumerate}[{\rm (a)}]
    \item $d=\dim(V) =n-\delta$ and $\mathrm{(PM1)}, \mathrm{(PM2)}$ hold, and $L\coloneq  N_{\GL(V)}(A_n)= S_n\times Z$;
    \item if $N_G(A_n)=G\cap L$ is not contained in any subgroup of $\bigcup_{i=1}^8\mathcal{M}_i(\bfX)$, then $q^u=p$, and either $\bfX = \bfSp$ with $p=2$, or $\bfX= \bfO^\varepsilon$ for some $\varepsilon\in\{+,-\}$ depending on $p$;
    \item with $G, \bfX$ as in part {\rm(b)},  the set of all subgroups of $\GX_d(p)$ isomorphic to $A_{d+1}$ either forms a single $\GX_d(p)$-conjugacy class, or is the union of two equal sized $\GX_d(p)$-conjugacy classes, and in the latter case $q=p$ is odd.
\end{enumerate}
\end{lemma} 

\begin{proof}
Part (a) follows from the discussion above. Assume now that $N_G(A_n)=G\cap L$ is not contained in any subgroup of $\bigcup_{i=1}^8\mathcal{M}_i(\bfX)$. Then in particular, since the representation of $S_n$ can be written over the prime field $\F_p$ of $\F_{q^u}$, and since $N_G(A_n)=G\cap L$ is not contained in a subgroup in $\mathcal{M}_5(\bfX)$, it follows that $q^u=p$. Also  (see for example \cite[pp.~186--187]{KL}), either $A_n$ preserves a non-degenerate quadratic form on $V$ defined over $\mathbf{F}_{p}$, or $n\equiv 2\pmod{4}$, $p=2$, and $A_n$ preserves a non-degenerate symplectic form on $V$ (but no such quadratic form) defined over $\mathbb{F}_{2}$. Then since $N_G(A_n)=G\cap L$ is not contained in a subgroup in $\mathcal{M}_8(\bfX)$, it follows that either  $\bfX=\bfO^\varepsilon$ for some $\varepsilon=\pm$, or $\bfX=\bfSp$ with $p=2$, and part (b) is proved. 

It remains to prove part (c). By part (a), $n=d+\delta\geq d+1$, and so $n\geq 10$ since $d>8$. Note that if $n=d+2$ then all subgroups $A_{d+1}$ of $A_n$ are conjugate in $A_n$. Then since $n\geq10$, $A_n$ lies in a unique $\GL_d(p)$-conjugacy class by~\cite[Proposition~5.3.5]{KL}, and hence each $A_{d+1}\leq A_n$ also lies in a unique $\GL_d(p)$-conjugacy class of subgroups isomorphic to $A_{d+1}$. 
Now $A_{d+1}\le\Omega_d^\eps(p)$ by~\cite[p.\,187]{KL}, and $A_{d+1}$ lies in a unique
conjugacy class of the conformal orthogonal group $\textup{CO}_d^\eps(p)$ by~\cite[Lemma~1.8.10(ii)]{BHRD}. Since $|\textup{CO}_d^\eps(p):\GO_d^\eps(p)|=p-1$, there is a  unique conjugacy class of subgroups $A_{d+1}$ in $\GO_d^\eps(p)$ when $p=2$.  Also, since $\GO_d^\eps(2)=\SO_d^\eps(2)<\Sp_d(2)$ and $\Sp_d(2)$ has a unique conjugacy class of subgroups $\SO_d^\eps(2)$, it follows that $\Sp_d(2)$ has a unique conjugacy class of subgroups $A_{d+1}$.

Finally suppose that $p$ is odd, and let $H,H^x$ be subgroups of $\GO_d^\eps(p)$ isomorphic to $A_{d+1}$ where $x\in\textup{CO}_d^\eps(p)$. We prove that $H,H^{x^2}$ are conjugate in $\GO_d^\eps(p)$. If $B$ is the Gram matrix of the symmetric form preserved by $\GO_d^\eps(p)$, then $x^TBx=\lambda B$ for some $\lambda\in\F_p^\times$. Hence $(x^2)^TBx^2=\lambda^2 B$ and $x'\coloneq \lambda^{-1}x^2$ satisfies $(x')^TBx'=B$,  so $x'\in\GO_d^\eps(p)$. Since $H^{x'}=H^{x^2}$, and  $\textup{CO}_d^\eps(p)/\GO_d^\eps(p)\cong C_{p-1}$, it follows that there are at most two conjugacy classes of subgroups of $\GO_d^\eps(p)$ isomorphic to $A_{d+1}$, and if there are two classes then they are interchanged by $\textup{CO}_d^\eps(p)$ and so have equal size. 
\end{proof}

Next we restrict the cycle types of  stingray elements  in these subgroups $S_n\times Z$.

\begin{lemma}\label{lem:c9-1}
Using the assumptions of Proposition~$\ref{prop:c9-1}$, suppose also that $V$ is the fully deleted permutation module for $S_n$ as defined above, that $M\coloneq N_G(A_n)$ is an $\Asch_9$-subgroup of $\GX(V)$ such that $M\not\leq L$ for any  $L\in \bigcup_{i=1}^8\mathcal{M}_i(\bfX)$, and that $(g,g')$ is a stingray duo in $g_1^G\times g_2^G$ with $H\coloneq \langle g,g'\rangle\leq M$ and $H$ irreducible on $V$.  Then the following all hold.
\begin{enumerate}[{\rm (a)}]
    \item  $q^u=p$, $d=\dim(V)=n-\delta$, and $\bfX\in\{\bfO, \bfSp\}$, as in Lemma~$\ref{lem:c9-2}$;
    \item If $p=2$, then $H\leq A_{n'}<T$, where either $\bfX=\bfO^\varepsilon$ and $T=\Omega_d^\varepsilon(2)\unlhd G$, or $\bfX=\bfSp$ and $T\cong \Omega_d^\varepsilon(2)$ with $T.2$ a maximal subgroup of $\Sp_d(2)$,   for some $\varepsilon\in\{+,-\}$; in either case the inclusions $A_{n'}<\GO_d^\varepsilon(2)$ are as in the relevant line of  Table~$\ref{t:c9-1}$.
    
    \item  $r_i=e_i+1$ for each $i$, so $d, e_1, e_2$ are all even; also $g$ is an $r_1$-cycle and $g'$ is an $r_2$-cycle.
    \item Let $\Delta, \Delta'$ denote the subset of $\Omega=\{1,\dots,n\}$ of points permuted nontrivially by $g, g'$, respectively. Then $|\Delta\cap\Delta'|=1$, and $H= A_{d+1}$, fixing $\delta-1$ points of $\Omega$. 
\end{enumerate}
\end{lemma}

\begin{table}
  \caption{$\varepsilon, d, n'$  for Lemma~\ref{lem:c9-1}(b) when $q^u=p=2$}
\begin{tabular}{ccll}
  \toprule
  $\varepsilon$ & $d$ & $n'$ & Inclusions of alternating groups \\
  \midrule
$+$ & $0\pmod{8}$  & $d+1$  & $A_{d+1}<\GO^+_d(2)$ \\
               & $6\pmod{8}$  & $d+1, d+2$  & $A_{d+1}<A_{d+2}<\GO_{d}^+(2)$ \\
$-$ & $4\pmod{8}$  & $d+1$  & $A_{d+1}<\GO^-_d(2)$ \\
               & $2\pmod{8}$  & $d+1, d+2$  & $A_{d+1}<A_{d+2}<\GO_{d}^-(2)$ \\
\bottomrule
\end{tabular}
\label{t:c9-1}
\end{table}

\begin{proof}
(a) This part comes from Lemma~\ref{lem:c9-2}. 
Since $|g|, |g'|$, are odd and do not divide $q^u-1=p-1$, it follows that
$H\leq A_n$. Moreover, since  $\bfX\in\{\bfO, \bfSp\}$, the parameters $e_1, e_2$ are even, by Table~\ref{tab:stingraycond}, and hence also $d$ is even.   

(b) We assume that $p=2$ and we refer to \cite[pp.~186--187]{KL}, or see \cite[Section 4]{Liebeck}. 
If $\bfX=\bfSp$, then by Lemma~\ref{lem:spgen} (and Lemma~\ref{lem:redOm}), $H\leq T\cong\Omega_d^\varepsilon(2)$ with $T.2=\SO_d^\varepsilon(2)$ a maximal subgroup of $G$ (for some $\varepsilon\in\{+,-\}$), so $H\leq T\cap N_G(S_n)$. Thus, whether  $\bfX=\bfO$ or $\bfX=\bfSp$, we have, by \cite[pp.~186--187]{KL}, that  either one of the inclusions given in  Table~\ref{t:c9-1} holds with $n'=n$, or
\begin{equation}\label{e:c9-fd}
\bfX=\bfSp,\quad  d=n-2\equiv 0\pmod{4},  \mbox{and\quad $H\leq A_{d+2}<\Sp_{d}(2)$.}
\end{equation}  
However, if \eqref{e:c9-fd} holds, then $A_{d+2}$ preserves no quadratic form on $V$ (see comment after \cite[(5.3.9) on p.~187]{KL}), and so $A_{d+2}$ does not lie in any orthogonal subgroup of $\Sp_d(2)$. Thus $H\cap A_{d+2}$ is a proper subgroup of $A_{d+2}$; in fact $H\cap A_{d+2}=A_{n'}$ with $n'=d+1$ and Line~1 or~3 of Table~\ref{t:c9-1} holds. To clarify what is happening in this case, let $A\cong A_{d+1}$ be the stabiliser in $A_{d+2}$ of the point $d+2$, and let $Y'=\langle y_1',\dots,y_{d+1}'\rangle=\F_2^{d+1}$ denote
the permutation module for $A$ with corresponding fully deleted $A$-permutation module $V'= \langle y_i'-y_{i+1}'\mid 1\leq i\leq d\rangle$, (recall that $p=2$ and $d+1$ is odd). The map $\varphi:Y'\to Y$ given by 
$y_i'\to y_i-y_{d+2}$, for $1\leq i\leq d+1$, is an $A$-module monomorphism, and $\varphi(V')\leq W$ with $\varphi(V')\cap D=0$. Thus $W=\varphi(V')\oplus D$ and $\varphi$ induces an $A$-module isomorphism from $V'$ 
to $V=W/D$. By  \cite[(5.3.9) on p.~187]{KL} (or see Tab;e~\ref{t:c9-1}), $A$ preserves a quadratic form on $V$, and hence $A < \Omega_d^\varepsilon(2) < \Sp_d(2)=G$, with $\varepsilon = +$ if $d\equiv 0\pmod{8}$, and  $\varepsilon = -$ if $d\equiv 4\pmod{8}$. Therefore, if \eqref{e:c9-fd} holds, then $T\cong \Omega^\varepsilon_d(2)$ \emph{for this value of $\varepsilon$}, and up to conjugacy $A_{d+2}\cap T = A$. Also  $H\leq A\leq N_T(A)$, and by Remark~\ref{r:MX-Sp}, $N_T(A)$ is an $\Asch_9$-subgroup for $\bfX= \bfSp$, and  the inclusion on Line~1 or 3 of Table~\ref{t:c9-1} holds. Thus part (b) is proved.

(c) Here we apply results from \cite{PSY}. For $i=1, 2$, let $h=g$ if $i=1$, and let $h=g'$ if $i=2$; also let $r=r_i, e=e_i$ and $e'=e_{3-i}=d-e$, so that $h$ is an $e$-ppd stingray element of $G$ of odd prime order $r$. As noted in Section~\ref{s:sting}, $r=\ell e+1$ for some positive integer $\ell$. Moreover, as a permutation of $\Omega=\{1,\dots,n\}$ under the natural action of $A_n$, $h$ therefore has $c$ cycles of length $r$ and $f$ fixed points, for some $c, f$ such that $n=cr+f$, and hence $h$ has exactly $s\coloneq c+f$ cycles. Note that $s=c+f = n-c(r-1) = n-c\ell e$. By \cite[Lemma 4.3, and Table~4]{PSY} if $p$ divides $n$ (so $d=n-2$), or by \cite[Lemma 4.4]{PSY} if $p$ does not divide $n$ (so $d=n-1$), the fixed point subspace $F_h$ of $h$ in $V$ has dimension $s-\tau$ for some $\tau\in\{0,1,2\}$.  Since $\dim(F_h)=e'=d-e$, and since $d=n-\delta$ by part (a), it follows that $n-\delta -e=  d-e = e' =s-\tau = n-c\ell e-\tau$, and hence 
\[
(c\ell-1)e = \delta-\tau\leq \delta\leq 2. 
\]
Suppose that $c\ell\geq 2$. The above equation then implies that $e\leq 2$, and since $e\geq2$, we conclude that $e=2, c\ell=2$, $\delta=2$, and $\tau=0$. Thus $d=n-\delta=n-2$ and so $p$ divides $n$ by (\textrm{PM2}). The more detailed information in \cite[Lemma 4.3, and Table~4]{PSY} now implies, since $\dim(F_h)=s-\tau=s$, that $p$ is odd and $p$ divides each cycle length of $h$. Thus $p=r$, which is a contradiction since $r$ divides $p^e-1$. Therefore $c\ell=1$ so $c=\ell=1$, and part (c) is proved. 

(d) Let $\Delta, \Delta'$ be as in the statement. By part (c), $d=e_1+e_2=r_1+r_2-2$, so $n=r_1+r_2-2+\delta$. We claim that, if $|\Delta\cap\Delta'|=1$, then $H= A_{d+1}$ and fixes $\delta-1$ points of $\Omega$. 
Assume that  $|\Delta\cap\Delta'|=1$. Then $\Delta\cup\Delta'$ is an $H$-orbit in $\Omega$ of size $|\Delta\cup\Delta'|= r_1+r_2-1= d+1$, and $H$ fixes the remaining $|\Omega\setminus(\Delta\cup\Delta')|=n-(d+1) =\delta-1$ points. In particular $|\Delta\cup\Delta'|\leq \min\{ r_1+r_2-1, n\}$. Note that $r_1=e_1+1\geqslant d/2 +1>5$ (since $d>8$), and hence $|\Delta\cup\Delta'|\geq \max\{r_1, r_2+3\}$. The claim therefore follows from Lemma~\ref{l:natAn}, and it remains to prove that $|\Delta\cap\Delta'|=1$.

Let $\nu\coloneq |\Delta\cap \Delta'|$. Then $n\geq |\Delta\cup\Delta'|=r_1+r_2-\nu$. 
If $|\Delta\cup\Delta'|\leq n-2$, then $H$ fixes at least two distinct points of $\Omega$, say $i$ and $j$, and hence $H$ fixes 
the $1$-subspace $(\langle y_i-y_j\rangle +D)/D$ of $V$, contradicting the assumption that $H$ is irreducible on $V$.
Thus $|\Delta\cup\Delta'|\geq n-1$, and we have 
\[
r_1+r_2-3+\delta = n-1\leq |\Delta\cup\Delta'|= r_1+r_2-\nu\leq n=r_1+r_2-2+\delta
\]
so $2\leq \delta+\nu\leq 3$. Suppose first that $\delta=2$. Then $p$ divides $n$ by (PM2),  $n=r_1+r_2$, and $\nu\leq 1$. If $\nu=0$, then $H= \langle g\rangle\times \langle g'\rangle$ and $H$ leaves invariant the $e_1$-subspace $(W_0+D)/D$, where $W_0=\langle y_i-y_j\mid i,j\in\Delta\rangle$. This is a contradiction, and hence $\nu=1$, and part (d) is proved in this case. Thus we may assume that $\delta=1$. Then $p$ does not divide $n$,  $n=r_1+r_2-1$, and $1\leq \nu\leq 2$. Also $W\cap D=0$ and the permutation module  $\F_p^n=W\oplus D$. If $\nu=2$ then $\Delta\cup\Delta'$ is an $H$-invariant $(n-1)$-subset of $\Omega$, and hence $H$ fixes some point, say $j$, of $\Omega$. This means that, in its action on $\F_p^n$, $H$ fixes the vector $w\coloneq \sum_{i\ne j} (y_i-y_j)$ of $W$. Now $w = v-ny_j$, where $v=\sum_{i=1}^ny_i$ spans the subspace $D$. Note that $ny_j$ is non-zero since $p$ does not divide $n$, and hence $ny_j\not\in D$. This implies that $w\ne0$, 
for if $w=0$ 
then we would have $ny_j=v\in D$. Thus $H$ leaves invariant the $1$-subspace $(\langle w\rangle +D)/D$ of $V$, which is a contradiction. Hence $\nu=1$ and the proof of part (d), and the lemma, is complete.
\end{proof}

We now have all the information needed to prove Proposition~\ref{prop:c9-1}.

\bigskip\noindent \textit{Proof of Proposition~$\ref{prop:c9-1}$.}\quad 
Suppose that $\textup{Prob}_9(g_1,g_2, \bfX; \textup{f.d. perm mod})>0$. Then, by Lemma~\ref{lem:c9-1}(a), 
$q^u=p$, and the type $\bfX= \bfO^\varepsilon$ for some $\varepsilon=\pm$, or $\bfX= \bfSp$ with $q^u=2$. 

If $\bfX= \bfSp$ with $q=2$, then $G=\Sp_d(2)$ and the maximal $\Asch_9$-subgroups we are concerned with are of the form $N_G(S_{d+2})$. By Lemma~\ref{lem:c9-2}(c), $G$ contains a unique conjugacy class of subgroups $A_{d+1}$, and by Lemma~\ref{lem:c9-1}(b), each such subgroup is contained in a maximal subgroup $\SO_d^{\varepsilon}(2)$ of $G$, for some $\varepsilon \in\{+,-\}$ (depending on the congruence of $d$ modulo $8$). Further, by Lemma~\ref{lem:c9-1}(d), for each stingray duo $(g,g')$ in $N_G(S_{d+2})$, the subgroup $H=\langle g,g'\rangle  =A_{d+1}$ (and so $H$ is contained in a subgroup $\SO_d^{\varepsilon}(2)$). Note that we treat the subgroups $N_G(H) = N_G(A_{d+1})$ as possible $\Asch_9$-subgroups  (recall Definition~\ref{e:gen} and Lemma~\ref{lem:spgen}), and we note that there is exactly $x(p)=1$ (a unique) $G$-conjugacy class of them. 

On the other hand if $\bfX= \bfO^\varepsilon$ for some $\varepsilon \in\{+,-\}$ then again, by Lemma~\ref{lem:c9-1}(d), contributions to the probability come from stingray duos $(g,g')$  
such that $\langle g,g'\rangle =A_{d+1}$ and $V$ is the fully deleted permutation module for $A_{d+1}$. Moreover, by Lemma~\ref{lem:c9-2}(c), these subgroups form $x(p)$ equal-sized $G$-conjugacy classes, where $x(p)=1$ if $p=2$, and $x(p)\leq 2$ if $p$ is odd. 

From now on we treat these cases together; and note that the number of $G$-conjugacy classes of these $\Asch_9$-subgroups is $x(p)$  in both cases. 
To estimate the probability, choose a subgroup $M=A_{d+1}$ and an element $g\in g_1^G\cap M$. Note that $N_G(M)=S_{d+1}\times Z'$ where $Z'$ is the subgroup of scalars in $G$, so $|Z'|=\gcd(p-1,2)$. Note also that $N_G(M)$ is self-normalising in $G$ as $S_n$ is complete. By Lemma~\ref{lem:c9-1}(c), each element of $g_1^G\cap M=g_1^G\cap N_G(M)$ is an $r_1$-cycle in $M=A_{d+1}$ and hence $g_1^G\cap N_G(M)=g^{N_G(M)}$. Thus we may apply Lemma~\ref{lem:countM} to the $G$-conjugacy classes $N_G(M)^G$ and $g_1^G$. Therefore, the number of subgroups in $N_G(M)^G$ that contain $g$  is equal to $|C_G(g_1)|/|C_{N_G(M)}(g)|$, and the number of subgroups $M$ of this kind that contain $g$ is $x(p)\cdot|C_G(g_1)|/|C_{N_G(M)}(g)|$. 
Thus, by \eqref{e:asch3}, 

\begin{align}\label{e:c9-2}
   \textup{Prob}_9(g_1,g_2,\bfX;\textup{f.d. perm mod})&   \leq x(p)\cdot \frac{|C_G(g_1)|}{|C_{N_G(M)}(g)|}\cdot 
\frac{|\mathcal{S}_9(g,M)|}{|N(d,p,e_1,\bfX)|} \\ \nonumber
       &  =   x(p)\cdot  \frac{|C_G(g_1)|}{|N(d,p,e_1,\bfX)|}\cdot \frac{|\mathcal{S}_9(g,M)|}{|C_{N_G(M)}(g)|},
\end{align}
with $|N(d,p, e_1,\bfX)|$ as in \eqref{e:ndex}, and  $\mathcal{S}_9(g,M)$ as in \eqref{e:siM},  the set of all $g'\in g_2^G\cap M$ such that $(g,g')$ is a stingray duo and $\langle g,g'\rangle$ is not contained in a subgroup of $\cM_j(\bfX)$ for any $j\leq 8$. 
The first fraction in the last expression in \eqref{e:c9-2} is determined by Proposition~\ref{p:Xbds}, namely
\[
\frac{|C_G(g_1)|}{|N(d,p,e_1,\bfX)|}\leq  
 \frac{(p^{e_1/2}+1)(p^{e_2/2}+1)}{\mathbf{k}(d,p, e_1,\bfX)\cdot|\cU(d,p,e_1,\bfX)|}
 \leq  \frac{(p^{e_1/2}+1)(p^{e_2/2}+1)}{(1-3/(2p))\cdot p^{e_1e_2}\cdot \bfa(p,\bfX)}
\]
with $\bfa(p,\bfX)$ as in Table~\ref{t:abLU}, which we write also in Table~\ref{t:asch9-perm}.

\noindent 
To estimate the second fraction in \eqref{e:c9-2} we note first that 
\[
|C_{N_G(M)}(g)|=|Z'|\,|C_{S_{d+1}}(g)|= \gcd(p-1,2)\cdot r_1\cdot (d+1-r_1)!\geq  x(p)\cdot  r_1\cdot e_2!.
\]
We estimate $|\mathcal{S}_9(g,M)|$ as follows. As in Lemma~\ref{lem:c9-1}(d), $g$  has support $\Delta$, an $r_1$-subset of $\Omega'=\{1,\dots, d+1\}$.  Each of the elements $g'$ has support
$\Delta'=(\Omega'\setminus \Delta)\cup \{j\}$ for some $j\in\Delta$. There are $r_1$ choices for $j$ and, given $j$, there are $(r_2-1)! = e_2!$ choices for the $r_2$-cycle $g'$ with support $\Delta'$. Thus $|\mathcal{S}_9(g,M)|\leq r_1 \cdot e_2!$. Therefore,
\begin{align*}
 \textup{Prob}_9(g_1,g_2,\bfX;\textup{f.d. perm mod}) &\leq x(p)\cdot
 \frac{(p^{e_1/2}+1)(p^{e_2/2}+1)}{(1-3/(2p))\cdot p^{e_1e_2}\cdot \bfa(p,\bfX)} \cdot \frac{r_1 \cdot e_2!}{x(p)\cdot r_1\cdot e_2!}\\
 &< 
 \frac{y(p)}{(1-3/(2p))\cdot  \bfa(p,\bfX)} \cdot p^{-e_1e_2+d/2}
\end{align*}
where, noting that $e_1+e_2=d$ and $2\leq e_2 = d-e_1\leq d/2$, $y(p)$ is the maximum of $f(x)=(1+p^{-(d-x)/2} )(1+p^{-x/2})$ over all $x=e_2\in [2, d/2]$ and $d\geq10$. Since $f$ has a unique maximum of $(1+p^{-(d-2)/2})(1+p^{-1})$ in this interval,  at $x=2$, we have $y(p) = (1+p^{-(d-2)/2})(1+p^{-1})$ which is at most $y(2)\leq 51/32<1.6$ (with $d=10$), and for $p$ odd is at most $y(3)<1.35$  The upper bounds in Proposition~\ref{prop:c9-1} now follow from these estimates, as recorded in Table~\ref{t:asch9-perm}.

  \begin{table}
\caption{Bounds for Proposition~\ref{prop:c9-1}}
\begin{tabular}{lcccc}
  \toprule
$\bfX$ & $p$ & $\bfa(p,\bfX)$ & $y(p)$ & $y(p)/((1-3/(2p)) \bfa(p,\bfX))$\\
    \midrule
$\bfSp$              & $2$   & $1$    & $<1.6$&$< 6.38$\\
$\mathbf{O^\varepsilon}$   & $2$   & $5/22$ & $<1.6$&$< 28.1$\\
$\mathbf{O^\varepsilon}$   & odd   & $20/61$& $<1.35$&$< 8.24$\\
\bottomrule
\end{tabular}
\label{t:asch9-perm}
\end{table}

\subsection{Estimation of  \texorpdfstring{$\textup{Prob}_9^*(g_1,g_2,\bfX)$}{}.} \label{sub:c9-diff}
In this subsection, we prove that the probability $\textup{Prob}_9^*(g_1,g_2,\bfX)=0$. Since $d>8$, this follows immediately from a theorem \cite[Theorem 1.2, Table 1]{GNP4} of Zalesski and the authors for the case $e_1=e_2=d/2$, and so we assume that $e_2<e_1$ and hence that $e_2<d/2<e_1$. Thus $M$ is an $\Asch_9$-group containing an $e_1$-ppd element $g$ with $e_1>d/2$,  and hence $M$ is restricted by the classification given by~\cite[Main Theorem]{GPPS} and must be one of the groups in~\cite[Example 2.6--2.9]{GPPS}.
Recall that, for the contribution to $\textup{Prob}_9^*(g_1,g_2,\bfX)$,  the simple group $S$ involved in $M$ is not $A_n$ acting on its fully deleted permutation module. We prove in Proposition~\ref{lem:c9-3} 
that none of the remaining groups in~\cite[Example 2.6--2.9]{GPPS} contains a stingray duo $(g,g')$, and hence that $\textup{Prob}_9^*(g_1,g_2,\bfX)=0$.
  Our method of proof is to apply the classification in~\cite[Main Theorem]{GPPS} to deduce first that, in all cases, $e_1= d-3$ or $d-2$, and hence $e_2=2$ or $3$. We then apply a representation theoretic result~\cite[Theorem~2.11]{BG}, which restricts further the possibilities for $M$. We make an important remark about our application of the latter result.  

\begin{remark}\label{rem:emu}
The result~\cite[Theorem~2.11]{BG}  involves a quantity $\nu$ defined on cosets of the scalar subgroup $Z=Z(\GL_d(q^u))$, see \cite[(7)]{BG}. We  use this quantity for the coset $g'Z$ with $g'$ our $e_2$-ppd stingray element, and for this element $\nu(g'Z)$ is equal to $e_2$. 
Since $d\ge9$, \cite[Theorem~2.11]{BG} implies that either the simple group $S$ involved in $M$ belongs to one of two lists of groups, called List $\mathcal{A}$ and List $\mathcal{B}$ in~\cite[Tables~1,\,2]{BG}, or
$e_2>\max\{2,\sqrt{d}/2\}$.
Moreover, it can easily be checked that the latter inequality fails if either (i) $e_2=2$, or (ii) $e_2=3$ and $d\geq 36$. Note that the simple groups $S$ in List $\mathcal{A}$ are all alternating groups.
\end{remark}

\begin{proposition}\label{lem:c9-3}
Assume that Hypothesis~$\ref{hyp}$ holds with $d>8$. 
Then no maximal $\Asch_9$-group   contains a stingray duo in $g_1^G\times g_2^G$,  apart from those with a normal subgroup $A_n$ acting on its fully deleted permutation module. Consequently $\textup{Prob}_9^*(g_1,g_2,\bfX)=0$. 
\end{proposition}

\begin{proof}
As mentioned above, if $e_1=e_2=d/2$, then the result follows from \cite[Theorem 1.2, Table 1]{GNP4}, which shows that no $\Asch_9$-subgroup contains a $(d/2)$-stingray element when $d>8$, apart from the case of $S=A_n$ on its fully deleted permutation module. Thus we assume that $e_1>d/2>e_2\geq 2$.   Suppose that $M$ is a maximal $\Asch_9$-group such that $M$ contains a stingray duo  $(g,g')$ from $g_1^G\times g_2^G$, and $M$ does not have a normal subgroup $A_n$ acting on its fully deleted permutation module.
As discussed above, $M$ must be  one of the groups in~\cite[Example 2.6--2.9]{GPPS}.

\medskip\noindent
\textit{Groups from~\cite[Example 2.6]{GPPS}.}\quad Here $S=A_n$ and $n, d, p, r_1$ must be as in one of the lines of~\cite[Tables 2, 3, or 4]{GPPS}.  However, since $M$ is not acting on the fully deleted permutation module for $S$, there are no possibilities with $d\geq  9$. 

\medskip\noindent
\textit{Groups from~\cite[Example 2.7]{GPPS}.}\quad Here  the simple group $S$ is a sporadic group, and the derived subgroup $M', d, p, e_1, r_1$ must be as in one of the lines of~\cite[Table 5]{GPPS}. For convenience we list in Table~\ref{tab:GPPSTab5} the five triples $d, e_1, M'$ from this table which satisfy $d\geq 9$ and $d/2< e_1\le d-2$. Observe that in each case  $e_1=d-2$ and so $e_2=2$. Thus by~\cite[Theorem~2.11]{BG} and our discussion in Remark~\ref{rem:emu}, the
group $S$ must appear in List $\mathcal{B}$ in~\cite[Table 2]{BG}. The only sporadic group $S$ occurring in~\cite[Table 2]{BG} and also in Table~\ref{tab:GPPSTab5} is the group $S={\rm M}_{12}$. However, in~\cite[Table 2]{BG}, $\dim(V)=6$, while in Table~\ref{tab:GPPSTab5}, $d=\dim(V)=12$. Thus there are no examples from~\cite[Example 2.7]{GPPS}.

\begin{table}
\caption{Values of $d$, $e_1$, $M'$ from~\cite[Table~5]{GPPS}}
\begin{tabular}{rccccc}
  \toprule
  $d$   &  $12$ & $20$ &$18$ &$24$ &$12$\\
  $e_1$ & $10$& $18$& $16$& $22$& $10$\\
  $M'$  & $2.{\rm M}_{12}$ & ${\rm J}_1$ & $3.{\rm J}_3$ & $2.{\rm Co}_1$&$6.{\rm Suz}$\\
\bottomrule
\end{tabular}
\label{tab:GPPSTab5}
\end{table}

\medskip\noindent
\textit{Groups from~\cite[Example 2.8]{GPPS}.}\quad Here $S$ is a simple group of Lie type in characteristic $p$. In this case the subgroup $M^{(\infty)}$ together with $ d$ and $e_1$ must be as in one of the lines of~\cite[Table 6]{GPPS}. There are exactly two lines with $d>8$; each has $d=9$ and $e_1=6$ and hence $e_2=3$, and the subgroup $M^{(\infty)}$ is either $\SL_3(q_0^2)$ or $\PSL_3(q_0^2)$, where $q=q_0^k$ for some $k\ge1$. We argue that such a group $M$ does not contain a stingray duo $(g,g')$. By assumption $V$ is an absolutely irreducible $\mathbb{F}_qM^{(\infty)}$-module realised over no proper subfield of $\mathbb{F}_q$. Applying \cite[Proposition 5.4.6(i)]{KL} with $q=p^f$ and $q_0^2=p^e$ we find that either (a) $f=e$ and $q=q_0^2$, or (b) $f=e/2$ and $q=q_0$. Since $g\in M^{(\infty)}$, it follows that $r_1=|g|$ divides $q_0^{2i}-1$ for some $i\leq 3$. However, $r_1$ is a primitive prime divisor of $q^{e_1}-1=q^6-1$, and so (a) cannot hold (as $r_1$ would then divide $q^i-1$ with $i\leq 3$), and we conclude that $q=q_0$. Then  \cite[Proposition 5.4.6(i)]{KL} implies further that $V$ is a twisted tensor product $U\otimes  U^{(q)}$, where $U$ is the natural module $U=(\mathbb{F}_{q^2})^3$ of $M^{(\infty)}$, and an  element $x \in M^{(\infty)}$ acts on $V$ as $x \otimes x^{\sigma}$, where $\sigma$ is a field automorphism  of order $2$ (coming from the map $\lambda \mapsto \lambda^{q}$ on $\mathbb{F}_{q^2}$).
The element $g'$ lies in $M^{(\infty)}$ and  $r_2=|g'|$ is a primitive prime divisor of $(q^2)^3-1$, since $r_2$ is a primitive prime divisor of $q^3-1$ and so does not divide $q^2-1$ or $(q^2)^2-1$. Thus $g'$ is an irreducible element, and hence, in particular, a regular semisimple element of $\GL(U)$. Therefore the eigenspaces of $g'$ on $U$ (working over an algebraic closure of $\mathbb{F}_{q^2}$) are all $1$-dimensional. The same is true for the action of $(g')^{\sigma}$ on $U^{(q)}$. It follows from (the semisimple case of) \cite[Lemma 3.7]{LieSh} (or see \cite[Lemma 5.4.2]{BGbook}) that the maximum dimension of an eigenspace for $g'\otimes (g')^{(q)}$ acting on $V$ is $3$.   This a contradiction since, in its action on $V$, $g'$ has fixed point space of dimension $e_1=6$.  Thus there are no examples from~\cite[Example 2.8]{GPPS}.

\medskip\noindent
\textit{Groups from~\cite[Example 2.9]{GPPS}.}\quad Finally we consider  simple groups $S$ of Lie type in characteristic different from $p$. 
In this case either $M'$, $d$ and $e_1$ are as in one of the lines of~\cite[Table 7]{GPPS} (examples from individual groups), or $S, d, e_1, r_1$ are as in one of the lines of~\cite[Table 8]{GPPS} (examples from infinite families of simple groups). 
From ~\cite[Table 7]{GPPS} we find exactly three possibilities with $d\geq9$ and $e_1\le d-2$, namely $M'$ is ${}^2{\rm B}_2(8)$, ${\rm G}_2(3)$, or $\Sp_4(4)$ with $d= 14, 14$, or $18$, respectively, and in each case $e_1=d-2$ so $e_2=2$.  Thus by~\cite[Theorem~2.11]{BG} and our discussion in Remark~\ref{rem:emu}, the
group $S$ must appear in List $\mathcal{B}$ in~\cite[Table 2]{BG}. However, the entries in~\cite[Table~2]{BG} all have $d\le 10$, so we have no examples from~\cite[Table~7]{GPPS}.

\begin{table}
\caption{$d$, $e_1$, $S$ from~\cite[Table~8]{GPPS}}
\begin{tabular}{rcccccc}
  \toprule
  \textup{line \#} && 4 &    5 & 5 &7 & 9\\
  $d$       && $\frac{3^n+1}{2}$ & $s+1$ &$s$ & $s+1$&$\frac{s+1}{2}$ \\
  $e_1$&&$d-2$&$d-3$&$d-2$&$d-2$&$d-2$\\
  $S$     &&  $\textup{PSp}_{2n}(3)$ & $\PSL_2(s)$ &$\PSL_2(s)$ &$\PSL_2(s)$ &$\PSL_2(s)$ \\
  $s$&& 3&$s=2^\textup{prime}\ge7$&$s=2^\textup{prime}\ge7$&$\textup{prime}$&$\textup{odd}$\\
\bottomrule
\end{tabular}
\label{tab:GPPSTab8}
\end{table}

Finally, we consider the infinite families $S$ and possibilities for $d, e_1, r_1$ from~\cite[Table 8]{GPPS} with $d\geq9$ and $e_1\leq d-2$. This gives rise to the values of $S, d, e_1$ listed in Table~\ref{tab:GPPSTab8}. In each case $r_1=e_1+1$ and  $e_1=d-2$ or $d-3$, so $e_2=2$ or $3$. Thus by~\cite[Theorem~2.11]{BG} and our discussion in Remark~\ref{rem:emu}, if $e_2=2$, or if $e_2=3$ and $d\geq 36$,  then the group $S$ must appear in List $\mathcal{B}$ in~\cite[Table 2]{BG}. However, the only entry in~\cite[Table 2]{BG} for a group with $d\geq9$,  occurs in the first row of ~\cite[Table 2]{BG} with $d=10$ and $S=\PSU_5(2)$, and this group $S$ does not occur in Table~\ref{tab:GPPSTab8}. Thus we have $e_2=3$, $e_1=d-3$, and $9\leq d\leq 35$. The only entry in Table~\ref{tab:GPPSTab8} with $e_1=d-3$ is for $S=\PSL_2(s)$ with $d=s+1$ and $s=2^c\ge7$ with $c$ a prime. Since $r_1=\ell e_1+1\geq e_1+1=d-2=s-1$ and $r_i$ divides $|S|$, we must have $r_1=e_1=d-2=s-1$, and since $9\leq d\leq 35$ it follows that $s=8$ or $32$. Now $e_2=3$ and so $r_2=|g'|$ is a primitive prime divisor of $(q^u)^3-1$, and $r_2=|g'|=3k+1$, for some $k$. In particular, $g'$ is not a scalar, and $g'\in M^{(\infty)}$, so $r_2$ divides $|S|$.  Also since $r_1=s-1$  is a primitive prime divisor of $(q^u)^{e_1}-1$ and $e_1=d-3>3=e_2$, it follows that $r_1\ne r_2$. Thus $r_2$ is a prime divisor of $|S|$, distinct from $s-1$, and $r_2\equiv 1\pmod{3}$. However, $|\PSL_2(s)|=2^3\cdot3^2\cdot7$ or $2^5\cdot3\cdot11\cdot31$ for $s=8$ or $32$, respectively, and there is no such prime $r_2$.  Thus there are no examples from~\cite[Example 2.9]{GPPS}.

We have therefore proved that there are no possibilities for $M$, and  hence we have $\textup{Prob}_9^*(g_1,g_2,\bfX)=0$, completing the proof.
\end{proof}

\medskip\noindent
\textit{Proof of Proposition~$\ref{prop:c9}$.}\quad 
By Proposition~\ref{lem:c9-3}, 
\[
\textup{Prob}_9(g_1,g_2,\bfX)=\textup{Prob}_9(g_1,g_2, \bfX; \textup{f.d. perm mod}).
\]
Now all the assertions of Proposition~\ref{prop:c9} follow immediately from Proposition~\ref{prop:c9-1}.

\section{Proofs of Theorems~\ref{thm:appn}
and~\ref{thm:main} }\label{s:proofs}

In this section we prove both  Theorem~\ref{thm:appn}
and Theorem~\ref{thm:main}. Since Theorem~\ref{thm:appn} depends on Theorem~\ref{thm:main}, we focus first on the latter result, and give the proof of Theorem~\ref{thm:appn} in Subsection~\ref{ss:thm1.1}. 
We draw together results from the previous sections, and although some of these results hold for smaller $d$, we assume that $d>8$. We use \eqref{e:rhonongen}, namely,
\begin{equation*}
    \rhonongen(g_1,g_2,\bfX)  = \sum_{i=1}^9 {\rm Prob}_i(g_1,g_2,\bfX)
\end{equation*}
together with the upper bounds for the probabilities ${\rm Prob}_i(g_1,g_2,\bfX)$ given in Propositions~\ref{prop:redL}, \ref{prop:redUSO}, \ref{prop:imprim}, \ref{prop:extn}, \ref{p:tensor}, \ref{prop:sub}, \ref{prop:c6}, \ref{prop:classicalmax},
and \ref{prop:c9}. First we prove a crucial proposition that deals with the Aschbacher classes individually, and is used in the proof of Theorem~\ref{thm:main} given in Section~\ref{ss:thm1.2}.

\begin{table}
\footnotesize
\begin{tabular}{lrr|rr|rr|rr|rr}
  \toprule
$i\backslash \bfX$ & $\bfL$ & &$\bfU$ & &$\bfSp$ & & $\mathbf{O^+}$& &$\mathbf{O^-}$&\\
    \midrule
    & $q=2$ & $q\ge3$ & $q=2$ & $q\ge3$ & $q=2$ & $q\ge 3$ & $q=2$ & $q\ge3$ & $q=2$ & $q\ge 3$ \\
    \midrule
$\Asch_1$   & 0
& 0
&  $7.5\cdot 10^{-11}$
&  $  3.3\cdot 10^{-17}$ &0.008 &$1.02\cdot10^{-4} $& 0.035 &$ 3.1\cdot 10^{-4}$ & 0 & 0\\
$\Asch_2$ & 0 & $1.4\cdot 10^{-6}$ & $9.9\cdot 10^{-7}$
& $1.3\cdot 10^{-10}$ & 0&0 & 0 & 0.18 & 0 &0.18\\
$\Asch_3$ & $0.0081$ & $0.00032$ &$4.3\cdot 10^{-6}$ & $1.5\cdot 10^{-8}$ & $0.074$ & $0.6$& 
 $0.279^{\dagger}$ & $0.017^{\dagger}$ &
$0.086$ & $0.006$  \\
$\Asch_5$ & 0 & $0.016$ & 0 & $6\cdot 10^{-14}$& 0 & 0.91&0 &1.24 & 0 &2.73\\
$\Asch_6$ & 0 & $2.1\cdot 10^{-26}$ & 0&0 & 0 & $0.0053^{\dagger}$ & 0&$0.0053^{\dagger}$
&0  &0\\
$\Asch_8$ & 0.092 & $0.04$ & 0& 0& 0& 0& 0& 0& 0 &0\\
$\Asch_9$ & 0 & $0$ & 0& 0&0.399 &0 &$1.76^{\dagger}$ & 0.102& $1.76^{\dagger}$ & 0.102\\
\bottomrule\\[.1cm]
\end{tabular}
\caption{Values of $p(i,\bfX)$ for Proposition~\ref{prop:main}, for $i\ne 4,7$. 
For entries with a $\dagger$ see Remark~\ref{r:dagger}.}
\label{t:proofex}
\end{table}

\begin{remark}\label{r:dagger}
In Table~\ref{t:proofex}, the following entries were marked with a  $\dagger$.
\begin{enumerate}
    \item[(a)] Class  $\Asch_3$ in type $\bfO^+$. The bound in the table only applies when $e_2\geq3$; in the case where 
    $e_2=2$, the bound is $12.81$  for $q=2$, $2.005$ for $q\geq 4$, and we note that $q\ne 3$ here. 
    \item[(b)] Class  $\Asch_6$ in types $\bfSp$ or $\bfO^+$ with $q\geq 3$. The bound in the table applies 
    in all cases except when $d= 32$, $e_2=2$,  and $ q\in \{11,13,17\}$;  in these cases the bound is $6.9$.
    \item[(c)] Class  $\Asch_9$ in type $\bfO$ with $q=2$. By Lemma~\ref{lem:c9-1}(b) (see also Table~\ref{t:c9-1}), $p(9,\bfO^+)=0$ when $d\equiv 2, 4\pmod{8}$, and  $p(9,\bfO^-)=0$ when $d\equiv 0, 6\pmod{8}$. 
\end{enumerate}
Note that there are some some cases where we know that $p(i,\bfX)=0$, but which are not explicitly identified in Table~\ref{t:proofex} (for example, 
for $\bfX\ne \bfL$, $p(5,\bfX)=0$ if $q<8$). For full details we refer readers to the main propositions for each Aschbacher class $\Asch_i$ in the previous sections. 
\end{remark}

\begin{proposition}\label{prop:main}
Suppose that $G$ is a classical group of type $(\bfX, d, |\F|)$ on $V$ as in Table~$\ref{tab:G}$, and that $e_1, e_2$ are integers  such that $2\leq e_2\leq e_1$ and $d=e_1+e_2>8$.  If, for $j=1,2$, $g_j$ is an $e_j$-ppd stingray element in $G$, and if $i\in\{1,2,\dots,9\}$, then,
\[
{\rm Prob}_i(g_1,g_2,\bfX)\le \lambda_\bfX\cdot (q^{-1}+q^{-2}) + p(i,\bfX)\cdot q^{-d+3}
\]
 where $\lambda_\bfX=1$ if $(i,\bfX)=(1,\bfL)$ and $\lambda_\bfX=0$ otherwise, and where $p(i,\bfX)=0$ if $i\in\{4,7\}$, and otherwise $p(i,\bfX)$ is as in line $\Asch_i$ of Table~$\ref{t:proofex}$ for type $\bfX$.
\end{proposition}

The proof of Proposition~\ref{prop:main} uses Remark~\ref{rem:fx} a number of times.

\begin{remark}\label{rem:fx}
In Propositions~\ref{prop:imprim} to~\ref{prop:c9}, many of
the expressions involve terms of the form $q^{-e_1e_2}$. Finding upper bounds requires
setting
$e_2=d-e_1$ to be the integer maximising
a function $f_Q(x)$ over $x\in [c,d/2]$, where $2\le c\le d/2$, $d>8$, and $Q\ge 2$ is a  fixed integer. We note that for a fixed $Q$, the  functions 
$f_Q(x) = x(d-x)Q^{-x(d-x)}$,
and $g_Q(x) = Q^{-x(d-x)}$
are decreasing
for $x\in [c,d/2]$.
Thus their maximum value is at $x=c$.

\end{remark}

\begin{proof}
We prove the result separately for  each Aschbacher class $\Asch_i$, and each type $\bfX$. Observe that, for type $\bfL$ we have $d\ge 9$ and $e_2\ge 2$;
for type $\bfU$ we have $d\ge 10$ and $e_2\ge 3$;
and for all other types 
$d\ge 10$ and $e_2\ge 2, $ see Table~\ref{tab:one}. For the cases $i\in\{4,7\}$, ${\rm Prob}_i(g_1,g_2,\bfX)=0$ by Proposition~\ref{p:tensor}, and so the result holds with $p(i,\bfX)=0$. We assume from now on that $i\not\in\{4,7\}$.

We pursue the following strategy, which applies to all cases except $\Asch_1$ with $\bfX=\bfL.$ For each type $\bfX$
and each Aschbacher class $\Asch_i$, we first determine an upper bound for
${\rm Prob}_i(g_1,g_2,\bfX)$ of the form $c_{\bfX} \cdot q^{-f(d)} \cdot q^{-d+3},$ where 
$c_{\bfX}$
is a constant depending on $\bfX$ and $f(x)$ is a polynomial. We then bound  $q^{-f(d)}$
by choosing a lower bound for $q,$ usually $2$ or $3$, and an upper bound for $f(d)$ for $d\ge 9$, usually $f(9)$ or $f(10).$

\medskip\noindent
\emph{{Case} $\Asch_1$.} \quad
For  $\bfX=\bfL$ we have
${\rm Prob}_1(g_1,g_2,\bfL)\le
q^{-1}+q^{-2}$ by Proposition~\ref{prop:redL}, so the result holds with $a_1=1$ and $p(1,\bfL)=0$.  
For $\bfX\neq \bfL$,
upper bounds
for ${\rm Prob}_1(g_1,g_2,\bfX)$
are given in  Proposition~\ref{prop:redUSO}, namely ${\rm Prob}_1(g_1,g_2,\bfO^-)=0$, and otherwise the bounds involve a constant $c_{\bfX}$ depending on $\bfX$ and $q$.
For $q=2$ we have
$c_{\bfU}=2.56$,
$c_{\bfSp}=4$, and
$c_{\bfO^+}=17.6$,
and for $q\ge 3$ we have
$c_{\bfU}=81/50$, 
$c_{\bfSp}=2$, and 
$c_{\bfO^+}=6.1.$
For
 $\bfX=\bfU$ 
\[{\rm Prob}_1(g_1,g_2,\bfU)\le c_{\bfU} \cdot q^{-2e_1e_2} 
\le c_{\bfU} \cdot q^{-6(d-3)} 
= c_{\bfU}\cdot q^{-5d+15} q^{-d+3}.\]
For $\bfX=\bfSp, \bfO^+$
\[{\rm Prob}_1(g_1,g_2,\bfX)\le c_{\bfX}\cdot  q^{-e_1e_2} 
\le c_{\bfX} \cdot q^{-2(d-2)} 
= c_{\bfX}\cdot q^{-d+1} q^{-d+3}.\]
The constant $p(i,\bfX)$ listed
in Table~\ref{t:proofex} is an upper
bound for  $c_{\bfX}\cdot q^{-5d+15}$,
when $\bfX=\bfU$ and an upper bound for $c_{\bfX}\cdot q^{-d+1}$
when $\bfX=\bfSp, \bfO^+$,
where we use $d\ge 10$, and we treat $q= 2$ and $q\ge 3$ separately.

\medskip\noindent
\emph{{Case} $\Asch_2$.} \quad
By Proposition~\ref{prop:imprim}, 
${\rm Prob}_2(g_1,g_2,\bfX)= 0$ 
if (i) $\bfX=\bfSp$, or 
(ii) $\bfX=\bfO^\varepsilon$ with $q$ even, or
(iii) $q^u=2$, or (iv) $d$ is odd. The result holds in all these cases, so we now 
consider all other cases. In particular we now have $q^u\geq3$ and  $d$ even, so $d\ge 10.$
Also $\bfX=\bfL$, $\bfU$ or $\bfO^\varepsilon$, and we  set  $c_{\bfL}=1/2$, $c_{\bfU}=1.46$ and $c_{\bfO^\varepsilon}=4.12$; and also 
 $Q_{\bfX}=q^2$ for $\bfX=\bfL$, $\bfU$, and $Q_{\bfO^\varepsilon}=q$. 
 Then  ${\rm Prob}_2(g_1,g_2,\bfX)\le c_{\bfX} \cdot e_1 e_2\cdot  Q_{\bfX}^{-e_1e_2 +d/2}$, 
by Proposition~\ref{prop:imprim}.
 We will apply Remark~\ref{rem:fx}. 
For this we set $E_{\bfL}=2$, $E_{\bfU}=3$, and $E_{\bfO^\varepsilon}=2$ unless $q=3$ where we take $E_{\bfO^\varepsilon}=3$ (see Remark~\ref{rem:ppd}(a)).
Then by Remark~\ref{rem:fx}, 
 $f_{Q_{\bfX}}(x) = x(d-x){Q_{\bfX}}^{-x(d-x)}
$ is decreasing for $x\in [E_{\bfX},d/2]$. This yields
\begin{align*}
{\rm Prob}_2(g_1,g_2,\bfX)
&\le 
c_{\bfX} \cdot e_1 e_2\cdot  Q_{\bfX}^{-e_1e_2 +d/2}
\le  c_{\bfX}\cdot 
{E_{\bfX}(d{-}E_{\bfX})}\cdot Q_{\bfX}^{-E_{\bfX} (d{-}E_{\bfX}) +d/2}.
\end{align*}
Suppose first that $\bfX=\bfL$. Then $q=q^u\ge 3$, and we note that  $y\le q^{y/3}$ if $q\ge 3$ and $y\ge 7$. Hence $d-2\leq q^{(d-2)/3}$ since $d-2> 7$, and so
\begin{align*}
{\rm Prob}_2(g_1,g_2,\bfL)
 &\le c_{\bfL}\cdot 
{2 (d{-}2)}\cdot q^{-4(d{-}2)+d}
< q^{(d-2)/3}\cdot q^{-3d{+}8}\\
&=  q^{(-5d+13)/3}\cdot q^{-d+3} < 1.4\cdot 10^{-6} q^{-d+3},
\end{align*}
as in Table~\ref{t:proofex}.
Now consider $\bfX=\bfU$ and recall that here  $e_2\ge 3$.
 Note that $y\le q^{5y/12}$ if $q\ge 2$ and $y\ge 7$. Then since $d\geq10$ we have 
 $d-3\le q^{5(d-3)/12}$ and hence
 \begin{align*}
 {\rm Prob}_2(g_1,g_2,\bfU) &
\le
 c_{\bfU} \cdot 3(d{-}3)\cdot q^{-6(d{-}3)+d}\\ 
&\le 3\cdot c_{\bfU} \cdot q^{5(d-3)/12}\cdot q^{-6(d{-}3)+d} = 4.38
\cdot q^{(-43d+165)/12}\cdot q^{-d+3}.
\end{align*}
Inserting $q=2$ and $q=3$
yields the lower bounds in Table~\ref{t:proofex}. 
Finally consider
$\bfX=\bfO^\varepsilon$ with $q$ odd
and $d\ge 10$. Suppose first that $q\ge5$, and note that  $(d-2) < q^{(d-2)/5}$.   Hence 
 \begin{align*}
{\rm Prob}_2(g_1,g_2,\bfO^\varepsilon)
&\le 4.12\cdot 2(d-2) \cdot q^{-2(d-2) + d/2}<  8.24 \cdot q^{(d-2)/5}\cdot q^{-2(d-2) + d/2}\\
&= 8.24 \cdot q^{-3d/10 +3/5} q^{-d+3}\le 0.18\cdot q^{-d+3}
\end{align*}
as in Table~\ref{t:proofex}.  It remains to consider the case $q=3$.
Here $e_2\ge 3$ by Remark~\ref{rem:ppd}(a), and we note that  $(d-3) < 3^{(d-2)/4}= q^{(d-2)/4}$. Therefore
 \begin{align*}
{\rm Prob}_2(g_1,g_2,\bfO^\varepsilon)
&\le 4.12\cdot 3(d-3) \cdot q^{-3(d-3) + d/2}\le  12.36\cdot  q^{(d-2)/4}\cdot q^{-3(d-3) + d/2}\\
&=  12.36\cdot  q^{-5d/4 +11/2} q^{-d+3} < 0.0057\cdot q^{-d+3}\ll 0.18\cdot q^{-d+3}
\end{align*}
as in Table~\ref{t:proofex}.

\medskip\noindent
\emph{{Case} $\Asch_3$.} \quad
Proposition~\ref{prop:extn} identifies the cases where 
$ {\rm Prob}_3(g_1,g_2,\bfX)=0$ (see Table~$\ref{t:Prob3Zero}$) and shows that 
  otherwise
 \[
 {\rm Prob}_3(g_1,g_2,\bfX) < c_{\bfX} \cdot q^{-e_1e_2/\alpha},
 \]
 where $c_{\bfX}$ and $\alpha\in\{1,2\}$ are given in Table~$\ref{t:Prob3Res}$.
Note that for $\bfX = \bfL$ we have $\alpha=1$,  $e_2\ge 2$, and $d\ge 9.$ Using Remark~\ref{rem:fx}, 
\[
{\rm Prob}_3(g_1,g_2,\bfL)
\le  c_{\bf L} \cdot q^{-2(d-2)} = c_{\bfL}\cdot
q^{-d+1} q^{-d+3}.
\]
If $q=2$ then as $d\ge 9$, we have $c_{\bfL}\cdot
q^{-d+1} q^{-d+3} < c_{\bf L}\cdot 2^{-8}\cdot q^{-d+3}$, while if $q\geq3$ this becomes $c_{\bfL}\cdot
q^{-d+1} q^{-d+3} <  c_{\bfL} \cdot 3^{-8}\cdot q^{-d+3}$, as in Table~\ref{t:proofex}. 

For type  $\bfU$, we have $\alpha=1$,  $e_2\ge 3$, and $d\ge 10.$ Using Remark~\ref{rem:fx}, 
\[
{\rm Prob}_3(g_1,g_2,\bfU)
\le  c_{\bfU} \cdot q^{-3(d-3)} = c_{\bfU}\cdot
q^{-2d+6} q^{-d+3}.
\]
If $q=2$ then as $d \ge 10$, we have $c_{\bfU}\cdot
q^{-2d+6} q^{-d+3} < c_{\bfU}\cdot 2^{-14}\cdot q^{-d+3}$, while if $q\ge3$ then $ c_{\bfU}\cdot
q^{-2d+6} q^{-d+3} < c_{\bfU}\cdot 3^{-14}\cdot q^{-d+3}$, as in Table~\ref{t:proofex}.

Now consider the cases $\bfX=\bfSp$ or $\bfX=\bfO^\epsilon$. Then by  Table~\ref{t:asch3}, $d\ge 12$  and one of
the lines of that table holds and by Table~\ref{t:Prob3Res}, $\alpha=2$. 
In the cases in Table~\ref{t:asch3} where $e_2\ge 4$ we find
\[
{\rm Prob}_3(g_1,g_2,\bfX) 
\le c_{\bfX} \cdot q^{-e_1e_2/2} \le
c_{\bfX} \cdot  q^{-2(d-4)} = 
c_{\bfX} \cdot
q^{-d+5} q^{-d+3}\le 
c_{\bfX} \cdot
q^{-7} q^{-d+3}.
\]
The worst cases for $e_2\geq4$ are $c_\bfSp=9.42$,  $c_{\bfO^+}=35.64$, and $c_{\bfO^-}=10.94$, all for $d=12$. These yield the  
bounds in Table~\ref{t:asch3} for $\bfSp$ with $q=2$, and for $\bfO^+$ and $\bfO^-$.

This only leaves the cases $(\bfX,e_2) \in \{ (\bfSp, 2), (\bfO^+,2) \}$ (since by Table~\ref{t:asch3} we have $e_2\geq4$ for $\bfX=\bfO^-$);
where for $\bfX=\bfSp$ we have $q\ge 5$ and for $\bfX=\bfO^+$ we have $q\neq 3.$ Here we have
\[
{\rm Prob}_3(g_1,g_2,\bfX) 
\le c_\bfX \cdot q^{-e_1e_2/2} =
c_\bfX\cdot  q^{-(d-2)} = 
c_\bfX\cdot q^{-1} q^{-d+3}.
\]
The worst case for $\bfX=\bfSp$ has $c_\bfSp=3.0$ giving the bound $0.6$ as in Table~\ref{t:proofex} for $q\geq3$.  
For $\bfX=\bfO^+$ with $e_2=2$, the worst cases are $c_{\bfO^+}=25.62$  with $q=2$ giving bound $12.81$, and $c_{\bfO^+}=8.02$  with $q\geq 4$ giving bound $2.005$, as recorded in Remark~\ref{r:dagger}.

\medskip\noindent
\emph{{Case} $\Asch_5$.} \quad Recall $q=p^a$ with $p$ a prime. 
By Proposition~\ref{prop:sub},  ${\rm Prob}_5(g_1,g_2,\bfX)=0$ if $a=1$, or if $\bfX\ne \bfL$ and $a=2$, so we may assume that $q\geq 4$ for type $\bfL$ and  $q\geq8$ if $\bfX\ne \bfL$; and we use the bounds provided by Proposition~\ref{prop:sub}.
  Consider first type $\bfX=\bfL$, where by Remark~\ref{rem:fx}, and using $d\ge 9$ and $q\geq4$,
\begin{align*}
{\rm Prob}_5(g_1,g_2,\bfL)
&\le
4.04\cdot  q^{-e_1e_2+(d-1)/2} \le 
4.04\cdot q^{-2 (d-2) + (d-1)/2}\\
&= 4.04\cdot q^{-(d-1)/2} q^{-d + 3} \le 4.04\cdot 4^{-4} q^{-d+3} <0.016 \cdot q^{-d+3}.
 \end{align*}
For $\bfX\ne \bfL$ we have  $d\geq10$ and $q\geq8$. Further  $e_2\ge 3$ for type $\bfU$ and it follows that
\begin{align*}
{\rm Prob}_5(g_1,g_2,\bfU)& \le
2.11 \cdot  q^{-4e_1e_2/3+2(d-1)/3} \le 
2.11\cdot q^{-4(d-3)+ 2(d-1)/3}\\
 &= 2.11\cdot q^{-(7d- 25)/3} q^{-d + 3} \le  6\cdot 10^{-14} q^{-d+3} .\end{align*}
For the other types let
$c_{\bfSp}=1.82$, 
$c_{\bfO^+}=2.47$, and
$c_{\bfO^-}=5.46$. Then
\[
{\rm Prob}_5(g_1,g_2,\bfX)\le c_{\bfX}\cdot q^{-2e_1e_2/3+d/3} \le 
c_{\bfX}\cdot
 q^{-2\cdot 2 (d-2)/3 + d/3}
 = c_{\bfX}\cdot q^{-1/3} q^{-d + 3} \le c_{\bfX}/2\cdot  q^{-d+3} 
 \]
giving the bounds in Table~\ref{t:asch3}.

\medskip\noindent
\emph{{Case} $\Asch_6$.} \quad
Suppose that ${\rm Prob}_6(g_1,g_2,\bfX) >0$. Then by Proposition~\ref{prop:c6},   
(i) $q$ is an odd prime so $q\ge 3$, (ii) $d=2^n$ with $n\ge 4$,
(iii) $\bfX\neq \bfU, \bfO^-$, and (iv) there are two possibilities for $(e_1,e_2)$ which we consider separately for each type. Suppose first that $\bfX= \bfL$. If $(e_1,e_2)=(d/2,d/2)$ then by Proposition~\ref{prop:c6}, as $d\ge 16$, 
\[
{\rm Prob}_6(g_1,g_2,\bfL) \le \frac{45}{1024}\cdot 
q^{-d^2/4} = \frac{45}{1024}\cdot q^{-d^2/4+d-3} q^{-d+3}\le 2.1\cdot 10^{-26}\cdot q^{-d+3}. 
\] 
If 
$e_2=2$ then  $d\ge 32$ and $q\ge 5$, by Proposition~\ref{prop:c6}(ii), and 
 \[
 {\rm Prob}_6(g_1,g_2,\bfL) \le 
\frac{45}{640} \cdot q^{-\frac{59}{25}d+8} =  \frac{45}{640}\cdot q^{-\frac{34}{25}d+5} q^{-d+3}\le 8.4\cdot 10^{-29}\cdot q^{-d+3}, 
\] 
and we record the larger of these bounds in Table~\ref{t:asch3}.
Now suppose that $\bfX=\bfSp$
or  $\bfO^+$. If $(e_1,e_2)=(d/2,d/2),$ then by Proposition~\ref{prop:c6}, as $q\ge 3$ and $d\geq 16$
\[
{\rm Prob}_6(g_1,g_2,\bfX) \le  \frac{45}{1024}
q^{-\frac{1}{4}d^2+ \frac{23}{10}d} = \frac{45}{1024}\cdot q^{-\frac{1}{4}d^2+\frac{33}{10}d-3} q^{-d+3}\le 7.4\cdot 10^{-9}\cdot q^{-d+3}. 
\] 
Finally consider the case 
$e_2=2$. Here, by Proposition~\ref{prop:c6}, 
either $d=32$, $q\in \{11,13,17\}$, and
\[
{\rm Prob}_6(g_1,g_2,\bfX) \le 
6.9 \cdot   q^{-d+3}, 
\]
or one of $q\geq19$ or $d\geq 128$, and
 \[
 {\rm Prob}_6(g_1,g_2,\bfX) \le 
 5.3\cdot 10^{-3}\cdot q^{-d+3}. 
 \]
The latter bound is recorded in Table~\ref{t:asch3}, and the former special case is recorded in Remark~\ref{r:dagger}.

\medskip\noindent
\emph{{Case} $\Asch_8$.} \quad
Suppose that ${\rm Prob}_8(g_1,g_2,\bfX) >0$. Then by Proposition~\ref{prop:classicalmax},   $\bfX= \bfL$, $d$ is even so $d\geq10$, and 
${\rm Prob}_8(g_1,g_2,\bfL)\leq c_\bfL \cdot  q^{-e_1e_2+d/2}$, where $c_\bfL=1.46$ if $q=2$, and $c_\bfL =3.18$ if  $q\ge 3$. We obtain using Remark~\ref{rem:fx} that 
\[
{\rm Prob}_8(g_1,g_2,\bfL)\le
c_\bfL\cdot q^{-e_1e_2+d/2} \le
c_\bfL\cdot q^{-2(d-2)+d/2}  =
c_\bfL\cdot q^{-d/2 + 1}\cdot q^{-d+ 3}.
\]
Since $d\ge 10$, for $q=2$ we have $c_\bfL\cdot q^{-d/2 + 1}< 0.092$, while for 
$q\ge3$ we have $c_\bfL\cdot q^{-d/2 + 1}<0.04$, as in Table~\ref{t:proofex}.

\medskip\noindent
\emph{{Case} $\Asch_9$.} \quad
Suppose that ${\rm Prob}_9(g_1,g_2,\bfX) >0$. Then by  Proposition~\ref{prop:c9}, $q^u=p$ is a prime, and either    $\bfX=\bfSp$ with $q=2$ or  
$\bfX=\bfO^\varepsilon$, and 
${\rm Prob}_9(g_1,g_2,
\bfX) \le c_{\bfX}\cdot q^{-e_1e_2 + d/2}$, 
where $d\ge 10$, 
$c_{\bfSp} = 6.38$, and 
 $c_{\bfO^\varepsilon} =8.24$
if $q$ is odd or
$c_{\bfO^\varepsilon} =28.1$
if $q=2.$ Using Remark~\ref{rem:fx} we have 
 \[
 {\rm Prob}_9(g_1,g_2,
\bfX) \le c_{\bfX}\cdot q^{-e_1e_2 + d/2} \le
c_{\bfX}\cdot q^{-2(d-2) + d/2}  = 
c_{\bfX}\cdot q^{-(d-2)/2}\cdot q^{-d+3}.
\]
Since $d\ge 10$, for $q=2$ we have $c_{\bfX}\cdot q^{-(d-2)/2}\leq c_{\bfX}\cdot 2^{-4}$, while for 
$q\ge3$ we have $c_{\bfX}\cdot q^{-(d-2)/2}\leq c_{\bfX}\cdot 3^{-4}$, giving the bounds in Table~\ref{t:proofex}.

Since $q^{-(d-2)/2} \le 3^{-4}$ for $q\ge 3$ and  $q^{-(d-2)/2} \le 2^{-4}$ for $q=2$, and upper bounds in Table~\ref{t:proofex} follow on substituting the values of $c_{\bfX}$. Finally the comments in Remark~\ref{r:dagger}(c) follow from Lemma~\ref{lem:c9-1}(b) and Table~\ref{t:c9-1}.
\end{proof}

\subsection{Proof of Theorem~\ref{thm:main}}\label{ss:thm1.2}

Recall that by Equation~\eqref{e:aschnongen}
and Proposition~\ref{prop:main},
   \[
   \rhonongen(g_1,g_2,\bfX)  = \sum_{i=1}^9 {\rm Prob}_i(g_1,g_2,\bfX) \le \lambda_\bfX (q^{-1} + q^{-2}) + \sum_{i} p(i, \bfX)\cdot q^{-d+3}
   \]
where  $\lambda_\bfX=1$ for type $\bfL$ and $\lambda_\bfX=0$ otherwise.
To determine $\kappa_\bfX(q)$, we consider
the different types $\bfX.$
For most types, to determine the `constants' $\kappa_\bfX(q)$ it is sufficient to add up the upper bounds $p(i, \bfX)$ stated in the column labelled 
$\bfX$ in Table~\ref{t:proofex} for the corresponding value of $q$. We treat the exceptional cases in Remark~\ref{r:dagger} separately.
Let
\[
p_{\bfX} \coloneq  \sum_i p(i, \bfX).
\]

Consider first type $\bfL$. If $q=2$ then we may take  $p_{\bfL} = 0.1081$, so we may take $\kappa_\bfX(q) = 0.11$
while if $q\geq 3$ then $p_{\bfL} \le 0.05632$ so we take $\kappa_\bfX(q) = 0.06.$

For type $\bfU$ with $q=2$ 
we find $p_{\bfU} \le 5.3\cdot 10^{-6}$ and we take $\kappa_\bfX(q) = 5.3\cdot 10^{-6}$
while if $q\geq 3$ then we choose
$p_{\bfU} < 1.6 \cdot 10^{-8} = \kappa_\bfX(q).$

For type $\bfSp$ with $q=2$ 
we find $p_{\bfSp} = 1.147< 1.15=\kappa_\bfX(q)$,
while if $q\geq 3$ then 
$p_{\bfSp} = 1.5154<1.52 = \kappa_\bfX(q).$
In the exceptional case where $e_2=2$ with $q\geq3$ we have $p_{\bfSp} = 8.4101<8.42 = \kappa_\bfX(q).$

For type $\bfO^+$ with $q=2$ 
we find $p_{\bfO^+} = 2.074 <2.08=\kappa_\bfX(q)$,
while if $q\geq 3$ then 
$p_{\bfO^+} = 1.53984 < 1.54 = \kappa_\bfX(q).$
In the exceptional case where $e_2=2$, for $q=2$ we have $p_{\bfO^+} = 14.605< 14.61 = \kappa_\bfX(q)$. In the case with $e_2=2$ and $q\geq3$ 
 we have $p_{\bfO^+} = 3.52784<3.53 = \kappa_\bfX(q)$,  while in the special case with $d=32$ and $q\in\{11, 13, 17\}$, we have $p_{\bfO^+} = 10.4273 
 <10.43 = \kappa_\bfX(q)$.

For type $\bfO^-$ with $q=2$ 
we find $p_{\bfO^-} = 1.846 < 1.85=\kappa_\bfX(q)$,
while if $q\geq 3$ then 
$p_{\bfO^-} = 3.018 < 3.02 = \kappa_\bfX(q).$ This completes the proof of the main assertion of Theorem~\ref{thm:main}. In particular, the values of $\kappa_{\bfX}(q)$ in Table~\ref{t:kappa} are valid.

The final assertion is proved using the bounds in Table~\ref{t:kappa}, together with Remark~\ref{r:dagger}(c). First we consider the bounds in columns 2 and 3 of Table~\ref{t:kappa} for the types $\bfX\ne\bfL$. Recall that in this case $d$ is even and so $d\geq 10$. Hence
        \begin{align*}
        \rhogen(g_1,g_2,G) &\geq  \left\{ \begin{array}{lll}
             1- \max\{ 5.3\cdot 10^{-6}\cdot 2^{-d+3}, 1.6\cdot 10^{-8}\cdot 3^{-d+3}\}  &> 0.999             & \mbox{if}\  \bfX=\bfU\\ 
             1- \max\{ 1.15\cdot 2^{-d+3}, 1.52\cdot 3^{-d+3} \}     &> 0.991           & \mbox{if}\  \bfX=\bfSp\\
             1-  \max\{2.08\cdot 2^{-d+3}, 1.54\cdot 3^{-d+3}\}     &> 0.983           & \mbox{if}\  \bfX=\bfO^+\\
             1- \max\{ 1.85\cdot 2^{-d+3}, 3.02\cdot 3^{-d+3} \}     &> 0.985           & \mbox{if}\  \bfX= \bfO^-\\
        \end{array}\right.
    \end{align*}
 and so in these cases  $\rhogen(g_1,g_2,G)>0.983$. Next consider the bounds in  columns 2 and 3  of Table~\ref{t:kappa} for $\bfX=\bfL$ (so here $d\geq9$): for $q=2$ we have  $\rhogen(g_1,g_2,G)\geq 1-q^{-1}-q^{-2}-0.11\cdot q^{-d+3} > 0.248$, while for $q\geq3$ we get  $\rhogen(g_1,g_2,G)\geq 1-q^{-1}-q^{-2}-0.06\cdot q^{-d+3} > 0.555$. Finally we consider the exceptional bounds in column 4 of Table~\ref{t:kappa}. Here $d\geq 10$ and $d$ is even. In each of the three lines where $q\geq3$ we have $\rhogen(g_1,g_2,G)\geqslant 1- 10.43\cdot 3^{-7}>0.995$. Thus we may assume that  $(\bfX,q,e_2)=(\bfO^{+},2,2)$. If $d \geqslant 14$ then by Table~\ref{t:kappa}, $\rhogen(g_1,g_2,G)\geqslant  1 - 14.61\cdot 2^{-11} > 0.99$, so we may assume further that $d = 10$ or $d=12$. In these cases it follows from Remark~\ref{r:dagger}(c) that $p(9,\bfO^+)=0$. Also if $d=10$ then the condition of line $4$ of  Table~\ref{t:Prob3Zero} holds and hence $p(3,\bfO^+)=0$ by Proposition~\ref{prop:extn}. If $d=12$ then by Lemma~\ref{lem:c3boundV2} for the case $\Asch_3$ we can only have $b=2$ and we have the bound ${\rm Prob}_3(g_1,g_2, \bfO^+;b) < 12.585 \cdot 2 \cdot 2^{-10} < 0.0246$.  Thus if $d=10$ then by Table~\ref{t:proofex} we have $\rhogen(g_1,g_2,G)\geqslant 1- 0.035\cdot 2^{-7}>0.99$;  while if $d=12$ then by  Table~\ref{t:proofex} and our observations for the case $\Asch_3$, we have $\rhogen(g_1,g_2,G)\geqslant 1- 0.035\cdot 2^{-9} - 0.0246 >0.975$. This completes the proof.
\qed

\subsection{Proof of Theorem~\ref{thm:appn}}\label{ss:thm1.1}

We prove Theorem~\ref{thm:appn} as a corollary to the following more precise and technical result.

\begin{proposition}\label{prop:appn}
    Suppose that $G$ is a classical group of type $(\bfX, n, q)$  with $n > 8$  as in Table~$\ref{tab:G}$, and that $h_1, h_2$ are  independent randomly selected elements of $G$.  Then with probability at least $c_0/\log n$, for some positive constant $c_0$, some powers $ h_1^{m_1}, h_2^{m_2}$ generate a naturally embedded classical subgroup of type $(\bfY, d, q)$, where $2\alpha\log n_0<d\leq 4\alpha\log n_0$ with $\alpha, n_0$  as in Table~$\ref{tab:XY}$, and  $\bfX, \bfY$ are  as in Table~$\ref{tab:G}$.
\end{proposition}

 \begin{table}
  \caption{Values of $\alpha, n_0$ for Proposition~\ref{prop:appn}.}
\begin{tabular}{cll}
  \toprule
   Types &$\alpha$ & $n_0$\\
  \midrule
  {\bf L} or {\bf U} & $1$ & $n$\\
  {\bf Sp} or $\bfO^\eps$ & $2$ & $(n-\nu)/2$, where $\nu\in\{0,1\}$,\\
 && and $n\equiv \nu\pmod{2}$\\
\bottomrule
\end{tabular}
\label{tab:XY}
\end{table}

\begin{proof}
    Assume the hypotheses of Proposition~\ref{prop:appn} and let $V$ denote the underlying space $\F^n$ on which $G$ acts naturally, with $\F$ as in Table~\ref{tab:G}. We write $n=n_0$ if ${\bf X}= {\bf L}$ or ${\bf U}$, $n=2n_0$ if ${\bf X}= {\bf Sp}$ or ${\bf O}^\pm$, and finally $n= 2n_0+1$ if  ${\bf X}={\bf O}^\circ$ with $q$ odd. Thus $n_0$ is as in Table~\ref{tab:XY}.
    Then, for $i=1$ or $2$,  by \cite[Theorem 3.3(b)]{NP14}, the  probability that some power $g_i\coloneq h_i^{m_i}$ is an $e_i$-ppd stingray element, with $e_i\in (\alpha\log n_0, 2\alpha\log n_0]$, is at least $b/\log n$ for an explicit positive constant $b$ (where $\alpha, n_0$ are as in Table~\ref{tab:XY}).  Assume now that this is the case. Let $d\coloneq e_1+e_2$ and note that $d\in (2\alpha\log n_0, 4\alpha\log n_0]$. For $i=1,2$, let $U_{g_i}$ be the $g_i$-invariant $e_i$-subspace of $V$ on which $g_i$ acts irreducibly and let $F_{g_i}$ be the fixed point space of $g_i$ in $V$. Note also that if $\bfX\ne   \bfL$ then, by Lemma~\ref{lem:redOm} and Table~\ref{tab:stingraycond}, the $e_i$ have the same parity and so $d$ is even.

By \cite[Lemma 2.1]{GNP1} if  $\bfX=  \bfL$, or \cite[Theorem 1.1]{GNP2} (and \cite[Theorem 1.1]{GIM} for $|\F|=2$)  if  $\bfX\ne  \bfL$, there exists a positive constant $b'$ such that, with probability at least $b'$ we have $U_{g_1}\cap U_{g_2} = 0$, so that $V_0\coloneq  U_{g_1}\oplus U_{g_2}$ is a $d$-subspace, and in addition $V_0$ is nondegenerate if $\bfX\ne  \bfL$.  Assume now that $V_0$ is a $d$-space on which $H\coloneq \langle g_1,g_2\rangle$ acts, fixing the complement $F_{g_1}\cap F_{g_2}$, and that $V_0$ is nondegenerate if $\bfX\ne  \bfL$. Note that $H$ is a subgroup of a naturally embedded classical subgroup of $G$, acting on $V_0$ and fixing $F_{g_1}\cap F_{g_2}$ pointwise. 
Then, by Theorem~\ref{thm:main}, with probability at least a positive constant $b''$, $H$ is a classical subgroup of type $(\bfY, d, q)$ on $V_0$, where $\bfX, \bfY$ are as in Table~\ref{tab:G}.   

Putting these steps together we see that all the conditions of Proposition~\ref{prop:appn} hold for $h_1, h_2$ with probability at least $c_0/\log n$ where $c_0=bb'b''$. This completes the proof.
\end{proof}
 
 Now we prove Theorem~\ref{thm:appn}. We take the constant $k(\eta)$ in Theorem~\ref{thm:appn} to be $\log(\eta^{-1})/c_0$, with $c_0$ as in Proposition~\ref{prop:appn}.

\medskip\noindent
\textbf{Proof of Theorem~\ref{thm:appn}.}\quad  We use the notation from Proposition~\ref{prop:appn} and its proof above.
First we note that, if $\bfX=  \bfL$ or $ \bfU$, then the dimension $d$ in Proposition~\ref{prop:appn} satisfies $2\log n< d\leq 4\log n$. On the other hand if  $\bfX=  \bfSp$ or $ \bfO^\varepsilon$ with $\eps\in\{+,-,\circ\}$, then  $n=2n_0$ or $2n_0+1$ so $n_0=n/2$ or $(n-1)/2$ respectively (see Table~\ref{tab:XY}), and since $n>8$ we have 
\[
2\log n < 4\log( (n-1)/2) < d\leq 8\log(n/2) < 8\log n.
\]
Thus for all types $2\log n < d \leq c(\bfX)\log n$ with $c(\bfX)$ as in Table~\ref{tab:G}. 

Now suppose that $\eta > 0$, and that $h_1, h_2, \dots, h_{k+1}$ are independent uniformly distributed random elements from   $G$ for some $k>1$. Then by Proposition~\ref{prop:appn}, for each  $i$, the probability $\pi$ that some powers $h_i^{m_1}, h_{i+1}^{m_2}$ generate a naturally embedded $d$-dimensional classical subgroup of $G$ with $d$ as above, satisfies $\pi \geq c_0/\log n$ with $c_0$ the positive constant in Proposition~\ref{prop:appn}. Thus  the probability that at least one of these $k$ pairs generates such a subgroup is $1-(1-\pi)^k$, and this probability is at least $1-\eta$ if and only if $\tau\coloneq  (1-\pi)^k\leq \eta$. We claim that this holds for all $k\geq k(\eta)\log n$, where $k(\eta)= \log(\eta^{-1})/c_0$ (a positive constant).

Note that $0<c_0/\log n<\pi<1$, and hence $x\coloneq -c_0/\log n$ satisfies $-1<x<0$. Therefore (see, for example, \cite[(2)]{NieP}), $\log(1+x) < x<0$. Thus our assumption $k\geq k(\eta)\log n$ implies, since each of $x, \log(1+x)$, and $\log(\eta)$ is negative, 
\[
k\geq k(\eta)\log n = \frac{\log(\eta^{-1})}{c_0}\log n = \frac{\log(\eta^{-1})}{-x} = \frac{\log(\eta)}{x} > \frac{\log(\eta)}{\log(1+x)}.
\]
Therefore $\log(\eta) > k\log(1+x) = \log((1+x)^k)$, and hence $\eta > (1+x)^k = (1-c_0/\log n)^k \geq (1-\pi)^k = \tau$. Thus the claim is proved, and hence also Theorem~\ref{thm:appn} is proved.\qed

\section{Acknowledgements}
All three  authors thank Frank L\"ubeck for insightful discussions that helped shape the direction of this work. We thank
Tim Burness for  
his advice on various aspects of the work.
The research for this paper forms a major part of the Australian Research Council Discovery Project  Grant DP190100450 of the second and third authors which has supported the first author.   
The second author acknowledges funding by the Deutsche Forschungsgemeinschaft (DFG, German Research Foundation) - Project 286237555 - TRR 195.

\end{document}